\documentclass{article}
\usepackage[utf8x]{inputenc}
\usepackage[capposition=top]{floatrow}
\usepackage{amsmath,amsfonts,amssymb,amscd,amsthm,xspace}
\usepackage[shortlabels]{enumitem}
\usepackage{tikz-cd}
\usepackage{imakeidx}
\usepackage{indentfirst}
\usepackage{phoenician}
\theoremstyle{plain}
\newtheorem{theorem}{Theorem}[section]

\newtheorem{coro}[theorem]{Corollary}
\newtheorem{lemma}[theorem]{Lemma}
\newtheorem{prop}[theorem]{Proposition}

\newtheorem*{main}{Main Theorem}

\theoremstyle{definition}
\newtheorem{defi}[theorem]{Definition}
\newtheorem{question}[theorem]{Question}

\theoremstyle{remark}
\newtheorem{remark}[theorem]{Remark}
\newtheorem*{notation}{Notation}

\usepackage{float}
\usepackage{cite}

\usepackage{hyperref}
\hypersetup{colorlinks,allcolors=black,breaklinks}

\usepackage[a4paper,width=150mm,top=25mm,bottom=25mm]{geometry}

\title{Elementary totally disconnected, locally compact groups of higher complexity}
\author{João V. P. e Silva}
\date{}

\begin{document}

\maketitle

\begin{abstract}
\noindent The article focuses on a class of second countable groups assembled from profinite and discrete by elementary operations. We focus on a rank associated with these groups that measure their complexity, the decomposition rank. A collection of groups acting on $\aleph_0$-regular trees is defined and used for the first construction of a group with decomposition rank $\omega^\omega+1$.
\end{abstract}

\section{Introduction}
The class of elementary totally disconnected,\ locally compact,\ second countable (t.d.l.c.s.c.) groups,\ denoted as $\mathcal{E}$,\ was first defined in \cite{elementary_first} by Phillip Wesolek. It is the smallest class of second countable groups containing all discrete and profinite groups,\ closed under the operations of increasing unions and extensions. These groups often arise in the general theory of t.d.l.c.s.c.~groups and also play an essential role in the structure of these,\ one example being the following theorem:
\begin{theorem}\cite[Theorem 1.5]{elementary_first}
Let $G$ be a t.d.l.c.s.c.~group. Then
\begin{enumerate}
    \item There exists a unique maximal closed normal subgroup $\textup{Rad}_{\mathcal{E}}(G)$ of $G$ such that $\textup{Rad}_{\mathcal{E}}(G)$ is elementary.
    \item There exists a unique minimal closed normal subgroup $\textup{Res}_{\mathcal{E}}(G)$ of $G$ such that $G/\textup{Res}_{\mathcal{E}}(G)$ is elementary.
\end{enumerate}
\end{theorem}
One can associate to each elementary group a countable ordinal (possibly infinite),\ which we call the decomposition rank. Such rank is denoted as $\xi$. This rank measures the complexity of such groups,\ as it roughly tells how many steps are necessary to build the elementary group from the discrete and profinite ones. The rank also allows for induction proofs in this class. Also,\ alternatively,\ one can define elementary groups in relation to the rank: a group is elementary if,\ and only if,\ the decomposition rank is well-defined for it \cite[Theorem 4.7]{elementary_first}.\\
\indent The decomposition rank allows us to prove that any t.d.l.c.s.c.~group with a topologically simple,\ compactly generated,\ non-discrete subgroup or quotient is not elementary,\ as the rank is not well-defined for them. The Lie groups over local fields of the form $\textup{PSL}_n(\mathbb{Q}_p)$ and the Neretin groups are examples of simple compactly generated t.d.l.c.s.c.~groups,\ hence they are not elementary. Automorphisms groups of $q$-regular trees have a non-trivial,\ compactly generated simple subgroup,\ hence are also not elementary. Two questions that arise related to elementary groups are then the following:

\begin{question}\label{question:1}
Is the existence of a topologically simple,\ non-discrete,\ compactly generated subquotient necessary for a group $G$ to be non-elementary?
\end{question}

\begin{question}\label{question:2}
What is the least upper bound for the decomposition rank among the elementary groups? Is it a countable ordinal,\ or is it $\omega_1$?
\end{question}

Here we focus mostly on Question \ref{question:2}. If such a countable upper bound exists,\ there would be strong restrictions on the class of elementary groups. If there exist no upper bound,\ then there exist no SQ-universal elementary group,\ as any surjectively universal group to the elementary groups could not be second countable. No upper bound would,\ in particular,\ imply that there exist no surjectively universal t.d.l.c.s.c.~group for the class of t.d.l.c.s.c.~groups,\ answering a question of S. Gao and M. Xuan \cite{GAO2014343}.\\
\indent The ``descriptive complexity'' of whether or not a given t.d.l.c.s.c.~group is elementary depends on whether or not the rank is unbounded. Polish groups have a canonical $\sigma$-algebra on the set of closed subgroups,\ the Effros-Borel $\sigma$-algebra. Similarly to how it is proven on \cite{elemtr_amenable_descriptive} for elementary amenable groups,\ one can see that the set $\mathcal{E}$ of elementary t.d.l.c.s.c.~subgroups of Sym$(\mathbb{N})$ belong to its Effros-Borel $\sigma$-algebra if,\ and only if,\ the decomposition rank is unbounded below $\omega_1$. \\
\indent In \cite{dynamics_colin},\ a class of subgroups of Sym$(\mathbb{N})$ is defined. Such a class contains all the elementary t.d.l.c.s.c.~groups and belongs to the Effros-Borel $\sigma$-algebra. However,\ by a construction of Caprace-Wesolek \cite{counter_example_class_colin},\ this class contains non-elementary t.d.l.c.s.c.~groups. Hence the question of whether or not the class $\mathcal{E}$ belongs to the Effros-Borel $\sigma$-algebra of Sym$(\mathbb{N})$ remains open.\\
\indent Until now,\ not much is known about the upper bound. Before this article,\ only groups with rank up to $\omega+2$ were built (Proposition 9.7 \cite{elementary_groups_second}) via the construction of the groups $E_X(G,\ U)$,\ a family of groups acting on the $\aleph_0$-regular tree and fixing an end of the tree. In Section 3,\ we focus on proving results about the groups $E_X(G,\ U)$ defined in \cite{elementary_groups_second}. Such construction will be central for the proof of our main result:
\begin{main}\label{main}
Given $n\in\mathbb{N}$,\ there exist a compactly generated elementary group $G$ such that $\xi(G)=\omega^n+2$. 
\end{main}
To prove the theorem,\ we will show that,\ under a certain condition,\ restricted wreath products allow us to build a group with a higher rank than the original groups. The groups $E_X(G,\ U)$,\ built as in \cite{elementary_groups_second},\ contain a countable iteration of wreath products. Under some conditions,\ we can iterate such construction to build groups with the desired ranks. By some limitations on the conditions of the theorem,\ we are only able to iterate for groups with rank $\omega^n+2$ for $n\in\mathbb{N}$. Colin Reid,\ Simon Smith,\ and Phillip Wesolek first conjectured Main Theorem in \cite{UnpublishedWesolek}.\\
\indent It is well known that if there exist a group $G$ with decomposition rank $\xi(G)=\alpha+1$,\ then for every $\beta<\alpha$ there exist $H$ an elementary subgroup of $G$ with $\xi(H)=\beta+1$. On the other hand,\ a general way to build such a subgroup $H$ is still unknown. Section 4 introduces new concepts for t.d.l.c.s.c.~groups,\ the residual order,\ height,\ and rank. We can then use these concepts in Section 5 to prove that taking wreath products of groups as generated in Main Theorem gives us a new group whose rank is equal to the sum of the ranks of all building blocks. We will then be able to give an explicit construction of elementary groups with any given rank up to $\omega^\omega+1$. Such a work came from a suggestion given by George Willis. Initially,\ the idea was to work with an alternative rank for elementary groups. Later we noticed that such a rank could make building groups with an explicitly given rank easier.\\
\indent As in finite group theory,\ a notion of chief factors and chief series can be defined for the class of t.d.l.c.s.c.~groups. This notion in central to understanding the class of t.d.l.c.s.c.~groups. By \cite[Corollary 1.4]{chief_series},\ every compactly generated t.d.l.c. group admits an essentially chief series. Some central results about chief factors from \cite{elementary_groups_second} concern blocks that are either non-elementary or elementary of high enough rank,\ which also raises the question:
\begin{question}
What kind of chief factor with "big" decomposition rank can occur?
\end{question}
Under some conditions,\ the groups defined at \cite{elementary_groups_second} will have a minimal closed normal subgroup. This subgroup is the first explicit example of a chief factor of a specific type. Our construction allows us to build such a chief factor so that its decomposition rank is $\omega^n+1$,\ for $n$ a natural number. Proof of the existence of such chief factors is given in Section 5.

\section{Background}
This section presents some central definitions and results to understand the article. Subsections 2.1 and 2.2 focus on the general theory of t.d.l.c.s.c.~groups and elementary groups,\ giving what is essential to prove results about such groups. Subsection 2.3 focus on chief factors,\ giving what is necessary to discuss a special type of chief factor in Subsection 5.4. Subsection 2.4 gives the necessary concepts and results to define residual order,\ height,\ and rank in Section 4.

\subsection{Totally disconnected,\ locally compact groups}

Our first result is a central tool to build the topology of the groups acting on the $\aleph_0$-regular tree in such a way that the topology stays t.d.l.c.s.c.
\begin{prop}\label{prop:building groups}\cite[Chapter III,\ Proposition 1]{bourbaki2013general}
Let $G$ be a group and $\mathcal{F}$ be a family of subsets of $G$ satisfying the following:
\begin{enumerate}
    \item Given any $U\in\mathcal{F}$,\ there exist $V\in\mathcal{F}$ such that $V.V\subset U$.
    \item For all $g\in G$ and all $V\in\mathcal{F}$,\ we have $gVg^{-1}\in\mathcal{F}$.
\end{enumerate}
There exists a unique topology on $G$ such that $G$ is a topological group and $\mathcal{F}$ is a neighborhood basis of the identity element.
\end{prop}

\begin{theorem}[Van Dantzig]\index{van Dantzig Theorem|ndx}
Let $G$ be a totally disconnected,\ locally compact group,\ and suppose $U$ is an open neighborhood of $1$. Then there exist a compact open subgroup $V\subset U$.
\end{theorem}
Knowing a neighborhood of the identity of a topological group allows us to understand the whole topology of the group,\ as multiplication by an element is a homeomorphism. The Van Dantzig theorem tells us that there exist a neighborhood of the identity of open,\ compact subgroups. We will denote $\mathcal{U}(G)$ the set of all compact open subgroups of the group $G$.

\begin{defi}\index{compactly generated|ndx}
A topological group $G$ is \textbf{compactly generated} if there exist $K$ a compact subset of $G$ such that $G=\langle K\rangle$.
\end{defi}

Given $G$ a topological group,\ we denote $\mathcal{K}(G)=\{H;\ H$ is a closed,\ compactly generated subgroup of $ G\}$.\index{$\mathcal{K}(G)$|ndx}
\begin{remark}\label{remark:comp gen}
Given a t.d.l.c.s.c.~group $G$,\ a straightforward application of the Van Dantzig Theorem allows us to see $G$ as a countable union of compactly generated open subgroups. To see this,\ let $\{g_i\}_{i\in\mathbb{N}}\subset G$ countable dense subset. This subset exists because $G$ is second countable. Given $U\in\mathcal{U}(G)$,\ we can define an increasing sequence of compactly generated open subgroups $O_n:=\langle U,\ g_0,\ g_1,\ g_2,\ \ldots,\ g_{n}\rangle$. These are all open and compactly generated by construction. It then follows that $G=\bigcup_{i\in\mathbb{N}} O_n$. Hence every t.d.l.c.s.c.~group is a countable union of open,\ compactly generated subgroups.
\end{remark} 
\indent Tychonoff's theorem tells us that product of compact groups is always a compact group. However,\ this is different for locally compact groups. To go through this problem,\ we define the local direct product. 

\begin{defi}\index{local direct product|ndx}\label{defi:local direct product}
Let $X$ a countable set and $\{G_x\}_{x\in X}$ be a collection of t.d.l.c.s.c.~groups. For each $x\in X$,\ let $U_x\leqslant G_x$ a fixed compact,\ open subgroup. Let $\mathcal{X}=\{F\subset X| \ F $ is a finite subset$\}$. Define for each $F\in\mathcal{X}$ the group $S_F:=\prod_{x\in F}G_x\times \prod_{x\in X\backslash F}U_x$. Given $F'\subset F$,\ define $i_{F',\ F}:S_{F'}\rightarrow S_F$ the embedding that maps the $x$-th coordinate of $S_{F'}$ to the $x$-th coordinate of $S_F$. We define the \textbf{local direct product} of the $\{G_x\}_{x\in X}$ with respect to $\{U_x\}_{x\in X}$ as $\varinjlim_{F\in\mathcal{X}} S_F$ with the inductive limit topology. We denote it by $\bigoplus_{x\in X}(G_x,U_x)$.
\end{defi}

As the group $\prod_{x\in X}U_x$ is a compact,\ open subgroup of $S_F$,\ for all $F\in\mathcal{X}$,\ by definition of the final topology,\ it is also a compact,\ open in $\bigoplus_X(G,\ U)$. The construction above can be generalized for any given set $X$. As we only work with second countable groups in this article,\ we restrict the definition for the case $X$ is countable.\\
\indent If there exist $G$ a t.d.l.c.s.c.~group and $U\in\mathcal{U}(G)$ such that $G_x=G$ and $U_x=U$ for all $x\in X$,\ the local direct product will be denoted as $\bigoplus_X(G,\ U)$.\\
\indent Given $\phi\in $ Sym$(X)$,\ the action of $\phi$ on $X$ induces an automorphism of $\bigoplus_{X}(G,\ U)$ by shifting the coordinates. Hence,\ given a t.d.l.c.s.c.~group $K$ acting on a countable set $X$ and $k$ an element of $K$,\ one can define $\alpha_k$ an automorphism of $\bigoplus_X(G,\ U)$ given by moving the $x$ coordinate to the $k.x$ coordinate. Under such an action,\ one can define the wreath product of t.d.l.c.s.c.~groups in a similar way to the wreath product of discrete groups. For that,\ first,\ we need to define our topological permutation groups.

\begin{defi}
Let $G$ be a t.d.l.c.s.c.~group and $X$ a countable set. Let $G$ act (left) on $X$. We denote:
\begin{itemize}
    \item Given $x\in X$ we define the \textbf{fixator of the element $x$} as $G_{(x)}:=$Fix$_G(x)=\{g\in G;\ g.x=x\}\leqslant G$.\index{fixator|ndx}\index{$\textup{Fix}_G(x)$|ndx}
    \item Given $F\subset X$ a subset we define the \textbf{fixator of the subset $F$} as $G_{(F)}:=$Fix$_G(F)=\{g\in G;\ g.x=x $ for all $x\in F\}\leqslant G$ .
    \item For $g\in G$ we define the \textbf{support of $g$} as $\text{supp}(g):=\{x\in X;\ g.x\neq x\}\subset X$.\index{support of $g$|ndx}\index{$\textup{supp}(g)$|ndx}
    \item For $H\leqslant G$ a subgroup we define the \textbf{normal closure of $H$ in $G$} as $\langle\!\langle H\rangle\!\rangle_G=\bigcap_{H\leqslant N\trianglelefteq G}N$. \index{normal closure|ndx}\index{$\langle\!\langle N\rangle\!\rangle$|ndx}
\end{itemize}
\end{defi}

\begin{defi}\label{defi:permutation group}\index{permutation group|ndx}
Let $G$ be a t.d.l.c.s.c.~group and $X$ a countable set with the discrete topology. Let $G$ act (left) on $X$. The pair $(G,\ X)$ is called a \textbf{(left) t.d.l.c.s.c.~permutation group} if the action is faithful and has compact open point fixators.
\end{defi}
\indent  Notice that a permutation group $(G,\ X)$ can be naturally embedded on Sym$(X)$ under such action.
\begin{defi}\label{defi:wr product}\index{restricted wreath product|ndx}\index{wreath product|ndx}
The \textbf{restricted wreath product|ndx} of the t.d.l.c.s.c.~group $L$ with $(K,\ X)$ over $U\in\mathcal{U}(L)$ is defined to be
$$L\wr_U (K,\ X):=\bigoplus_X (L,\ U)\rtimes K$$
with the product topology on the pair $\bigoplus_{X}(L,\ U)$ and K.
\end{defi}

Let $(K,\ X)$,\ $(L,\ Y)$ be permutation t.d.l.c.s.c.~groups and $U\in \mathcal{U}(L)$. we define the \textbf{imprimitive action} of $L\wr_{U}(K,\ X)$ on $Y\times X$ as follows: given $((l_{x})_{x\in X},\ k)$ an element of $L\wr_{U}(K,\ X)$ and $(y,\ x)$ in the set $ Y\times X$, we have $((l_{x})_{x\in X},\ k).(y,\ x)=(l_{kx}.y,\ k.x)$.\index{imprimitive action|ndx}\\
\indent Given $(L,\ X)$ a permutation group and $U\in\mathcal{U}(L)$,\ one can iterate the wreath product construction to build topological permutation groups acting on $X^n$,\ for $n\in\mathbb{N}$. For such construction we define $(L_1,\ X):=(L,\ X)$. Given that $(L_n,\ X^n)$ was defined,\ define $(L_{n+1},\ X^{n+1}):=(L\wr_U (L_n,\ X^n),\ X\times X^n)$,\ where the action of $L_{n+1}$ on $X^{n+1}=X\times X^n$ is the imprimitive action given by the wreath product.\\
\indent Other types of topological groups that will be important for this article are the following:
\begin{defi}\label{defi:general defi}
Let $G$ be a topological group. Then:
\begin{itemize}
    \item (Topologically perfect) The group $G$ is \textbf{topologically perfect} if $G=\overline{[G,G]}$.\index{topologically perfect|ndx}
    \item The group $G$ is \textbf{topologically simple} if the only closed normal subgroups of $G$ are $G$ itself and $\{1\}$.\index{topologically simple|ndx}
    \item The group $G$ is \textbf{monolithic} if the intersection of all non-trivial closed normal subgroups is non-trivial. We will call this intersection the monolith of $G$.\index{monolithic group|ndx}\index{monolith|ndx}
    \item The group $\textup{Aut}(G)$ is the group of all \textbf{continuous automorphisms} of $G$.\index{$\textup{Aut}(G)$|ndx}
\end{itemize}
\end{defi}

\subsection{Elementary groups}

Formally we define the class of elementary group as follows:

\begin{defi}\cite[Definition 1.1]{elementary_first}\label{defi:elementary groups}\index{elementary groups|ndx}
The class of \textbf{elementary groups} is the smallest class $\mathcal{E}$ of t.d.l.c.s.c.~groups such that:\index{$\mathcal{E}$|ndx}
\begin{enumerate}
    \item[(i)] $\mathcal{E}$ contains all second countable profinite groups and countable discrete groups;
    \item[(ii)] $\mathcal{E}$ is closed under taking group extensions,\ that is,\ if there exist $N\trianglelefteq H$ closed subgroup such that $N\in\mathcal{E}$ and $H/N\in\mathcal{E}$ then $H\in\mathcal{E}$;
    \item[(iii)] $\mathcal{E}$ is closed under taking closed subgroups;
    \item[(iv)] $\mathcal{E}$ is closed under taking quotients by closed normal subgroups;
    \item[(v)] If $G$ is a t.d.l.c.s.c.~group and $\bigcup_{i\in\mathbb{N}}O_i=G$,\ where $\{O_i\}_{i\in\mathbb{N}}$ is an $\subset$-increasing sequence of open subgroups of $G$ with $O_i\in\mathcal{E}$ for each $i$,\ then $G\in\mathcal{E}$. We say that $\mathcal{E}$ is closed under countable increasing union.
\end{enumerate}
\end{defi}

As the class of elementary groups is closed under countable unions, Remark \ref{remark:comp gen} shows that compactly generated groups will be central in our work.\\
\indent Equivalently,\ the class of elementary groups is defined only by (i),\ (ii),\ and (v). The statements (iii) and (iv) may be deduced from the others \cite[Proposition 3.4,\ Theorem 3.8]{elementary_first}.\\
\indent Following are some definitions that will be central to defining the decomposition rank:

\begin{defi}[$\sup^+$]\index{$\sup^+$|ndx}
Let $\{\alpha_i\}_{i\in\mathbb{N}}$ an increasing collection of ordinals,\ that is,\ $\alpha_i\leqslant \alpha_{i+1}$ for every $i\in \mathbb{N}$. We define $\sup_{i\in\mathbb{N}}^+\alpha_i=\sup_{i\in\mathbb{N}}\alpha_i$ if $\sup_{i\in\mathbb{N}}\alpha_i$ is a successor ordinal and $\sup_{i\in\mathbb{N}}^+\alpha_i=\sup_{i\in\mathbb{N}}\alpha_i+1$ otherwise.
\end{defi}

The $\sup^+$ is defined so we can avoid limit ordinals when taking supremum. The decomposition rank is never a limit ordinal and this notion of supremum is used in the definition.

\begin{defi}\cite[Subsection 2.3]{elementary_first}\label{defi:discrete residual}
The \textbf{discrete residual} of a t.d.l.c.s.c.~group $G$ is defined as:
$$\textup{Res}(G):=\bigcap \left\{ O; O\text{ is an open normal subgroup of }G \right\}.$$\index{discrete residual|ndx}\index{$\textup{Res}(G)$|ndx}
\end{defi}

\begin{defi}\cite[Lemma 4.12]{elementary_first}\label{defi:decomposition rank}
Let $G$ be an elementary group. We define the \textbf{decomposition rank} $\xi:\mathcal{E}\longrightarrow [0,\omega_1)$ as follows:\index{decomposition rank|ndx}\index{$\xi$|ndx}
\begin{enumerate}
    \item[(i)] $\xi(\{1\})=1$;
    \item[(ii)] If $G$ is a non-trivial compactly generated t.d.l.c.s.c.~group then $\xi(G)=\xi(\textup{Res}(G))+1$;
    \item[(iii)] For $G$ a t.d.l.c.s.c.~group and $\{O_i\}_{i\in\mathbb{N}}$ a $\subset$-increasing sequence of subgroups of $G$ with $O_i$ open and compactly generated for each $i$,\ then $\xi(G)=\sup_{i\in\mathbb{N}}\{\xi(\textup{Res}(O_i))\}+1=\sup^+_{i\in\mathbb{N}}\{\xi(O_i)\}$. (This sequence will always exist by Remark \ref{remark:comp gen}).
\end{enumerate}
\end{defi}

As defined above,\ one can see that the rank tells how far an elementary group is from being residually discrete. On \cite[Proposition 4.19]{elementary_first},\ Wesolek shows that the decomposition rank as defined above tells,\ up to a finite ordinal,\ how many steps it is necessary to build the elementary group from discrete and profinite ones under increasing countable unions and extensions.\\
\indent Given $\alpha$ an ordinal,\ there exist $\beta_n>\beta_{n-1}>\ldots>\beta_0\geqslant 0$ ordinals and $a_n,a_{n-1},\ldots,a_0$ non-zero natural numbers such that
$$\alpha=\omega^{\beta_n}a_n+\omega^{\beta_{n-1}}a_{n-1}+\ldots +\omega^{\beta_1}a_1+\omega^{\beta_0}a_0.$$
This is known as the Cantor normal form of $\alpha$. It helps us perform operations like sums and multiplications of ordinals. The ordinal $\beta_n$ is called the \textbf{leading exponent} of $\alpha$. 

Below are some properties of the decomposition rank that will be useful for us during the proofs. 

\begin{remark}\label{remark:decomposition rank}
Notice that,\ by definition,\ the decomposition rank of an elementary group is always a successor ordinal. Hence if $G$ is a non-trivial compactly generated group,\ then it is also the case that $\xi(G)=\alpha+2$,\ for some $\alpha$ ordinal.
\end{remark}

\indent This remark will be used multiple times throughout the article and will be assumed to be known without further comment. The following are the main results relating to the rank of an elementary group and its structure.

\begin{prop}\cite[Corollary 4.10]{elementary_first}\label{prop:decomposition rank closed subgroup} 
Let $H\leqslant G$ elementary groups,\ where $H$ is closed in $G$. Then $\xi(H)\leqslant\xi(G)$.
\end{prop}

This proposition,\ in particular,\ implies that the decomposition rank (Definition \ref{defi:decomposition rank}) doesn't depend on the choice of $\{O_i\}_{i\in\mathbb{N}}$.

\begin{theorem}\cite[Theorem 4.19]{elementary_first}\label{prop:quotient elementary 101}
Let $N\trianglelefteq G$ be elementary groups and $N$ be closed in $G$. Then $G/N$ is elementary and $\xi(G/N)\leqslant\xi(G)$.
\end{theorem}

\begin{lemma}\cite[Lemma 3.7]{elementary_groups_second}\label{prop:product same rank}
Let $G_1,\cdots,\ G_n$ be a collection of elementary groups. Then  $$\xi(G_1\times G_2\times\ldots\times G_n)=\sup_{1\leqslant  i\leqslant  n}\xi(G_i).$$ 
\end{lemma}

\begin{lemma}\cite[Lemma 3.6]{elementary_groups_second}\label{lemma:short exact sequence decomposition rank}
Suppose
$$\{1\}\rightarrow N\rightarrow G \rightarrow G/N \rightarrow \{1\}$$
is a short exact sequence and $N$,\ $G/N$ are elementary groups. Then $G$ is an elementary group and we have that $\xi(G)\leqslant\xi(N)-1+\xi(G/N)$.
\end{lemma}

\begin{lemma}\cite[Lemma 3.10]{elementary_groups_second}\label{lemma:cocompact has same rank as group}
Suppose that $G$ is a t.d.l.c.s.c.~group and $N$ is a non-trivial closed cocompact normal subgroup of $G$. If $N$ is elementary,\ then $G$ is elementary with $\xi(G) = \xi(N)$.
\end{lemma}

The next two results follow directly from the definition of elementary groups,\ decomposition rank,\ and the results stated above:

\begin{lemma}\label{lemma:infinite restricted prod rank}\label{lemma: rank reduced prod}
Let $X$ be a countable set. If $\{G_x\}_{x\in X}$ is a collection of elementary groups,\ $U_x\in\mathcal{U}(G_x)$,\ for $x\in X$,\ then the group $\bigoplus_X(G_x,U_x)$ is also elementary and $\xi(\bigoplus_X(G_x,U_x))=\sup^+_{x\in X}\xi(G_x)$.
\end{lemma}

\begin{coro}\label{coro:wr product is elementary}
Let $L$ an elementary group,\ $(K,\ X)$ an elementary permutation group,\ and $U\in\mathcal{U}(L)$. Then the wreath product (Definition \ref{defi:wr product}) $L\wr_U (K,\ X)=\bigoplus_{x\in X}(L,\ U)\rtimes K$ is also an elementary group.
\end{coro}



\subsection{Chief factors}

\begin{defi}
A \textbf{normal factor} of a topological group $G$ is a quotient $K/L$ such that $K$ and $L$ are distinct closed normal subgroups of $G$ with $L < K$. We say that $K/L$ is a \textbf{topological chief factor} of $G$ if there exists no closed normal subgroup $M$ of $G$ such that $L < M < K$.\\
\indent We define the centralizer of a normal factor $K/L$ in $G$ as the set:
$$C_G(K/L):= \{g \in G;\ \forall k\in K,\ [g,\ k] \in L\}.$$
\end{defi}

For simplicity,\ topological chief factors are called chief factors in this paper. 

\begin{defi}
An \textbf{essentially chief series} for a topological group $G$ is
a finite series
$$\{1\} = G_0 \leqslant G_1 \leqslant \ldots \leqslant G_n = G$$
of closed normal subgroups such that each normal factor $G_{i+1}/G_i$ is either
compact,\ discrete,\ or a chief factor of $G$.
\end{defi}

We say a chief factor $K/L$ is \textbf{non-abelian} if the group $K/L$ is not an abelian group. We say a chief factor $K/L$ is an \textbf{elementary chief factor} if the group $K/L$ is an elementary group. As the groups studied in this article are elementary,\ all their chief factors are elementary.

\begin{defi}{\cite[Definition 1.10]{chief_series}}
Given a topological group $G$,\ we say that the closed normal
factors $K_1/L_1$ and $K_2/L_2$ are \textbf{associated} if the following equations hold:
$$\overline{K_1 L_2} = \overline{K_2 L_1};\ K_1 \cap \overline{L_1 L_2} = L_1;\ K_2 \cap \overline{L_1 L_2}= L_2.$$
\end{defi}

Non-abelian chief factors $K_1/L_1$ and $K_2/L_2$ are associated if,\ and only if,\ $C_G(K_1/L_1)=C_G(K_2/L_2)$ \cite[Proposition 6.8]{reid_wesolek_le}. In particular,\ association of non-abelian chief factors is an equivalence relation. However association of abelian chief factors is not an equivalence relation in general \cite[Proposition 10.1]{reid_wesolek_le}.

\begin{defi}
 A \textbf{chief block} is an association class of non-abelian chief factors. The set of chief blocks of a group $G$ is denoted by $\mathfrak{B}_G$,\ and we denote the equivalent classes as $\mathfrak{a},\ \mathfrak{b}$.
\end{defi}

\begin{defi}{\cite[Definition 7.1]{chief_series}}
Let $\mathfrak{a}$ be a chief block of a group $G$. We define the \textbf{centralizer of $\mathfrak{a}$} as:
$$C_G(\mathfrak{a}):=\{g\in G;\ [g,L]\subset K\}$$
for some (equivalently,\ any) representative $K/L$ of $\mathfrak{a}$ \cite[Proposition 6.8]{reid_wesolek_le}. This set is a closed normal subgroup of $G$.
\end{defi}

Given $G$ a group,\ one can order the chief blocks of $G$ in relation to their centralizers. We say that $\mathfrak{a}\leqslant \mathfrak{b}$ if $C_G(\mathfrak{a})\leqslant C_G(\mathfrak{b})$.

\begin{defi}{\cite[Definition 1.15]{reid_wesolek_le}}
Let $\mathfrak{a}$ be a chief block and $N$ be a normal subgroup of $G$. We say $N$ \textbf{covers} $\mathfrak{a}$ if there exists $B\leqslant A\leqslant N$ such that $A/B\in \mathfrak{a}$.\\
\indent A block $\mathfrak{a}$ is said to be \textbf{minimally covered} if there exist a unique smallest closed normal subgroup $N$ covering the block. We denote as $\mathfrak{B}_G^{min}$ the collection of minimally covered chief blocks.
\end{defi}

Notice that given a group $G$ and $K/N$ a chief factor of $G$,\ it is only sometimes the case that $K/N$ is a simple group. Hence,\ $K/N$ usually has chief factors. We describe two types of elementary chief factors via their minimally covered chief blocks and one type by their normal subgroups.\\

\begin{defi}{\cite[Definition 8.19]{chief_series}}\label{defi:chief fac types}
Let $G$ be a t.d.l.c.s.c.~group and $K/N$ a chief factor of $G$.
\begin{enumerate}
    \item The chief factor $K/N$ is of \textbf{weak type} if $\mathfrak{B}_{K/N}^{min}=\emptyset$.
    \item The chief factor $K/N$ is of \textbf{stacking type} if $\mathfrak{B}_{K/N}^{min}\neq \emptyset$ and for all $\mathfrak{a},\ \ \mathfrak{b}\in \mathfrak{B}_{K/N}^{min}$,\ there exist $\psi\in \textup{Aut}(G)$ such that $\psi.\mathfrak{a} \subset \mathfrak{b}$.
    \end{enumerate}
\end{defi}

\begin{defi}{\cite[Definition 6.11,\ Theorem 6.12]{chief_series}}
Let $G$ be a t.d.l.c.s.c.~group and $K/N$ a chief factor of $G$. The chief factor $K/N$ is of \textbf{semisimple type} if there exist a collection $\{N_i\}_{i\in I}$ of normal subgroups of $K/N$ such that
\begin{itemize}
    \item For every $i\in I$,\ $N_i$ is non-abelian and topologically simple group.
    \item For every $i\neq j\in I$,\ $[N_i,\ N_j]=\{1\}$.
    \item $K/N=\overline{\langle N_i|\ i\in I \rangle}$.
\end{itemize}
\end{defi}

\begin{theorem}{\cite[Theorem 8.21,\ Theorem 8.23]{chief_series}}\label{thrm:chief factors for elementary}
Let $G$ be an elementary group and $K/N$ a chief factor of $G$. Then exactly one of the following hold:
\begin{enumerate}
    \item The group $K/N$ has decomposition rank $2$ or $\omega+1$.
    \item The group $K/N$ is of semisimple type and has decomposition rank greater than $\omega+1$.
    \item The group $K/N$ is a group of stacking type and decomposition rank greater than $\omega+1$
\end{enumerate}
\end{theorem}

\subsection{Discrete residual}

Given $G$ a t.d.l.c.s.c.~group and $S\subset \textup{Aut}(G)$ a subset,\ we define the \textbf{discrete residual of $S$ in relation to $G$} as:
$$\textup{Res}_G(S):=\bigcap\{O;O\text{ is an open subgroup of } G \text{,\ and }\phi(O)=O \text{ for every }\phi\in S\}.$$
\indent For such a definition,\ no topology is given to $\textup{Aut}(G)$. Note that the group $\textup{Aut}(G)$ is a topological group under the Braconnier topology. It is only sometimes the case that the Braconnier topology makes $\textup{Aut}(G)$ a locally compact group.\\
\indent Similarly,\ we define $\textup{Res}_G(H)$ to be $\textup{Res}_G(\psi(H))$ on the case there is a homomorphism from $H$ to $\textup{Aut}(G)$ that is obvious from the context,\ such as the conjugation action of $H$ on $G$ inside some ambient group. Notice that for $\textup{Inn}(G)\leqslant \textup{Aut}(G)$, the subgroup of inner automorphisms of $G$, we have $\textup{Res}_G(\textup{Inn}(G))=\textup{Res}(G)$,\ where $\textup{Res}(G)$ is defined as in Definition \ref{defi:discrete residual}.\\
\indent Following are some results that will play an important role in working with the discrete residual of a subset of $G$. These will be central when working with residual height and rank in Section 4.

\begin{theorem}[\cite{residual_theorem_colin},\ Theorem  G]\label{thrm:colin result for res commuting}
Let $G$ be a t.d.l.c. group and $H\leqslant \textup{Aut}(G)$ a compactly generated subgroup. Then $\textup{Res}_{\textup{Res}_G(H)}(H)=\textup{Res}_G(H)$.
\end{theorem}

\begin{theorem}[\cite{residual_theorem_colin},\ Theorem  B]\label{thrm:colin theorem G}
Let $G$ be a t.d.l.c. group and let $H$ be a compactly generated subgroup of $G$. Then there exist an open subgroup $E$ of $G$ with the following properties:
\begin{enumerate}
    \item  $E = H \textup{Res}_G(H)U$,\ where $U$ is a compact open subgroup of $G$;
    \item $E$ is compactly generated;
    \item $\textup{Res}_G(H) = \textup{Res}(E)$,\ so in particular,\ $\textup{Res}_G(H)$ is normal in $E$.
\end{enumerate}
\end{theorem}

\begin{prop}[\cite{residual_theorem_colin},\ Proposition 3.13]\label{prop:res do what i want in elementary}
Let $G$ be a t.d.l.c. group. The following are equivalent:
\begin{enumerate}
    \item For every nontrivial compactly generated closed subgroup $H$ of $G$,\ then $\textup{Res}_G(H)\ngeqslant H$;
    \item For every non-trivial compactly generated closed subgroup $H$ of $G$,\ then
$\textup{Res}(H)\neq H$.
\end{enumerate}
\end{prop}

It then is the case that if $G$ is an elementary group and $K\in\mathcal{K}(G)$ then $\textup{Res}_{G}(K)\geqslant K$ if,\ and only if,\ $K=\{1\}$.

\section{Building groups acting on an infinite regular tree}

\subsection{Directed trees}

Here we will define our $\aleph_0$-regular directed tree and colour it so we can later define our group action in relation to the colouring. We will call a $\aleph_0$-regular tree $T$ a tree such that each vertex has countable degree. That is,\ for each $v\in VT$,\ the degree of the vertex $v$ is countable infinity.
\index{countable regular tree|ndx}

\begin{defi}\index{$\aleph_0$-regular tree|ndx}\index{infinite ray|ndx}\index{infinite line|ndx}\index{ends of a tree|ndx}\index{$\vec{T_{\chi}}$|ndx}\index{$\textup{Aut}(\vec{T_\chi})$|ndx}
Let $T$ be a $\aleph_0$-regular tree. We define the following:
\begin{itemize}
    \item An \textbf{infinite ray} of the tree to be a sequence $\delta:=\{v_n\}_{n\in\mathbb{N}}$ such that $(v_i,v_{i+1})$ are edges of $T$ and given $v_i,v_j\in VT$ if $i\neq j$ then $v_i\neq v_j$.
    \item An \textbf{infinite line} of the tree to be a sequence $\{v_n\}_{n\in \mathbb{Z}}$ such that $(v_i,v_{i+1})$ are edges of $T$ and given $v_i,v_j\in VT$ if $i\neq j$ then $v_i\neq v_j$.
    \item Given $\delta_1=\{v_n\}_{n\in\mathbb{N}},\ \delta_2=\{w_n\}_{n\in\mathbb{N}}$ rays of the infinite tree $T$ we say that \textbf{two rays $\delta_1$,\ $\delta_2$ are equivalent},\ denoted as $\delta_1\sim \delta_2$,\ if there exist $m_1,\ m_2\in\mathbb{N}$ such that for all $i\in\mathbb{N}$ we have $v_{m_1+i}=w_{m_2+i}$. The relation $\sim$ is an equivalence relation on the set of rays of the tree.
    \item Equivalence classes of rays of the tree are called \textbf{ends} of $T$. We denote them as $\chi$.
    \item Let $\chi$ be a fixed end from $T$. Given $e$ an edge of $T$,\ there exist $\delta$ a ray of $\chi$ including $e$. The \textbf{orientation of $e$ in relation to $\chi$} is the orientation of $e$ given by the ray $\delta$. We denote the tree where all edges are oriented towards $\chi$ as $\vec{T_{\chi}}$. (Figure \ref{figure:image tree}) 
    \begin{figure}[ht]
\includegraphics[scale=0.2]{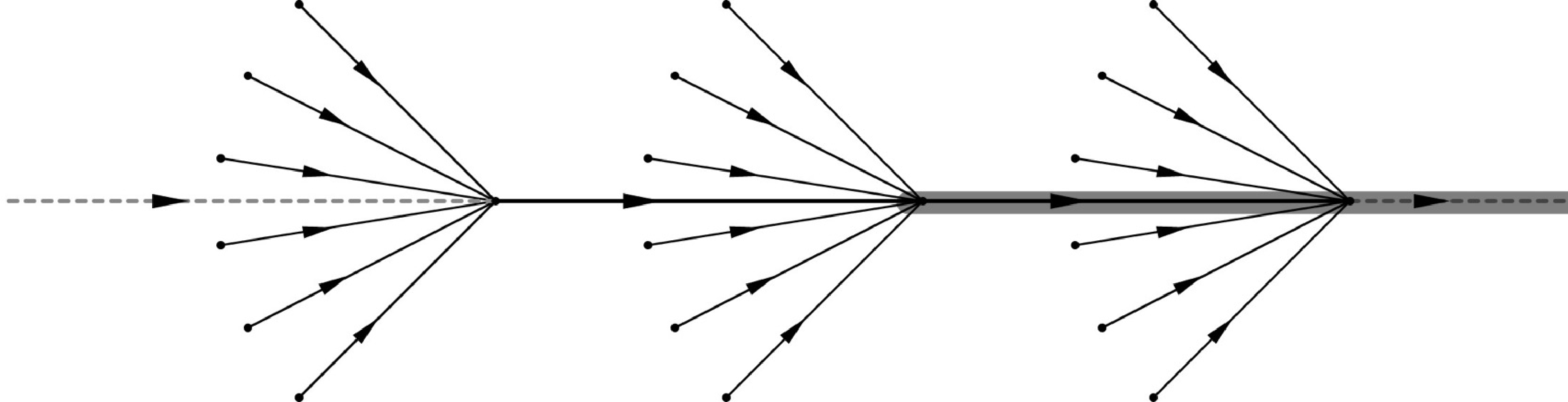}\\
\caption{Example of $\vec{T_{\chi}}$. }\label{figure:image tree}
\floatfoot{The fixed end $\chi$ is represented by the highlighted ray.}
\end{figure}
    \item An \textbf{automorphism of $\vec{T_{\chi}}$} is an element from $\textup{Aut}(T)$ preserving the orientation of the edges of $\vec{T_{\chi}}$.
    \item Given $\vec{T_{\chi}}$,\ we denote $V\vec{T_{\chi}}$ the \textbf{set of vertices} of the directed tree and $E\vec{T_{\chi}}$ the \textbf{set of oriented edges} of the directed tree. 
\end{itemize}
\end{defi}

The intersection of two equivalent rays $\delta_1$,\ $\delta_2$ is a ray. Hence,\ given $\chi$ an end,\ the orientation on $\vec{T_{\chi}}$ is unique.\\
\indent Given an edge $e$ in the directed tree $\vec{T_\chi}$,\ we write it as $e:=(v,w)$ when we want to specify $v$ is the initial vertex of $e$ and $w$ is the terminal vertex of $e$,\ that is,\ $e$ points towards $w$.\\
\indent Given $v\in V\vec{T_{\chi}}$ we define the set of \textbf{inner vertices of $v$}\index{$\textup{inn}^n(v)$|ndx}\index{inner vertices|ndx} as $\textup{inn}^1(v)=:\left\{ w\in V\vec{T_{\chi}}; \ (w,v) \in E\vec{T_{\chi}}\right\}$. Recursively we define for $m\geqslant 2$,\ $\textup{inn}^m(v):=\cup_{w\in \textup{inn}^{m-1}(v)}\textup{inn}^{1}(w)$.\\
\indent Below we have an example to illustrate $\textup{inn}^1(v)$ and $\textup{inn}^2(v)$ for a directed tree. Here $\textup{inn}^1(v)=\{v_1,v_2\}$ and $\textup{inn}^2(v)=\{v_{11},v_{12},v_{21},v_{22}\}$.\\
\begin{figure}[ht]
\includegraphics[scale=0.45]{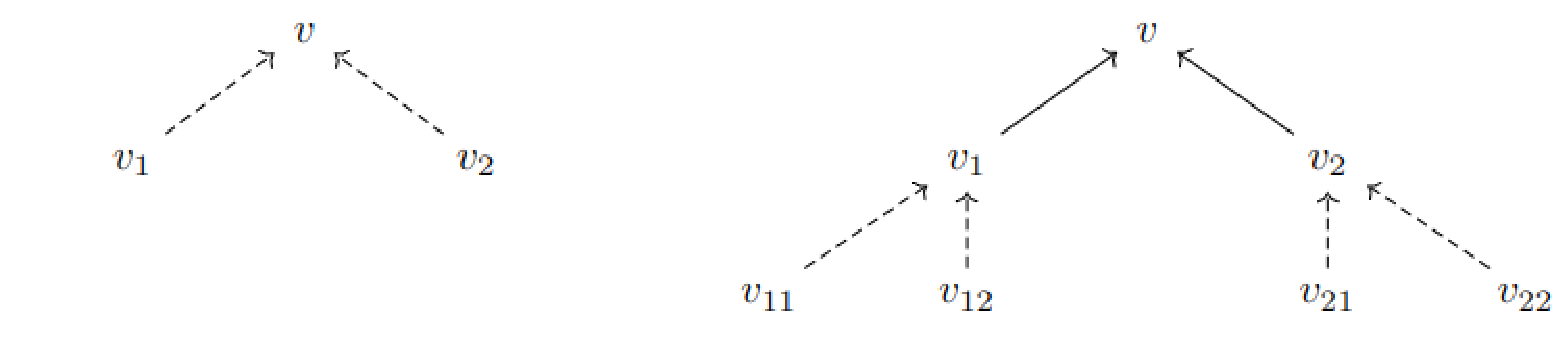}\\
\caption{Example of $\textup{inn}^n(v)$ for $n=1,2$.}
\end{figure}
\indent Given a countable set $X$,\ a \textbf{colouring of $\vec{T_{\chi}}$ in relation to $X$}\index{colouring of a tree|ndx} is a function $c:V\vec{T_{\chi}}\rightarrow X$ such that for each $v\in V\vec{T_{\chi}}$,
$$c_v:=c\vert_{\textup{inn}^1(v)}: \ \textup{inn}^1(v)\rightarrow X$$
is a bijection. We call the colouring \textbf{ended}\index{ended colouring|ndx} if there exist a \textbf{monochromatic ray}\index{monochromatic ray|ndx} $\{v_n\}_{n\in\mathbb{N}}\in \chi$,\ that is,\ $c(v_n)=c(v_m)$ for all $n,m\in\mathbb{N}$. We will always assume our colouring is ended without any additional comment.\\
\indent We can define a local action of $g\in \textup{Aut}(\vec{T_{\chi}})$ at $v\in V\vec{T_{\chi}}$ that tells us how the colouring in $\textup{inn}^1(v)$ moves as follows:
$$\sigma(g,\ \ v)\colon= c_{g(v)}\circ g\circ c_v^{-1}\in  \text{Sym}(X)$$
and,\ with this action of $\textup{Aut}(\vec{T_{\chi}})$ on the colourings of $\textup{inn}^1(v)$,\ we can define the group $E_X(G,\ U)$ as follows:

\begin{defi}\label{defi:EX(G,U)}\index{singularities|ndx}\index{$S_U(g)$|ndx}\index{$E_X(G,\ U)$|ndx}
Given $g\in \textup{Aut}(\vec{T_{\chi}})$ and $U\leqslant $ Sym$(X)$ a subgroup,\ we define the \textbf{singularities} of $g$ in relation to $U$ as

$$S_U(g)\colon=\left\{ v\in V\vec{T_{\chi}};\ \sigma(g,\ v)\notin U \right\}.$$

\indent Suppose that $(G,\ X)$ is a t.d.l.c.s.c.~permutation group and let $U\in \mathcal{U}(G)$. We define the group $E_X(G,\ U)\leqslant \textup{Aut}(\vec{T_{\chi}})$ as the set of all $g\in \textup{Aut}(\vec{T_{\chi}})$ such that $S_U(g)$ is finite and $\sigma(g,\ \ v)\in G$ for all $v\in V\vec{T_{\chi}}$. In the case $g\in E_X(G,\ U)$ we will denote $S(g):=S_U(g)$.
\end{defi}
In other words,\ we are getting a subgroup of $\textup{Aut}(\vec{T_{\chi}})$ with action on the colouring of the edges as the action of $G$ in $X$. The condition of acting as an element of $U$ almost everywhere will be necessary to define the topology.\\ 
\indent The assumption that the colouring is ended ensures that there is a unique infinite line $\Tilde{\delta}:=\{v_n\}_{n\in\mathbb{Z}}$ such that for all $m\in \mathbb{Z}$ the ray $\{v_n\}_{n\leqslant m}\in \chi$,\ and for every $i,\ j\in\mathbb{Z}$ we have $c(v_i)=c(v_j)$. We call such a set $\Tilde{\delta}$ the \textbf{monochromatic line}\index{monochromatic line|ndx} of $\vec{T_{\chi}}$.\\
\indent Given the monochromatic line,\ we define the \textbf{translation}\index{translation|ndx} $t\in E_X(G,\ U)$ as the element $t\in E_X(G,\ U)$ such that $t.v_n=v_{n+1}$,\ for every $n \in \mathbb{Z}$,\ and for every vertex $v$, $\sigma(t,\ v)=1$.

\subsection{The structure of \texorpdfstring{$E_X(G,\ U)$}{TEXT}}

Given $G$ a group acting on $X$ a countable set,\ we say that a sequence $\{g_n\}_{n\in\mathbb{N}}$ converges pointwise to $g\in G$ if for every $x\in X$ we have $g_n x\xrightarrow{n\rightarrow\infty} g x$. The topology induced by such sequences allows us to make $\textup{Aut}(\vec{T_{\chi}})$ a totally disconnected second countable group,\ which is not locally compact. We will show that the subgroups of $E_X(G,\ U)$ of the form $E_X(U,\ U)_{(F)}$,\ for $F\subset V\vec{T_{\chi}}$ a finite subset,\ are compact in the topology of $\textup{Aut}(\vec{T_{\chi}})$. We then show these subgroups can form a neighborhood basis of the identity in $E_X(G,\ U)$,\ giving us the desired t.d.l.c.s.c.~topology.


\begin{prop}\label{prop:compact subgrp}
Let $(G,\ X)$ be a t.d.l.c.s.c.~permutation group and $U\in \mathcal{U}(G)$. Given $v\in V\vec{T_{\chi}}$, the group $E_X(U,\ U)_{(v)}\leqslant \textup{Aut}(\vec{T_{\chi}})$ is compact with respect to the subspace topology.
\end{prop}

\begin{proof}
Under the pointwise convergence,\ given $g\in E_X(G,\ U)\leqslant \textup{Aut}(\vec{T_{\chi}})$ and $v\in S(g)$, then $\{h\in \textup{Aut}(\vec{T_{\chi}});\ h.v=g.v\}$ is an open neighbourhood of $g$ with $v\in S(h)$ for all $h$ in the neighbourhood. Hence the complement of $E_X(U,\ U)$ is open. Therefore, $E_X(U,\ U)=\{g\in E_X(G,\ U);\ S(g)=\emptyset\}$ is a closed subgroup of $\textup{Aut}(\vec{T_{\chi}})$. Fix $v  \in  V\vec{T_{\chi}}$ and consider $W:=E_X(U,\ U)_{(v)}$. The action of $W$ on $\textup{inn}^1(v)$ is equivalent to the action of $U$ on $X$. Moreover,\ given $x\in X$,\ the stabilizer subgroup $U_{(x)}:=G_{(x)}\cap U\leqslant G$ is compact,\ and the index $\left[ U:\ U_{(x)}\right]$ is finite. The orbit-stabilizer theorem implies that the orbits of $W\curvearrowright  \textup{inn}^1(v)$ are finite.\\
\indent Given $w\in V\vec{T_{\chi}}$,\ there exist $m,n\in \mathbb{N}$ such that $w\in t^n \textup{inn}^m(v)$. We will show by induction on $m$,\ for $n=0$,\ that the action $W \curvearrowright V\vec{T_{\chi}}$ has finite orbits. We then prove for the case $n>0$.\\
\indent As seen above,\ it is true for the case $n=0,\ m=1$. Assume it is true for some $m>0$ and $n=0$. If $w\in \textup{inn}^{m+1}(v)$,\ there exist $w'\in \textup{inn}^{m}(v)$ a vertex such that $(w,\ w')$ is a directed edge. Notice that $|E_X(U,\ U)_{(v)}.w|=|E_X(U,\ U)_{(v)}.w'|[U:U_{(x)}]$,\ where $x=c(w)\in X$. Hence,\ by the induction hypothesis,\ the orbit of $w$ is finite.\\
\indent Notice that for all $n\in\mathbb{N}$,\ $t^{n}v$ is fixed by $E_X(U,\ U)_{(v)}$. Hence $E_X(U,\ U)_{(t^{n}v)}\leqslant E_X(U,\ U)_{(v)}$ for all $n\in\mathbb{N}$. Hence,\ if $w\in t^{n}\textup{inn}^{m}(v)=\textup{inn}^m(t^{n}v)$ then,\ by the argument above on the group $E_X(U,\ U)_{(t^{n}v)}$,\ the orbit of $w$ is finite.\\
\indent The induction above shows that for all vertices $w\in V\vec{T_{\chi}}$,\ the orbit of $w$ is finite. As any closed subgroup of $\textup{Aut}(\vec{T_{\chi}})$ having finite orbits is compact,\ it is the case that $E_X(U,\ U)_{(v)}$ is compact \cite[Theorem 2.1]{symmetric_group}.
\end{proof}

The idea now is to use Proposition \ref{prop:building groups} and Proposition \ref{prop:compact subgrp} to give our desired topology to $E_X(G,\ U)$. To construct the basis of open,\ compact subgroups for $E_X(G,\ U)$,\ we will need to define some sets related to the permutation group $(G,\ X)$ and $V\vec{T_{\chi}}$.\\
\indent Suppose that $(G,\ X)$ is a t.d.l.c.s.c.~permutation group and $U\in\mathcal{U}(G)$. Since $U$ is open, and $G$ has compact open point stabilizers,\ there is $Y\subset X$ a finite subset such that $G_{(Y)}\leqslant U$. Fixing such $Y$ we define,\ for $h\in G$,\ the set $Y_h:=h^{-1}.Y\subset X$. The set $Y_h$ has the property that $hG_{(Y_h)}h^{-1}\leqslant U$. Let $\vec{T_{\chi}}$ an $\aleph_0$-regular tree with an ended colouring in relation to $X$. For each $v\in V\vec{T_{\chi}}$,\ let 
$$N_h(v):=\left\{ w\in \textup{inn}^1(v);\ c_v(w)\in Y_h  \right\}.$$
Observe that $|N_h(v)|=|Y_h|\in\mathbb{N}$. For a finite set of vertices $Z\subset V\vec{T_{\chi}}$,\ we put $N_h(Z):=\bigcup_{v\in Z} N_h(v)$.\\
\indent As seen in Definition \ref{defi:EX(G,U)},\ given $g\in E_X(G,\ U)$ and $v\in V\vec{T_{\chi}}$ it is always the case that $\sigma(g,\ v)\in G\leqslant\ $Sym$(X)$. Given $g\in E_X(G,\ U)$ and $v,\ w\in V\vec{T_{\chi}}$,\ we define $N_{\sigma(g,\ v)}(w)$ as above.
\begin{prop}\label{prop: comp gener}
Let $(G,\ X)$ be a t.d.l.c.s.c.~permutation group (not necessarily transitive) with $U\in\mathcal{U}(G)$. There is a t.d.l.c.s.c.~group topology on $E_X(G,\ U)$ such that the inclusion $E_X(U,\ U)_{(v)}$ into $E_X(G,\ U)$ is an open embedding for any $v\in V\vec{T_{\chi}}$,\ where the topology of $E_X(U,\ U)_{(v)}$ is as given in (Proposition \ref{prop:compact subgrp}).
\end{prop}

\begin{proof}
Let $v\in V\vec{T_{\chi}}$ be a fixed vertex and define $W:=E_X(U,\ U)_{(v)}$. We want to show that the set
$$\mathcal{F}=\left\{\bigcap_{i=1}^{n} g_{i} W_{(S)} g_{i}^{-1};\ n \in \mathbb{N},\ g_{1},\ \ldots,\ g_{n} \in E(G,\ U),\ \text { and } S \subset V \vec{T_{\chi}} \text { is finite }\right\}$$
defines a basis of open,\ compact subgroups of $E_X(G,\ U)$.\\
\indent As this family consist of subgroups of $E_X(G,\ U)$,  for every $U\in \mathcal{F}$ we have $U.U\subset U$. The family $\mathcal{F}$ is, by definition, invariant under conjugation and finite intersections. By Proposition \ref{prop:building groups},\ this set defines a group topology $\tau$ on $E_X(G,\ U)$ such that $\mathcal{F}$ is a basis at $1$ of open subgroups. We need to prove that the natural embedding of $W$ into $E_X(G,\ U)$ is continuous, from which it will follow that $\mathcal{F}$ is a basis of compact,\ open subgroups.\\
\indent First prove that given $g\in E_X(G,\ U)$ there exist $R\subset V\vec{T_{\chi}}$ such that $gW_{(R)}g^{-1}\leqslant W$. Notice that it is the case $gWg^{-1}\leqslant E_X(G,\ U)_{(g.v)}$, but not necessarily it is the case $gWg^{-1}\leqslant E_X(U,\ U)_{(g.v)}$. We need to describe a finite set $R$ such that for every $w\in W_{(R)}$ we have $S(gwg^{-1})=\emptyset$.\\ 
\indent For each $g\in E_X(G,\ U)$,\ set
$$R:=\bigcup_{v'\in S(g)}N_{\sigma(g,v')}\left( S(g)\cup \{ g.v',\ v' \} \right)\cup \{g^{-1}v\}.$$
We argue that $gW_{(R)}g^{-1}\leqslant W$. It is immediate that $gW_{(R)}g^{-1}\leqslant E_X(G,\ U)_{(v)}$,\ as for every $w\in W_{(R)}$ it is the case that $gwg^{-1}.v=gw(g^{-1}.v)=g(g^{-1}.v)=v$. It suffices to show that the local action of each $gwg^{-1}\in gW_{(R)}g^{-1}$ at any $x\in V\vec{T_{\chi}}$ then $\sigma(gwg^{-1},\ x)$ is in $U$. For such a $gwg^{-1}$ and $x$,\ we see that
$$\begin{aligned}
\sigma\left(g w g^{-1},\ x\right) &=\sigma\left(g w,\ g^{-1}.x\right) \sigma\left(g^{-1},\ x\right) \\
&=\sigma\left(g,\ w g^{-1}.x\right) \sigma\left(w,\ g^{-1}.x\right) \sigma\left(g^{-1},\ x\right).
\end{aligned}$$
If $x\notin S(g^{-1})$,\ then $g^{-1}x\notin S(g)$,\ and thus,\ $wg^{-1}x\notin S(g)$. The previous equation then implies that $\sigma(gwg^{-1},\ x)\in U$.\\
\indent Let us now suppose that $x\in S(g^{-1})$. We see immediately that $wg^{-1}.x=g^{-1}.x$,\ so
$$\begin{aligned}
\sigma\left(g w g^{-1},\ x\right) &=\sigma\left(g,\ g^{-1}.x\right) \sigma\left(w,\ g^{-1}.x\right) \sigma\left(g^{-1},\ x\right) \\
&=\sigma\left(g,\ g^{-1}.x\right) \sigma\left(w,\ g^{-1}.x\right) \sigma\left(g,\ g^{-1}.x\right)^{-1}.
\end{aligned}$$
Set $h:=\sigma(g,\ g^{-1}.x)\in G\leqslant\ \textup{Sym}(X)$. The element $w$ fixes $N_h(g^{-1}.x)$,\ by definition. Hence,\ the local action $\sigma(w,\ g^{-1}.x)$ is an element of $G_{(Y_h)}$. Recalling how we choose $Y_h$,\ we conclude that $h\ \sigma(w,g^{-1}.x)\ h^{-1}\in U$. We thus demonstrate that $gW_{(R)}g^{-1}\leqslant W$. It follows that $W_{(R)}\leqslant W\bigcap g^{-1}W g$.\\
\indent For each $O\in\mathcal{F}$,\ the last paragraph ensures there exist $W_{(F)}\leqslant O$ for some $F$ finite collection of vertices. The preimage of $O$ under the inclusion map $W\hookrightarrow (E_X(G,\ U),\tau)$ is then open in $W$,\ as it contains the subgroup $W_{(F)}$. Hence the inclusion map is continuous. Since $W$ is compact,\ we deduce that $W$ is indeed isomorphic as a topological group to an open,\ compact subgroup of $E_X(G,\ U)$,\ namely its image. Hence under the topology induced by $\mathcal{F}$,\ $E_X(G,\ U)$ is a t.d.l.c. group. As the basis of $E_X(G,U)$ are groups of the form $W_{(F)}$,\ the inclusion of $W$ into $E_X(G,\ U)$ is an open embedding.\\
\indent It remains to show $\tau$ is a second countable topology. As $E(G,\ U)_{(v)}$ has countable index in $E_X(G,\ U)$,\ it suffices to show that $E_X(G,\ U)_{(v)}$ is second countable. To this end,\ let $\Tilde{G}$ be a countable dense subgroup of $G$ and form $E_X(\Tilde{G},\{1\})$. The group $E_X(\Tilde{G},\{1\})$ is plainly countable. As $\Tilde{G}U=G$ it follows that $E_X(\Tilde{G},\{1\})_{(v)}W=E_X(G,\ U)_{(v)}$. We thus deduce that $\tau$ is also a second countable topology.
\end{proof}
The next few results will be useful to understand the structure of our topological group and prove some properties necessary when building elementary groups.\\
\indent To build groups with high decomposition rank,\ we need to be able to iterate the group construction. Therefore first,\ we will prove that the action of $E_X(G,\ U)$ on the set $V\vec{T_{\chi}}$ is transitive under the condition $(G,\ X)$ is transitive. \\
\indent For the next lemma,\ given $F\subset V\vec{T_{\chi}}$ a finite subset,\ we will denote $\textup{inn}^1(F):=\bigcup_{v\in F} \textup{inn}^1(v)$.  \\
\indent We will define the function $d:V\vec{T_{\chi}}\times V\vec{T_{\chi}} \rightarrow \mathbb{N}$ by $d(w,\ v)=\min\{n\in\mathbb{N};$ there exist a geodesic from $v$ to $w$ or from $w$ to $v$ of length $n\}$.\\
\indent For the next lemma we identify the group $G$ with a subgroup of $E_X(G,\ U)$ as follows. Fix a vertex $v\in V\vec{T_{\chi}}$. Then $$G\cong\left\{ g\in E_X(G,\ U)_{(v)}; \ \forall u\notin \textup{inn}^1(v),\ \sigma(g,\ u)=1 \right\}\leqslant P_X(G,\ U)$$
the subgroup of $E_X(G,\ U)$ that only acts on $\textup{inn}^1(v)$ and has a trivial action everywhere else.

\begin{lemma}\label{lemma: element t and G}
For $G$,\ $t$ as above and $v$ in the monochromatic line,\ the following hold:
\begin{enumerate}[(1)]
    \item The group $\langle G,t\rangle$ acts transitively on $V\vec{T_{\chi}}$ and $E\vec{T_{\chi}}$;
    \item For all $w\in V\vec{T_{\chi}}$ such that $w$ is connected to $v$ by a directed geodesic $w=w_0,\cdots,w_n=v$ from $w$ to $v$,\ there is $\gamma\in \langle G,t\rangle$ for which $\gamma.w=v$ and $\sigma(\gamma,\ u)=1$ for all $u\notin \textup{inn}^1(\left\{ w_1,\cdots,w_n \right\})$. Such an element can be made so that $\sigma(\gamma^{-1}h\gamma,w)=\sigma(h,v)$ and $\sigma(\gamma^{-1}h\gamma,\ u)=1$ for all $u\notin \textup{inn}^1(\left\{ w_1,\cdots,w_n \right\})$ and $h\in G$.
\end{enumerate}
\end{lemma}

\begin{proof}
$(1)$ Let $v$ be our fixed vertex and $w$ in $\textup{inn}^1(v)$. As $G$ acts transitively on $X$,\ $t^{-1}v$ is an element of $\textup{inn}^1(v)$ and $c_v$ is a bijection from $\textup{inn}^1(v)$ to $X$,\ there is $g_1\in G$ such that $g_1t^{-1}v=w$. Inductively it is easy to see that if $w$ is in $\textup{inn}^n(v)$,\ then there are $g_1,g_2,\ldots,g_n\in G$ such that $g_nt^{-1}\ldots g_1t^{-1}. v=w$. On the other hand,\ as $V\vec{T_{\chi}}=\bigcup_{n\in \mathbb{N}}\left(\cup_{m\in\mathbb{N}}\textup{inn}^{m}(t^n.v)\right)$ there is $m\in\mathbb{N}$ such that $w$ is in $\cup_{n\in\mathbb{N}}\textup{inn}^n(t^{m}.v)$. By the argument above,\ there is  $h\in \langle G,t\rangle$ such that $h .v=t^{-m}. w$. It then follows that $t^{m}h. v= w$,\ as desired. \\
\indent $(2)$ We will argue by induction on $d(w,\ v)$. The base case,\ $d(w,\ v)=0$,\ is immediate. Suppose that $w\in V\vec{T_{\chi}}$ and that there is a directed geodesic $w=w_0,\cdots,w_{n+1}=v$. By the induction hypothesis,\ there is $\gamma\in \langle G,t\rangle$ such that $\gamma.w_1=v$ and $\sigma(\gamma,\ u)=1$ for all $u\notin \textup{inn}^1(\{w_2,\cdots,w_{n+1}\})$. Since $G\curvearrowright X$ transitively,\ there is $g\in G$ for which $g\gamma.w$ lies on the monochromatic line. It follows that $tg\gamma.w=v$.\\
\indent We now consider the local action of $tg\gamma$ on $V\vec{T_{\chi}}$:

$$\begin{aligned}
\sigma(t g \gamma,\ x) &=\sigma(t g,\ \gamma.x) \sigma(\gamma,\ x) \\
&=\sigma(t,\ g \gamma.x) \sigma(g,\ \gamma.x) \sigma(\gamma,\ x) \\
&=\sigma(g,\ \gamma.x) \sigma(\gamma,\ x)
\end{aligned}$$

\noindent,\ where the last line follows since $\sigma(t,\ u)=1$ for all $u\in V\vec{T_{\chi}}$. The local action $\sigma(g,\ u)$ is trivial for all $u\neq \textup{inn}^1(v)$,\ and the induction hypothesis ensures that $\sigma(\gamma,\ u)$ is trivial for all $u\notin inn (\{w_2,\cdots,w_{n+1}\})$. The proof of the induction step is now complete. \\
\indent For the second part,\ let $h$ in $G$. The local action of $\gamma^{-1}h\gamma$ is as follows:

$$\begin{aligned}
\sigma(\gamma^{-1}h \gamma,\ x) &=\sigma(\gamma^{-1}h,\ \gamma.x) \sigma(\gamma,\ x) \\
&=\sigma(\gamma^{-1},\ h \gamma.x) \sigma(h,\ \gamma.x) \sigma(\gamma,\ x) 
\end{aligned}.$$
Observing that $\sigma(\gamma^{-1},h\gamma.x)=\sigma(\gamma,\gamma^{-1}h\gamma.x)^{-1}$,\ we deduce from claim $(2)$ that $\sigma(\gamma^{-1},h\gamma.x)\neq 1$ only when $x\in \gamma^{-1}h^{-1}\gamma(\textup{inn}^1(\{ w_1,\cdots,w_n \}))$. The element $h$ fixes all elements of the monochromatic ray after $v$,\ and thus,\ $h$ also fixes $\gamma.w_i$ for $1\leqslant i\leqslant n+1$. Then $\gamma.w_i$ necessarily falls on the monochromatic line,\ since every vertex appears as the initial vertex of exactly one directed edge. We thus indeed have $\sigma(\gamma^{-1},h\gamma.x)\neq 1$ only when $x\in \textup{inn}^1(\{ w_1,\cdots,w_{n+1} \})$. Since $\sigma(h,\gamma.x)\neq 1$ when $x=w$ it follows that $\sigma(\gamma^{-1}h\gamma,w)=\sigma(h,v)$. It is also the case that $\sigma(\gamma^{-1}h\gamma,\ x)=1$ for all $x\notin \textup{inn}^1(\left\{ w_1,\cdots,w_n \right\})$.
\end{proof}

\begin{prop}\label{prop:writing ex(g,U) as semidirect product}\index{$P_X(G,\ U)$|ndx}
Let $(G,\ X)$ be a t.d.l.c.s.c.~permutation group and define 
$$P_X(G,\ U):=\bigcup_{v\in V\vec{T_{\chi}}} E_X(G,\ U)_{(v)}$$
the set of all elements that fix at least one vertex of $\vec{T_{\chi}}$. Then $P_X(G,\ U)$ is an open normal subgroup of $E_X(G,\ U)$ and $E_X(G,\ U)=P_X(G,\ U)\rtimes \langle t\rangle \cong P_X(G,\ U)\rtimes \mathbb{Z}$. It is also the case that for every $g\in E_X(G,\ U)$, either $g$ fixes a vertex, or $g$ translates down the line $\Tilde{\delta}$.
\end{prop}

\begin{proof}
Let $g\in P_X(G,\ U)$ and $\{v_i\}_{i\in \mathbb{Z}}$ the monochromatic line of $\vec{T_{\chi}}$. Notice that if $v\in V\vec{T_{\chi}}$ is such that $g.v=v$ and $v\in \textup{inn}^1(w)$, then $g$ also fixes $w$. Hence,\ if $g$ fixes a vertex in $\vec{T_{\chi}}$ then,\ by how the oriented edges are defined,\ there exist $N\in \mathbb{Z}$ such that for all $i\geqslant N$,\ we have $g v_i=v_i$. For that reason, we have
$$P_X(G,\ U)=\bigcup_{i\in \mathbb{Z}} E_X(G,\ U)_{(v_i)}$$
that is,\ it is the union of an increasing chain of subgroups. It then follows that the subset $P_X(G,\ U)$ is a subgroup of $E_X(G,\ U)$. It also follows that it is the union of open subgroups,\ hence is open.\\
\indent To show it is normal let $h\in E_X(G,\ U)$,\ $g\in P_X(G,\ U)$ and $w\in V\vec{T_{\chi}}$ be such that $gw=w$. Notice that the vertex $h^{-1}w$ is fixed by the element $h^{-1}gh$. Hence $h^{-1}gh\in P_X(G,\ U)$,\ and the subgroup is normal.\\
\indent To show $\langle P_X(G,\ U),\ t\rangle = E_X(G,\ U)$ let $g\in E_X(G,\ U)$,\ fix $v$ in the monochromatic line and let $G\leqslant P_X(G,\ U)$ as in Lemma \ref{lemma: element t and G}. Define $w=g^{-1}v$. Then by Lemma \ref{lemma: element t and G} $(2)$ there is $\gamma\in \langle G,t\rangle\leqslant \langle P_X(G,\ U),\ t\rangle$ such that $\gamma w=v$. Hence $(\gamma^{-1}g)w=w$,\ that is,\ $\gamma^{-1}g=h\in P_X(G,\ U)$. As $t$ normalizes $G$, there are $n\in\mathbb{Z}$ and $g'\in G$ such that $\gamma=t^ng'$. It then follows that $g=t^n(g' h)\in\langle P_X(G,\ U),\ t\rangle$.\\
\indent Because $t$ moves all vertices it is easy to see that $P_X(G,\ U)\cap \langle t\rangle =\{1\}$. Hence $E_X(G,\ U)=P_X(G,\ U)\rtimes \langle t \rangle$.
\end{proof}

To work with the rank more easily,\ we also need to show that the group is compactly generated under some conditions.

\begin{prop}\label{prop:new grp is comp gener}
Suppose that $(G,\ X)$ is a transitive t.d.l.c.s.c.~permutation group. If $G$ is compactly generated,\ the $E_X(G,\ U)$ is compactly generated for any $U\in\mathcal{U}(G)$.
\end{prop}

\begin{proof}
Let $\Tilde{\delta}$ be the monochromatic line and fix a vertex $v\in \Tilde{\delta}$. Identify the group $G$ with the subgroup of $E_X(G,\ U)$ as in Lemma \ref{lemma: element t and G}. Let $Z$ be a compact generating set for $G$. \\
\indent Let $H:=\langle Z,E_X(U,\ U)_{(v)},t\rangle\leqslant E_X(G,\ U)$. We argue by induction on the number of singularities $|S(g)|$ that if $g\in E_X(G,\ U)$ then $g\in H$,\ that is,\ $H=E_X(G,\ U)$. For the base case,\ $|S(g)|=0$,\ we see that $g\in E_X(U,\ U)$. By Proposition \ref{prop:writing ex(g,U) as semidirect product},\ it follows that the element $g$ either fixes a vertex or translates down the line $\Tilde{\delta}$. For the former case,\ $g$ fixes some $w$ on the ray $\delta$. Taking $n\in\mathbb{Z}$ such that $t^n .w=v$,\ we see that $t^ngt^{-n}$ fixes $v$,\ and $t^ng t^{-n}\in E_X(U,\ U)_{(v)}$. Thus $g\in H$. For the latter case,\ there are $w$ and $w'$ on the ray $\delta$ such that $g.w=w'$. Taking $n\in\mathbb{Z}$ such that $t^n.w'=w$,\ the element $t^ng$ has no singularities and fixes a point. Applying the first case,\ we deduce that $g\in H$.\\
\indent Let us now suppose that $|S(g)|=n+1$. As in the base case,\ we may assume that $g$ fixes $v$ by acting with $t$. Acting again by $t$ if needed,\ we may further suppose that there exists $w\in S(g)$ such that there is a directed geodesic from $w$ to $v$ and the local action of $g$ on each vertex on the geodesic is by an element of $U$. Let $w\in S(g)$ admit such a geodesic and be such that $d(w,v)$ is the least.\\
\end{proof}

Let $(G,\ X)$ a t.d.l.c.s.c.~permutation group,\ $U\in\mathcal{U}(G)$,\ $\vec{T_{\chi}}$ a tree coloured by the countable set $X$ and $v\in V\vec{T_{\chi}}$ be a fixed vertex in the monochromatic ray. One can identify,\ as denoted in Definition \ref{defi:wr product},\ the set $X^n$ with $\textup{inn}^n(v)$ and the permutation group $(G_n,\ X^n)$ with 
$$\left\{ g\in E_X(G,\ U)_{(v)}; \ \forall w\notin \cup_{1\leq i \leq n}\textup{inn}^i(v),\ \sigma(g,\ w)=1 \right\}$$
acting on $\textup{inn}^n(v)$. Given the fixed vertex $v$,\ we define $T_0$\index{$T_0$|ndx} to be the half tree containing $\bigcup_{n\in\mathbb{N}}\textup{inn}^n(v)\cup \{v\}$ and all the edges with vertices in this union. We will denote by $T_R$\index{$T_R$|ndx} its complement.
\begin{lemma}\label{lemma:Gn 101}
Suppose that $(G,\ X)$ is a t.d.l.c.s.c.~permutation group with $U \in \mathcal{U}(G)$. Let $\vec{T_{\chi}}$ be coloured by the countable set $X$,\ and $v\in V\vec{T_{\chi}}$ a fixed vertex in the monochromatic line. Form $E_X(G,\ U)_{(T_R)}$. Then,
\begin{enumerate}[(1)]
    \item For any $n\geqslant 1$,\ the action on the wreath product of $n$ copies of $(G,\ X)$ over $U$,\ denoted $G_n$,\ on the infinitely branching rooted tree gives a continuous closed embedding of $G_n$ into $E_X(G,\ U)_{(T_R)}$;
    \item Under this embedding $\overline{\bigcup_{n\geqslant 1}G_n}=E_X(G,\ U)_{(T_R)}$; and
    \item Letting $t\in E_X(G,\ U)$ be a translation toward the distinguished end $\delta$,\ $$\overline{\bigcup_{n\geqslant 0}t^n (E_X(G,\ U)_{(T_R)})t^{-n}}=P_X(G,\ U)$$
\end{enumerate}
\end{lemma}

\begin{proof}
$(1)$ Observe that the action,\ as defined above,\ is an injective homomorphism from $G_m$ to $E_X(G,\ U)_{(T_R)}$.
\ Notice that in both $G_m$ and $E_X(G,\ U)_{(T_R)}$,\ a sequence $\{g_n\}_{n\in\mathbb{N}}$ converges to $g$ if,\ and only if there exist $N\in\mathbb{N}$ such that the sequence $\{g_n^{-1}g\}_{n\geqslant N}$ is in the subgroup with all entries in $U$ and converges pointwise to the identity,\ that is,\ the embedding is continuous. If $\{g_n\}_{n\in\mathbb{N}}\subset G_m$ converges to $g$ in $E_X(G,\ U)$ it must converge to an element that acts trivially on the colour of every vertex outside $\cup_{0\leqslant i\leqslant m} \textup{inn}^i(v)$. As every element fixing the colours of the vertices in $V\vec{T_{\chi}}\backslash\cup_{0\leqslant i\leqslant m} \textup{inn}^i(v)$ is an element of $G_m$, by the action defined above,\ it follows that $g\in G_m$.
Hence the group $G_m$ can be seen as a closed subgroup of $E_X(G,\ U)_{(T_R)}$ under this embedding.\\
\indent $(2)$ Let $g\in E_X(G,\ U)_{(T_R)}$. Given $n\in\mathbb{N}$,\ let $g_n\in G_n$ be the element such that for every $m>n$, $(x_1,\cdots,x_m)\in X^m$ we have $g_n.x_i=g.x_i$,\ if $i\leqslant n$,\ and $\sigma(g_n,\ x_i)=1$ if $i>n$. Let $\{g_n\}_{n\in\mathbb{N}}$ be the sequence defined as above. Observe that because $S(g)$ is finite,\ there exist $N\in\mathbb{N}$ such that for all $n\geqslant N$,\  $g_ng^{-1}\in E_X(U,\ U)_{(v)}$. Because $E_X(U,\ U)_{(v)}$ has the pointwise convergence,\ we get that $g_ng^{-1}\rightarrow 1$ as $n\rightarrow \infty$,\ proving our assumption.\\
\indent (3) Let $g\in P_X(G,\ U)$. By definition, there is $w\in V\vec{T_{\chi}}$ such that $g.w=w$. We want to build a sequence of elements in $\cup_{n\geqslant 0} t^{n}(E_X(G,\ U)_{(T_R)})t^{-n}$. Given $n\in \mathbb{N}$ let $T_0^{(n)}:=t^n.T_0$,\  the half tree defined as above but rooted at $t^{n}v$ instead of $v$. Let $k\in \mathbb{N}$ be such that $w\in T_0^{(k)}$. Define $\{g_n\}_{n\in\mathbb{N}}$ as follows:
$$g_n.w'= \left\{ \begin{array}{ll}
    g.w'  & \mbox{if }w'\in \textup{inn}(t^{k+n}.w)  \\
    w' & \mbox{otherwise.}
\end{array}\right.$$
\noindent Therefore $g_n\in E_X(G,\ U)_{(T_R^{(k+n)})}=t^{k+n} E_X(G,\ U)_{(T_R)}t^{-k-n}$. Notice that as $S(g)$ is finite,\ there exist $N\in\mathbb{N}$ such that the subsequence $\{g_ng^{-1}\}_{n\geqslant N}\in E_X(U,\ U)_{(w)}$. Because the subgroup $E_X(U,\ U)_{(w)}$  has the pointwise convergence,\ it follows that $g_ng^{-1}\rightarrow 1$ as $n\rightarrow \infty$,\ as desired. 

\end{proof}

\begin{prop}\label{prop: monolith tree}
Suppose that $(G,\ X)$ is a transitive elementary permutation group. Then,\ $E_X(G,\ U)$ is monolithic (Definition \ref{defi:general defi}),\ and the monolith is
$$M:=\overline{\langle[P_X(G,\ U),P_X(G,\ U)] \rangle}.$$
If $G$ is also topologically perfect (Definition \ref{defi:general defi}),\ then the monolith is equal to $P_X(G,\ U)$.
\end{prop}

\begin{proof}
Given $u\in V\vec{T_{\chi}}$, we define the set of vertices $T^{u}:=V\vec{T_{\chi}}\backslash \cup_{n\in\mathbb{N}}\textup{inn}^n(u)$ and the subgroup $L_u:=E_X(G,\ U)_{(T^{u})}\leqslant E_X(G,\ U)$.\\
\indent Colour $\vec{T_{\chi}}$ by $X$ and let $\delta$ be a monochromatic ray giving the distinguished end $\chi$. Suppose that $H$ is a non-trivial closed normal subgroup of $E_X(G,\ U)$ and let $h\in H$ act non-trivially on $\vec{T_{\chi}}$. By Proposition \ref{prop:writing ex(g,U) as semidirect product} we have $h=t^n p$ for some $n\in \mathbb{Z}$ and $p\in P_X(G,\ U)$. \\

\begin{figure}[ht]
\includegraphics[scale=0.21]{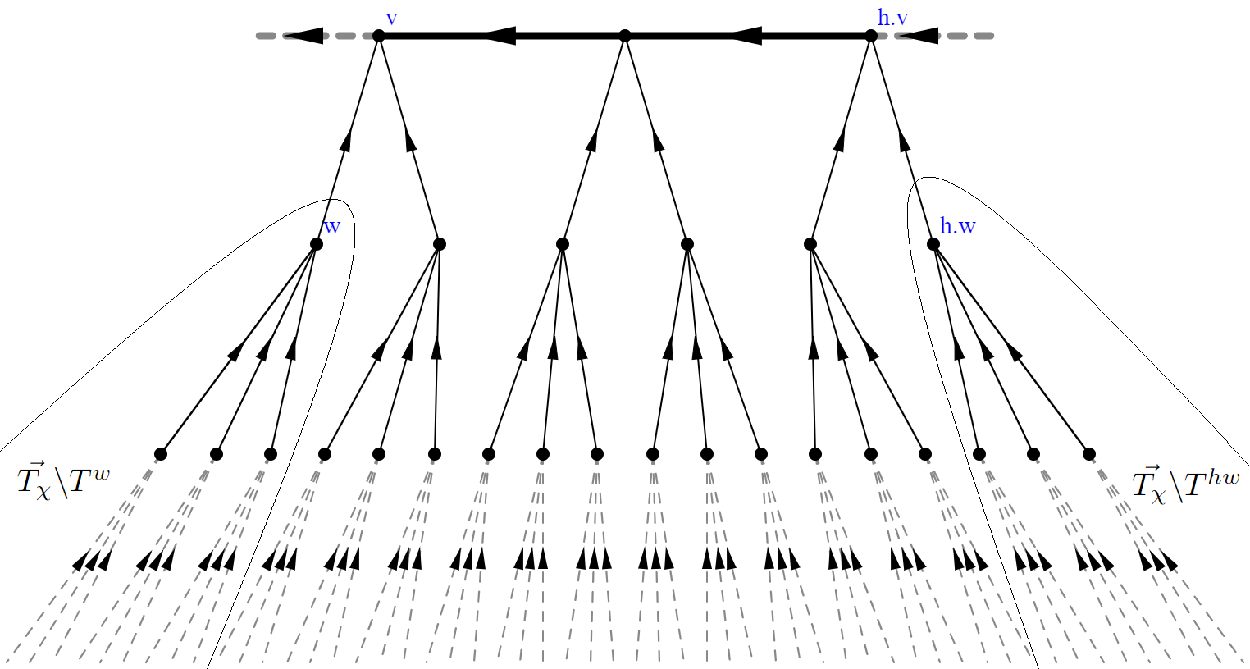}\\
\includegraphics[scale=0.21]{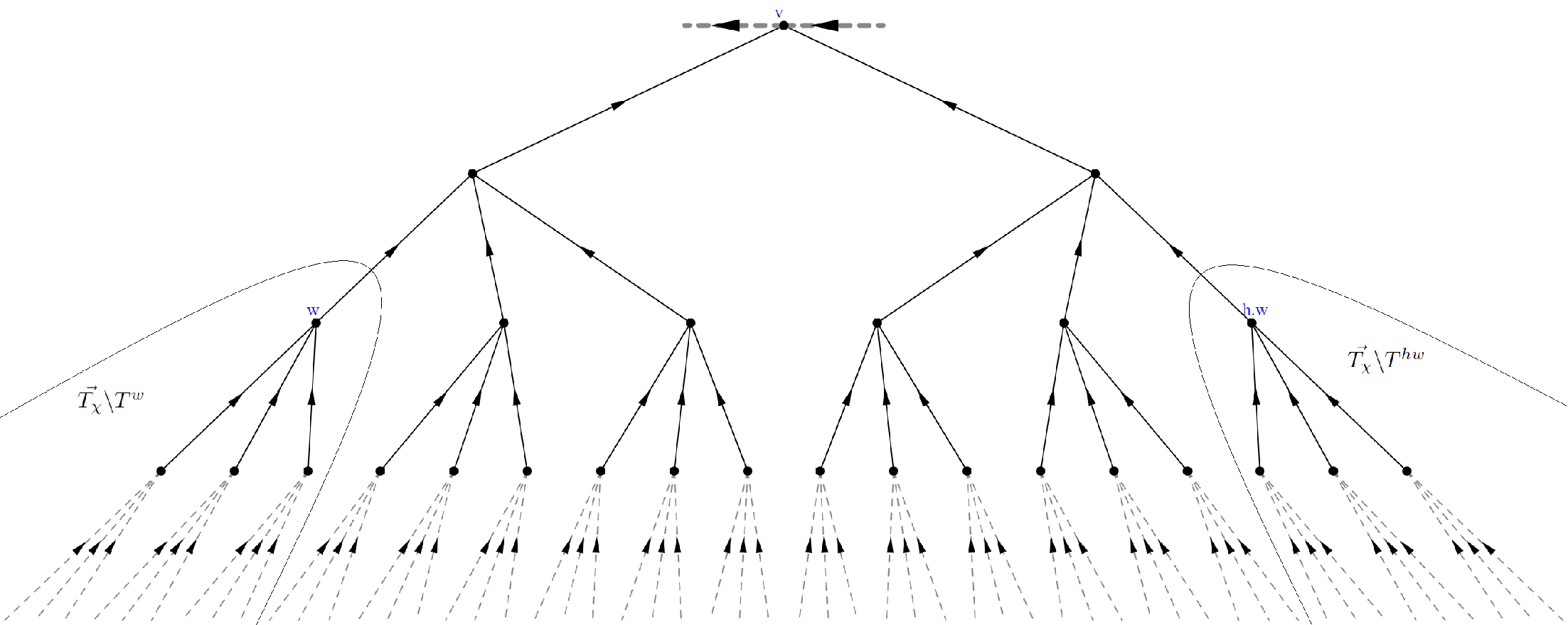}
\caption{The sets $\vec{T_{\chi}}\backslash T^w$ and $\vec{T_{\chi}}\backslash T^{hw}$ are disjoint in both cases,\ hence their respective subsets $\text{supp}(x)$ and $h.\text{supp}(x)=\text{supp}(hxh^{-1})$ are disjoint.}
\floatfoot{The monochromatic ray is represented by the highlighted line on top. Images made using \cite{draw_graph}.}
\end{figure}

\indent First,\ we will prove for the case $n\neq 0$. Since all translations in $E_X(G,\ U)$ must fix the distinguished end $\chi$,\ for all $\delta\in \chi$ we have that $h\delta\cap \delta\in \chi$. Fix such a monochromatic ray $\delta$. There are thus distinct vertices $v$ and $v'$ on $\delta$ such that $h.v=v'$. Fix such a $v$. Take $w\in \textup{inn}^1(v)$ not on the monochromatic ray. For $x,y\in L_w$,\ both $x$ and $y$ commute with $hxh^{-1}$,\ since $\textup{supp}(hxh^{-1})=h \textup{supp}(x)\subset h (\vec{T_{\chi}}\backslash T^w)=\vec{T_{\chi}}\backslash T^{hw}$ is disjoint from $\textup{supp}(x),\ \textup{supp}(y)\subset T^{w}$. This implies that $[[x,h],y]=[x,y]$. Since $H$ is normal,\ it is also the case that $[[x,h],y]\in H$. We conclude that $\overline{[L_w,L_w]}\leqslant H$.\\
\indent For the case $n=0$,\ there exist $v\in \delta$ such that $h.v=v$ and for some $w \in \cup_{n\in\mathbb{N}} \textup{inn}^n(v)$ we have that $h.w\neq w$. Fix such a $v$. As for the case $n\neq 0$ we have that $\overline{[L_w,\ L_w]}\leqslant H$.\\
\indent The group $E_X(G,\ U)$ acts transitively on $V\vec{T_{\chi}}$ via Lemma \ref{lemma: element t and G},\ so indeed $\overline{[L_u,L_u]}\leqslant H$ for all vertices $u\in V\vec{T_{\chi}}$. It now follows that $\overline{\langle [L_u,L_u];u\in V\vec{T_{\chi}}\rangle}=\overline{\langle[P_X(G,\ U),P_X(G,\ U)]\rangle}$ is the monolith of $E_X(G,\ U)$.\\
\indent Now assume $G$ is also topologically perfect. By the argument above,\ we have that if $v$ is a vertex in the monochromatic ray,\ $\overline{[L_v,\ L_v]}\leqslant H$. The group $G$ is topologically perfect,\ so the wreath product $G_n$ of $n$ copies of $(G,\ X)$ over $U$ is also topologically perfect. In view of Lemma \ref{lemma:Gn 101}, we conclude that $G_n$ can be embedded as a closed subgroup of $\overline{[L_v,\ L_v]}\leqslant H$ for all $n\in\mathbb{N}$. Hence $L_v\leqslant H$. Letting $t$ be the translation down the monochromatic ray,\ $t^nL_v t^{-n}\leqslant H$ for all $n\geqslant 0$. Applying again Lemma \ref{lemma:Gn 101},\ we have that $P_X(G,\ U)\leqslant H$.
\end{proof}

\begin{prop}\label{prop:gluing result remaining}
Let $(G,\ X)$ be a transitive permutation group. If $U\in\mathcal{U}(G)$ is such that $\langle\!\langle U\rangle\!\rangle_G=G$ then for every $V\in\mathcal{U}(P_X(G,\ U))$  we have
$$\langle\!\langle V\rangle\!\rangle_{P_X(G,U)}=P_X(G,U).$$
\end{prop}

\begin{proof}
Let $V$ an open, compact subgroup of $P_X(G,\ U)$. As $P_X(G,\ U)$ is open in $E_X(G,\ U)$ it follows that $V$ is open in $P_X(G,\ U)$ if,\ and only if,\ there exist a finite set $F\subset V\vec{T_{\chi}}$ such that $E_X(U,\ U)_{(F)}\leqslant V$. As for every $v\in V\vec{T_{\chi}}$ the action of $P_X(G,\ U)$ is transitive in $\cup_{n\in\mathbb{N}}\textup{inn}^n(v)$ it follows that $P_X(U,\ U)\leqslant \langle\!\langle E_X(U,\ U)_{(F)}\rangle\!\rangle\leqslant \langle\!\langle V\rangle\!\rangle_{P_X(G,\ U)}$. As $\langle\!\langle U\rangle\!\rangle_G= G$,\ the result follows.
\end{proof}

\subsection{Building a group with infinite decomposition rank}

In this subsection,\ we will show that under the condition $(G,\ X)$ is an elementary permutation group,\ then $E_X(G,\ U)$ is also elementary. We show the existence of groups with decomposition rank up to $\omega^2+1$.\\
\indent To simplify the notation,\ we say that $(G,\ X)$ is an \textbf{elementary permutation group}\index{elementary permutation group|ndx} if $(G,\ \ X)$ is a t.d.l.c.s.c.~permutation group and $G$ is elementary.

\begin{lemma}\label{lemma: bad bound}
Suppose that $(K,\ X)$ is an elementary permutation group,\ and $L$ is an elementary group. Let $U\in \mathcal{U}(L)$ be a fixed compact open subgroup. Then $L\wr_U (K,\ X)$ is elementary. \\
\indent If in addition $K$ and $L$ are compactly generated and $K$ is not residually discrete,\ then 
$$\xi(L\wr_U (K,\ X)) \geqslant  \xi(L)+1.$$
\end{lemma}

\begin{proof}
The fact that $L\wr_U (K,\ X)$ is elementary follows directly from Definition \ref{defi:wr product},\ Corollary \ref{coro:wr product is elementary},\ and Definition \ref{defi:elementary groups}.\\
\indent For the second claim,\ we want to show that there exists a compactly generated subgroup $G$ of $L\wr_U (K,\ X)$ such that $\textup{Res}(G)$ contains a copy of $L$. First,\ let us find a subset $Y$ of $X$ such that $G:=L\wr_U(K,\ Y)$ is compactly generated and $\textup{Res}(G)\cap K$ is non-trivial. For that,\  let $\{Y_i\}_{i\in I}$ be the partition of $X$ into its orbits. For each $i\in I$ let $\textup{Fix}_K(Y_i)=\bigcap_{y\in Y_i}K_{(y)}:=N_i\leqslant K$,\ which is a closed normal subgroup of $K$. It follows from the fact that the action of $K$ in $X$ is faithful that $\cap_{i\in I}N_i=\{1\}$. Hence,\ if for all $i\in I$,\ we have that $K/N_i$ is residually discrete,\ then $K$ is residually discrete,\ a contradiction to the assumption. Let $i_0\in I$ be such that $K/N_{i_0}$ is not residually discrete. Notice that $Y_{i_0}$ is infinite,\ because otherwise $K/N_{i_0}$ would be profinite,\ hence residually discrete. Denote $Y$ as our fixed $Y_{i_0}$.\\
\indent Denote $G:=L\wr_U(K,Y)$,\ $R:=\textup{Res}(G)$ and let $\pi:G\rightarrow G/R$ be the canonical projection. The quotient $G/R$ is residually discrete,\ so the restriction $\pi|_K:K\rightarrow G/R$ is not injective,\ that is,\ $K\cap R\neq N_{i_0}$. By the asumption on $Y$ and this fact it follows that $\exists k\in K\cap R\backslash \{1\}$ and $y\in Y$ with $ky\neq y$. Fix such a $k$ and $y$. For each $l \in L$,\ let $f_l\in \bigoplus (L,\ U)$ be the function which is defined by $f_l(y)=l$ and $f_l(x)=1$,\ for all $x\neq y$. Because $\bigoplus_Y (L,\ U)\cap R$ is normal,\ the element $f_l k f_l^{-1}k^{-1}$ is in it,\ and moreover,\\
$$f_{l} k f_{l}^{-1} k^{-1}(z)=\left\{\begin{array}{ll}
l & \text { if } z=y \\
l^{-1} & \text { if } z=k(y) \\
1 & \text { else. }
\end{array}\right. $$
Letting $\pi_y: \bigoplus_Y (L,\ U)\rightarrow L$ be the projection onto the $y$-th coordinate and using the fact $\pi_y(f_lkf_l^{-1}k^{-1})=l$ we get
$$\pi_y\left( \bigoplus_Y(L,\ U)\cap R \right)=L.$$
That is,\ $\pi_y$ is a surjection. By Proposition \ref{prop:quotient elementary 101},\ $\xi(R)\geqslant\xi(L)$. From the fact $L\wr_U(K,\ X)\geqslant G$ we have that $\xi(L\wr_U(K,\ X))\geqslant \xi(G)$ (Proposition \ref{prop:decomposition rank closed subgroup}) and,\ by Definition \ref{defi:decomposition rank},\ $\xi(G)=\xi(R)+1\geqslant \xi(L)+1$,\ proving the result.
\end{proof}

Knowing that given $(G,\ X)$ an elementary permutation group,\ then $G_n$,\ as defined in Lemma \ref{lemma:Gn 101},\ is elementary,\ we can use Lemma \ref{lemma:Gn 101} to prove $E_X(G,\ U)$ is elementary.

\begin{prop}\label{prop:grp is elementary}
Suppose that $(G,\ X)$ is an elementary permutation group and $U\in \mathcal{U}(G)$. Then $E_X(G,\ U)$ is an elementary group.
\end{prop}

\begin{proof}
As in Lemma \ref{lemma:Gn 101},\ colour the tree $\vec{T_{\chi}}$,\ fix $v\in V\vec{T_{\chi}}$ a vertex,\ and define $T_R$ and $T_0$. With the same notation of Lemma \ref{lemma:Gn 101},\ we see a copy of the wreath product of $n$ copies of $G$ over $U$,\ denoted $G_n$,\ as a closed subgroup of $E(G,\ U)_{(T_R)}$.\\
\indent Given the $v$ that determines $T_R$ and $T_0$,\ the compact open subgroup $W:=E_X(U,\ U)_{(v)}$ may be written as $W=W_{(T_R)}W_{(T_0)}$. Because the group $G_n$ acts on $T_0$,\ it is easy to see that $G_n$ centralizes $W_{(T_0)}$. For a similar reason,\ letting $B_n(v)$ be the $n$ ball around $v$,\ the group $G_n$ normalizes $L_n:=W_{(T_R\cup B_n(v))}$. Define $H_n:=\langle G_n,\ W_{(T_0)},\ L_n\rangle$. Because $W\leqslant H_n$ then $H_n$ is open in $P_X(G,\ U)$. Also,\ both $L_n,\ W_{(T_0)}$ are compact,\ hence so is $L_n W_{(T_0)}$. By construction,\ the subgroup $L_n W_{(T_0)}$ is normal in $H_n$,\ and $H_n/L_n W_{(T_0)}\cong G_n$. Then $H_n$ is an open elementary group of $P_X(G,\ U)$.\\
\indent As $G_n$ acts on $B_n(v)$, we have that $L_0\backslash L_n\subset G_n$. Hence, the collection $\{H_n\}_{n\in\mathbb{N}}$ is directed. So the union $\bigcup_{n\in\mathbb{N}}H_n$ is an open elementary subgroup of $P_X(G,\ U)$,\ by Definition \ref{defi:elementary groups}. Applying Lemma \ref{lemma:Gn 101} $(2)$,\ it follows that
$$\bigcup_{n\in\mathbb{N}}H_n=E_X(G,\ U)_{(T_R)}W_{(T_0)}.$$
\indent Fixing $t$ a translation toward $\delta$,\ the family of elementary groups $\{t^n E_X(G,\ U)_{(T_R)}W_{(T_0)}t^{-n}\}_{n\in\mathbb{N}}$ is directed. Another application of Lemma \ref{lemma:Gn 101} implies that
$$\bigcup_{n\in\mathbb{N}}t^n E_X(G,\ U)_{(T_R)}W_{(T_0)}t^{-n}=P_X(G,\ U)$$
and therefore,\ $P_X(G,\ U)$ is elementary. Since $E_X(G,\ U)$ is the extension of $P_X(G,\ U)$ by $\mathbb{Z}$,\ Definition \ref{defi:elementary groups} implies $E_X(G,\ U)$ is elementary.
\end{proof}

\begin{prop}\label{prop: bad bound rank}
Suppose that $(G,\ X)$ is a compactly generated transitive elementary group. Then $E_X(G,\ U)$ is elementary with 
$$\xi(E_X(G,\ U))\geqslant \xi(G)+\omega+2$$
for any non-trivial $U\in\mathcal{U}(G)$.
\end{prop}

\begin{proof}
Propositions \ref{prop: comp gener} and \ref{prop:grp is elementary} imply $E_X(G,\ U)$ is an elementary compactly generated group. Then $\xi(E_X(G,\ U))= \xi(\textup{Res}(E_X(G,\ U)))+1$,\ so it is enough to show that
$$\xi(\textup{Res}(E_X(G,\ U)))\geqslant \xi(G)+\omega+1.$$
\indent Fix $v\in V\vec{T_{\chi}}$ and,\ as in Lemma \ref{lemma:Gn 101},\ see $G_n$ as a closed subgroup of $E_X(G,\ U)$. \\
\indent To apply Lemma \ref{lemma: bad bound},\ we need to work with a group that is not residually discrete. As $X$ is infinite,\ it follows that even when $G=G_1$ is residually discrete,\ $G_2=G\wr_U (G,\ X)$ is not residually discrete,\ as the transitive action of $G$ on $X$ implies that every open normal subgroup of $G_2$ contains the compact,\ open subgroup $\prod_{x\in X}U$. Since $G$ acts transitively on $X$,\ the discrete residual of $G_2$ is non-trivial.\\
\indent Applying Lemma \ref{lemma: bad bound},\ we deduce $\xi(G_{2+n})\geqslant\xi(G)+n$ for each $n\geqslant 1$. These groups are all compactly generated,\ so $\xi(\textup{Res}(G_{2+(n+1)}))\geqslant\xi(G)+n$. On the other hand,\ $\textup{Res}(G_{2+(n+1)})\leqslant \textup{Res}(E_X(G,\ U))$ for all $n\in\mathbb{N}$. We thus infer that
$$\begin{aligned}
    \xi(\textup{Res}(E_X(G,\ U))) \geqslant \sup_{n\geqslant 1} \xi(\textup{Res}(G_{2+(n+1)})) \geqslant \sup_{n\geqslant 1} (\xi(G)+n)= \xi(G)+\omega.
\end{aligned}$$
By Definition \ref{defi:decomposition rank},\ the decomposition rank is always a successor ordinal,\ therefore $\xi(\textup{Res}(E_X(G,\ U)))\geqslant \xi(G)+\omega+1$,\ completing the proof.
\end{proof}

A straightforward application of Proposition \ref{prop: bad bound rank},\ the proof of Proposition \ref{prop:grp is elementary} and Lemma \ref{lemma:short exact sequence decomposition rank} gives us a way to build groups with rank $\omega+2$.

\begin{coro}\label{coro:group omega+2}
Suppose that $(G,\ X)$ is a compactly generated transitive elementary permutation group with finite rank. Then $E_X(G,\ U)$ is elementary with
$$\xi(E_X(G,\ U))=\omega+2$$
for any non-trivial $U\in\mathcal{U}(G)$.
\end{coro}

Notice that if $(G,\ X)$ is an elementary permutation group,\ so is $(E_X(G,\ U),\ V\vec{T}_{\chi})$. These results can  then be iterated to prove the existence of groups with decomposition rank at least $\omega n +2$,\ for $n\in\mathbb{N}$. \\
\indent Notice that Lemma \ref{lemma: bad bound} does not get close to the upper bound of how much the decomposition rank can increase under extensions (Lemma \ref{lemma:short exact sequence decomposition rank}). Hence Proposition \ref{prop: bad bound rank} does not optimize the rank of $E_X(G,\ U)$. In Section 5, our focus will be on making the rank of the wreath product reach the upper bound,\ under some conditions on $G$ and $U\in\mathcal{U}(G)$. We will then be able to build groups with bigger decomposition rank.

\section{Residual height and rank}

In this Section,\ we give an order for the set,  $\mathcal{K}(G)$, of compactly generated subgroups of a group $G$. This order is central to our construction of elementary groups with a given decomposition rank up to $\omega^\omega+1$.

\subsection{Ordering closed,\ compactly generated subgroups}

The following lemma follows from the definition of the discrete residual of a subgroup $H\leqslant G$ in relation to the group $G$.

\begin{lemma}\label{lemma:properties residue that will be reused}\label{lemma:residual rank and semidirect product}
Let $G$ be a t.d.l.c.s.c.~group,\ then the following are true:
\begin{enumerate}
    \item If $N\leqslant G$ and $K\leqslant G$ normalizes $N$ then $\textup{Res}_N(K)\leqslant \textup{Res}_G(K)$.
    \item For every $H\leqslant G$ and $K\leqslant G$ such that $K$ normalizes $H$,\ $\textup{Res}_H(K)\leqslant \textup{Res}(G)$.
    \item If $O\leqslant G$ is an open subgroup and $K\leqslant G$ such that $K$ normalizes $O$. Then $\textup{Res}_O(K)=\textup{Res}_G(K)$.
    \item For any $K\leqslant G$,\ $\textup{Res}_G(K)=\textup{Res}_G(K\backslash (K\cap Z(G)))$,\ where $Z(G)$ is the center of $G$.
    \item  If $G=N\rtimes K$ and $H$ is a subgroup of $K$,\ then $\textup{Res}_K(H)\leqslant \textup{Res}_G(N\rtimes H)$.
    \item Let $H$ be a t.d.l.c.s.c.~group,\ $\pi:G\rightarrow H$ a continuous surjection. If $K\in\mathcal{K}(H)$ is such that there exists $K'\in\mathcal{K}(G)$ with $\pi(K')\geqslant K$, then $\textup{Res}_H(K)\leqslant \pi(\textup{Res}_G(K'))$.
\end{enumerate}
\end{lemma}

\begin{proof}
$\textit{1.}$ Assume $O\leqslant G$ an open subgroup is such that $kO=Ok$ for every $k\in K$. Because $N$ normalizes $K$ it is the case that $k(O\cap N)=(O\cap N)k$. Hence
$$ \textup{Res}_N(K)   \leqslant  O\cap N  \leqslant  O .$$
Taking the intersection under all open subgroups of $G$ normalized by $K$ gives us the desired result
$$   \textup{Res}_N(K)   \leqslant  \bigcap_{\substack{O\leqslant_o G \\ kO=Ok,\  \forall k\in K}} O\cap N  \leqslant  \bigcap_{\substack{O\leqslant_o G \\ kO=Ok,\  \forall k\in K}}O  =  \textup{Res}_G(K). $$
\indent $\textit{2.}$ If $O$ is an open normal subgroup of $G$ then for every $k\in K$ we have that $kO=Ok$. As $H$ is normalized by $K$ it follows that $k(O\cap H)=(O\cap H)k$, therefore $\textup{Res}_H(K)\leqslant O\cap H\leqslant O$. Then, it follows
$$    \textup{Res}_H(K) \leqslant   \bigcap_{\substack{O\trianglelefteq_o G }} O\cap H\leqslant \textup{Res}(G). $$
\indent $\textit{3.}$ Let $O\leqslant G$ be an open subgroup and $K$ such that $K$ normalizes $O$. Assume $V\leqslant O$ is an open subgroup of $V$ such that for every $k\in K$,\ $kV=Vk$. Notice that $V$ is also open and normalized by $K$ in $G$. Therefore
$$\textup{Res}_G(K)\leqslant V.$$
By taking the intersection of all $V\leqslant O$ open subgroup of $O$ normalized by $K$ we see that:
$$\textup{Res}_G(K)\leqslant \bigcap_{{\substack{V\leqslant_o O \\ kV=Vk,\  \forall k\in K}}}V=\textup{Res}_O(K).$$
With \textit{1.} we get $\textup{Res}_G(K)=\textup{Res}_O(K)$,\ as desired.\\
\indent \textit{4.} If $g\in Z(G)$ then for every $O\leqslant G$ open subgroup we have $gO=Og$. Hence,\ an open subgroup $O\leqslant G$ is such that $kO=Ok$ for every $k\in K$ if,\ and only if,\ $k'O=Ok'$ for every $k'\in K\backslash (K\cap Z(G))$.\\
\indent \textit{5.} A special case of \textit{1.}\\
\indent \textit{6.} Let $K,\ K'$ be as in the statement. Let $O\leqslant G$ be an open subgroup such that for every $k\in K'$ we have $kO=Ok$. Because $\pi$ is surjective, it follows that it is an open map and that $\overline{k}\pi(O)=\pi(O)\overline{k}$ for every $k\in K$. Hence
$$\textup{Res}_H(K)\leqslant \pi(O).$$
Taking the intersection of all such $O$ implies that
$$\textup{Res}_H(K)\leqslant \pi(\textup{Res}_G(K').$$
\end{proof}

\begin{defi}\index{discrete residual order|ndx}
Let $G$ be a topological group. We define the \textbf{discrete residual order} on $\mathcal{K}(G)$, denoted by $K_1\preccurlyeq K_2$, if $K_1\leqslant \textup{Res}_G(K_2)$ or $K_1=K_2$. If $K_1\preccurlyeq K_2$ and $K_1\neq K_2$ we will denote it as $K_1\prec_G K_2$.
\end{defi}


\begin{defi}[Residual order hierarchy]\index{residual order hierarchy|ndx}\index{$\mathcal{K}_{\omega_1}$|ndx}
Let $G$ be a topological group. We define recursively the \textbf{residual order hierarchy of $G$} as
$$\mathcal{K}_\alpha(G)=\left\{\begin{array}{ll}
  \{ \{1_G\} \} & \mbox{if } \alpha=0,\ \\
   \{K\in\mathcal{K}(G); \ \textup{Res}_G(K)=\{1_G\}\} & \mbox{if }\alpha=1\\
   \{K\in \mathcal{K}(G); \text{ all } L\prec K\text{ satisfy } L\in\mathcal{K}_{\beta}(G) \text{ for some } \beta\leqslant\gamma\}  & \mbox{if } \alpha=\gamma,\ 0<\gamma<\omega_1.
\end{array}\right.$$
We define the complete residual order hierarchy of $G$ as the set
$$\mathcal{K}_{\omega_1}(G)=\bigcup_{\alpha<\omega_1}\mathcal{K}_{\alpha}(G).$$
\end{defi}

The set $\mathcal{K}_1(G)$ is defined individually so that $\mathcal{K}_{\alpha}\subset\mathcal{K}_{\beta}$,\ for $\alpha<\beta$. Otherwise the trivial subgroup,\ $\{1\}$,\ wouldn't be contained in $\mathcal{K}_1(G)$.

\begin{defi}\label{defi:residual height/rank}\index{residual height|ndx}\index{residual rank|ndx}\index{$\textphnc{y}$|ndx}\index{$\textphnc{f}$|ndx}
Let $G$ be a topological group and $K\in \mathcal{K}_{\omega_1}(G)$. We define the \textbf{residual height of $K$},\ denoted $\textphnc{y}_G(K)$,\ as
$$\textphnc{y}_G(K)=\min\{\alpha;\ \alpha<\omega_1 \text{ and } K\in K_{\alpha}(G)\}.$$
Given that $\mathcal{K}(G)=\mathcal{K}_{\omega_1}(G)$,\ we say that the \textbf{residual rank is well-defined in $G$} and define it as
$$\textphnc{f}(G)=\left(\sup\!^+_{K\in\mathcal{K}(G)}\textphnc{y}_G(K)+1\right).$$
\end{defi}

\begin{remark}\label{remark:less or equal height}
Notice that, by definition, if $K_1,\ K_2\in \mathcal{K}_{\omega_1}(G)$ are such that $\textup{Res}_G(K_1)\geqslant\textup{Res}_G(K_2)$, then $K_2\succ L$ implies $K_1\succ L$. By definition it is then the case that $\textphnc{y}_G(K_1)\geqslant \textphnc{y}_G(K_2)$. If $\textup{Res}_G(K_1)=\textup{Res}_G(K_2)$ then it is also the case $\textphnc{y}_G(K_1)= \textphnc{y}_G(K_2)$.
\end{remark}

\indent One possibility for a group $G$ to not have a well-defined residual rank is having a collection of compactly generated subgroups that can be ordered as follows:
$$\begin{tikzcd}
\{1\} & \ldots \arrow[l] & G_{n} \arrow[l] & \ldots \arrow[l] & G_2 \arrow[l] & G_1 \arrow[l] & G_0. \arrow[l]
\end{tikzcd}$$

As the set $\{G_{n}\}_{n\in\mathbb{N}}$ has no minimal element under the residual order, it is impossible to put such groups on the residual order hierarchy of $G$. We will see that this must always be the case.

\begin{lemma}
Let $G$ be a t.d.l.c.s.c. group and $K\in\mathcal{K}(G)$. Then $K\in\mathcal{K}_{\omega_1}(G)$ if, and only if, every chain of $\{L\in \mathcal{K}(G);\ L\preccurlyeq K\}$ can be embedded into a countable ordinal via an order-preserving map.
\end{lemma}

\begin{proof}
Assume $K\in\mathcal{K}_{\omega_1}(G)$. Let $\{L_{\gamma}\}_{\gamma\in \Gamma}$ be a chain of $\{L\in \mathcal{K}(G);\ L\preccurlyeq K\}$, that is $\Gamma$ is a totally ordered set and if $\gamma_1<\gamma_2$ then $L_{\gamma_1}\prec L_{\gamma_2}$. Define $\phi:\{L_{\gamma}\}_{\gamma\in \Gamma}\rightarrow \textphnc{y}_G(K)$ as $\phi(L_{\gamma})=\textphnc{y}_G(L_{\gamma})$. By the definition of residual height, it follows that $\phi$ is an order-preserving embedding.\\
\indent Now assume $K\notin\mathcal{K}_{\omega_1}(G)$. There exists $L\prec K$ such that $L\in\mathcal{K}(G)\backslash \mathcal{K}_{\omega_1}(G)$ as if this wasn't the case, we would have a contradiction (notice that as $G$ is second countable, it is not possible that for every $\alpha<\omega_1$ there exists $L\prec K$ such that $\textphnc{y}_G(L)\geqslant \alpha$). We can then create a chain $\{L_n\}_{n\in\mathbb{N}}\subset \{L\in \mathcal{K}(G);\ L\preccurlyeq K\}$ such that $L_n\prec L_{n+1}$ for every $n\in\mathbb{N}$. Hence, there is a chain that cannot be embedded into an ordinal.
\end{proof}

\begin{coro}\label{coro:when do we have no residual rank}
Let $G$ be a t.d.l.c.s.c. group. If $K\notin\mathcal{K}(G)$ then there exist a chain $\{L_n\}_{n\in\mathbb{N}}\subset \{L\in \mathcal{K}(G);\ L\preccurlyeq K\}$ such that $L_n\prec L_{n+1}$ for every $n\in\mathbb{N}$.
\end{coro}

 The next result follows directly from the definition of a group having a well-defined residual rank.

\begin{lemma}\label{lemma:auxiliary for height on nice cases}
Let $G$ be a t.d.l.c.s.c.~group with well-defined residual rank. Then given $K\in\mathcal{K}(G)$ the height of $G$ can alternatively defined either as
$$\textphnc{y}_G(K)=\left\{ \begin{array}{ll}
 0 & \mbox{if }K=\{1\}\\
   \sup_{K'\prec K} \textphnc{y}_G(K)+1 & \mbox{otherwise}
\end{array}  \right.$$
or
$$\textphnc{y}_G(K)=\left( \sup_{K'\in (\textup{Res}_G(K))}\textphnc{y}_G(K')\right)+1.$$
\end{lemma}
Our objective now is to show that the rank and height are well behaved in the class of elementary groups,\ that is,\ if $G$ is elementary then the residual rank is well-defined in $G$. For that reason we will use the equivalent definition for residual height on groups with well behaved residual rank in this section. 



\begin{prop}\label{prop:easier residual height}
Let $G$ be a t.d.l.c.s.c.~group. Then:
\begin{enumerate}
    \item If $O\leqslant G$ is open,\ $K\leqslant O$ and $K\in\mathcal{K}_{\omega_1}(O)$, then $K\in\mathcal{K}_{\omega_1}(G)$ and $\textphnc{y}_O(K)=\textphnc{y}_G(K)$.
    \item If $K\leqslant \textup{Res}(G)$ and $K\in\mathcal{K}_{\omega_1}(\textup{Res}(G))$, then $K\in\mathcal{K}_{\omega_1}(G)$ and $\textphnc{y}_{\textup{Res}(G)}(K)=\textphnc{y}_G(K)$.
\end{enumerate}
\end{prop}

\begin{proof}
$\textit{1.}$ We want to show that the set $\mathcal{K}_{\omega_1}(O)$ is a \textbf{lower set} of $\mathcal{K}_{\omega_1}(G)$, that is, if $K\in \mathcal{K}_{\omega_1}(O)$ and $K_1\prec K$ in $\mathcal{K}_{\omega_1}(G)$, then $K_1\in \mathcal{K}_{\omega_1}(O)$ and $K_1\prec K$ in $\mathcal{K}_{\omega_1}(O)$. \\
\indent Let $O\leqslant G$ an open subgroup and $K\leqslant O$ a compactly generated subgroup such that $K\in\mathcal{K}_{\omega_1}(G)$. By Lemma \ref{lemma:properties residue that will be reused} \textit{3.} it follows that $\textup{Res}_O(K)=\textup{Res}_G(K)$. Hence $K_1\prec K$ in $\mathcal{K}_{\omega_1}(G)$ if,\ and only if,\ $K_1\in \mathcal{K}_{\omega_1}(O)$ and $K_1\prec K$ in $\mathcal{K}_{\omega_1}(O)$. It is then the case that $\mathcal{K}_{\omega_1}(O)$ is a lower set of $\mathcal{K}_{\omega_1}(G)$. Therefore $K\in\mathcal{K}_{\omega_1}(G)$ and $\textphnc{y}_O(K)=\textphnc{y}_G(K)$.\\
\indent $\textit{2.}$ Similarly to $1.$, we want to show that $\mathcal{K}_{\omega_1}(\textup{Res}(G))$ is a lower set of $\mathcal{K}_{\omega_1}(G)$.\\
\indent Assume that $K\leqslant \textup{Res}(G)$. Lemma \ref{lemma:properties residue that will be reused} \textit{1.} implies $\textup{Res}_G(K)\geqslant \textup{Res}_{\textup{Res}(G)}(K)$. Because $\textup{Res}_G(K)$ normalizes $K$,\ \textit{1.} and \textit{3.} from Lemma \ref{lemma:properties residue that will be reused} implies $\textup{Res}_{\textup{Res}(G)}(K) \geqslant \textup{Res}_{\textup{Res}_G(K)}(K)$. Hence $\textup{Res}_G(K)\geqslant \textup{Res}_{\textup{Res}(G)}(K)\geqslant \textup{Res}_{\textup{Res}_G(K)}(K)$. Theorem \ref{thrm:colin result for res commuting} then implies $\textup{Res}_G(K)=\textup{Res}_{\textup{Res}(G)}(K)$. Therefore $\textup{Res}_G(K)=\textup{Res}_{\textup{Res}(G)}(K)$. 
Hence  $K_1\prec K$ in $\mathcal{K}_{\omega_1}(G)$ if,\ and only if,\ $K_1\in \textup{Res}(G)$ and $K_1\prec K$ in $\mathcal{K}_{\omega_1}(\textup{Res}_G(K))$. It is then the case that $\mathcal{K}_{\omega_1}(\textup{Res}(G))$ is a lower set of $\mathcal{K}_{\omega_1}(G)$. Therefore $K\in\mathcal{K}_{\omega_1}(G)$ and $\textphnc{y}_{\textup{Res}(G)}(K)=\textphnc{y}_G(K)$.
\end{proof}

\begin{prop}\label{prop:auxiliary for 4.9}
Let $G,\ H$ be t.d.l.c.s.c. groups and $\pi:G\rightarrow H$ a continuous surjective map. Then for every $K\in\mathcal{K}(H)$ there exists $K'\in\mathcal{K}(G)$ such that $\pi(K')=K$. If $K,\ L\in \mathcal{K}(H)$ are such that $K\succ L$ you can pick an $L'\in\mathcal{K}(G)$ such that $\pi(L')=L$ and $K'\succ L'$.
\end{prop}

\begin{proof}
Let $K\in\mathcal{K}(H)$, as in the statement. Let $\pi':\pi^{-1}(K)\rightarrow K$ the restriction of $\pi$ on $\pi^{-1}(K)$. Let $\{O_n\}_{n\in\mathbb{N}}$ a collection of compactly generated, open subgroups of $\pi^{-1}(K)$ such that $\cup_{n\in\mathbb{N}}O_n=\pi^{-1}(K)$. As $\pi'$ is an open map, there is $N\in\mathbb{N}$ such that $\pi'(O_N)=K$. Define $K':=O_N$, for such $N$. It is then the case that $K'\in\mathcal{K}(G)$ and $\pi(K')=K$.\\
\indent Now assume $L\in\mathcal{K}(H)$ is such that $K\succ L$. By Lemma \ref{lemma:properties residue that will be reused} $6.$ it is the case that $\textup{Res}_{H}(K)<\pi(\textup{Res}_G(K'))$. Hence $\pi(\pi^{-1}(L)\cap \textup{Res}_G(K'))=L$. Define $\pi'':\pi^{-1}(L)\cap \textup{Res}_G(K')\rightarrow L$ the restriction  of $\pi$ on $\pi^{-1}(L)\cap \textup{Res}_G(K')$. Let $\{O'_n\}_{n\in\mathbb{N}}$ a collection of compactly generated, open subgroups of $\pi^{-1}(L)\cap \textup{Res}_G(K')$ such that $\cup_{n\in\mathbb{N}}O'_n=\pi^{-1}(L)\cap \textup{Res}_G(K')$. As $\pi''$ is an open map, there is $N'\in\mathbb{N}$ such that $\pi'(O_{N'})=L$. Define $L':=O_{N'}$, for such $N'$. By definition of $L'$ we have that $L'\leqslant \textup{Res}_{G}(K')$ and $L'\in\mathcal{K}(G)$, therefore $K'\succ L'$.
\end{proof}

\begin{prop}\label{coro:height decreases on surject}
Let $G,\ H$ be t.d.l.c.s.c.~groups and $\pi:G\rightarrow H$ a continuous surjective map. If $K\in\mathcal{K}(H)$ and $K'\in \mathcal{K}_{\omega_1}(G)$ are such that $\pi(K')= K$, then $K\in\mathcal{K}_{\omega_1}(H)$ and $\textphnc{y}_G(K')\geqslant \textphnc{y}_H(K)$. More than that, if $G$ has well-defined residual rank, then $H$ has well-defined residual rank and $\textphnc{f}(G)\geqslant \textphnc{f}(H)$.
\end{prop}

\begin{proof}
We will prove the first statement by induction on $\textphnc{y}_G(K')$. It follows trivially for the case $\textphnc{y}_G(K')=0$, that is, $K'=\{1_G\}$ and, as $\pi$ is surjective, $K=\{1_H\}$.\\
\indent Assume that for every $\beta<\alpha+1$, if $K\in\mathcal{K}(H)$ is such that there exists $K'\in \mathcal{K}_{\omega_1}(G)$ with $\pi(K')=K$ and $\textphnc{y}_G(K')<\beta$, then $K\in \mathcal{K}_{\omega_1}(H)$ and $\textphnc{y}_H(K)\leqslant \beta$. \\
\indent Assume $K\in\mathcal{K}(H)$ and $K'\in\mathcal{K}_{\omega_1}(G)$ are such that $\textphnc{y}_G(K')=\alpha+1$ and $\pi(K')\geqslant\ K$. Such $K'$ always exists by Proposition \ref{prop:auxiliary for 4.9}. Let $L\in\mathcal{K}(H)$ be such that $K\succ L$. By Proposition \ref{prop:auxiliary for 4.9} there is $L'\in\mathcal{K}(G)$ such that $K'\succ L'$ and $\pi(L')=L$. As $\textphnc{y}_G(K')=\alpha+1$ we have that $\alpha\geqslant \textphnc{y}_G(L')$. By the induction hypothesis, it is then the case that $\alpha\geqslant \textphnc{y}_H(L)$. As this is the case for all $L\prec K$ we have that $\textphnc{y}_G(K')=\alpha+1\geqslant \textphnc{y}_H(K)$, as required.\\
\indent Assume now $G$ has a well-defined rank. By Proposition \ref{prop:auxiliary for 4.9} it is the case that for every $K\in\mathcal{K}(H)$ there is $K'\in\mathcal{K}(G)$ with $\pi(K')=K$. By the proof above it is also the case that $\textphnc{y}_G(K')\geqslant \textphnc{y}_H(K)$ for all such $K,\ K'$. Hence
$$\textphnc{f}(G)=\left(\sup_{K'\in\mathcal{K}(G)}\textphnc{y}_{G}(K')\right)+1 \geqslant \left(\sup_{K\in\mathcal{K}(H)}\textphnc{y}_{H}(K)\right)+1=\textphnc{f}(H).$$
\end{proof}

Now for the proof that for the class of elementary groups the decomposition rank is the same as the residual rank we will use steps similar to the construction of the decomposition rank (Definition \ref{defi:decomposition rank}) for our proof.

\begin{prop}\label{prop:rank 2 is the same in both ranks}
Let $G$ be a t.d.l.c.s.c.~group. Then $G$ is elementary with $\xi(G)=2$ if,\ and only if,\ $\textphnc{f}(G)=2$.
\end{prop}

\begin{proof}
Assume $\xi(G)=2$. Let $K\in \mathcal{K}(G)$ non-trivial and $U$ be a compact open subgroup of $G$. The subgroup $O=\langle U,K\rangle\leqslant G$ is a compactly generated open subgroup. Because $O$ is non-trivial then $\xi(O)=2$ and because $O$ is compactly generated then $\textup{Res}(O)=\{1\}$. Lemma \ref{lemma:properties residue that will be reused} implies $\textup{Res}_O(K)=\textup{Res}_G(K)\leqslant \textup{Res}(O)$. Then $\textup{Res}_G(K)=\{1\}$ and $\textphnc{y}_G(K)=1$. Then $\textphnc{f}(G)=2$.\\
\indent If $\textphnc{f}(G)=2$ then for every $K\in\mathcal{K}(G)$,\ $\textphnc{y}_G(K)=1$,\ that is,\ $\textup{Res}_G(K)=\{1\}$. This implies that for every $O\leqslant G$ compactly generated open subgroup of $G$,\ $\textup{Res}(O)=\textup{Res}_G(O)=\{1\}$. Hence $\xi(O)=2$. As this is the case for all open compactly generated subgroups of $G$ it then follows that $\xi(G)=2$.
\end{proof}

\begin{prop}\label{prop:simplified residual rank on compactly generated}
Let $G$ be a compactly generated t.d.l.c.s.c.~group. If $\textup{Res}(G)$ has a well-defined residual rank then $G$ has a well-defined residual rank and $\textphnc{f}(G)=\textphnc{y}_G(G)+1=\textphnc{f}(\textup{Res}(G))+1$.
\end{prop}

\begin{proof}
Let $\textphnc{f}(\textup{Res}(G))=\alpha+1$. Given $K\in\mathcal{K}(G)$, Lemma \ref{lemma:properties residue that will be reused} implies that $\textup{Res}(G)\geqslant \textup{Res}_G(K)$. Hence if $K,\ L\in\mathcal{K}(G)$ are such that $K\succ L$, then $G\succ L$,\ that is,\ $G$ is the maximum element of $\mathcal{K}(G)$.  Hence, if the height is well-defined for $G$ then it is well-defined for every $K\in\mathcal{K}(G)$ and $\textphnc{y}_G(G)\geqslant \textphnc{y}_G(K)$.\\
\indent Given $L\in\mathcal{K}(\textup{Res}(G))$,\ Lemma \ref{lemma:properties residue that will be reused} implies $\textphnc{y}_G(L)=\textphnc{y}_{\textup{Res}(G)}(L)$. By definition of the residual rank we have that $\alpha\geqslant\textphnc{y}_{\textup{Res}(G)}(L) =\textphnc{y}_G(L)$.  Hence
$$\alpha+1\geqslant \left(\sup_{L\in\mathcal{K}(\textup{Res}(G))}\textphnc{y}_G(L)\right)+1=\textphnc{y}_G(G).$$ As $\textphnc{f}(\textup{Res}(G))=\alpha+1$ it is the case that $\alpha+1=\textphnc{y}_G(G)$. The result then follows from the definitions of residual height and rank.
\end{proof}

\begin{prop}\label{coro:simplified residual rank general}
Let $G$ be a t.d.l.c.s.c.~group. If $G$ has $\{O_n\}_{n\in\mathbb{N}}$ an increasing collection of compactly generated subgroups with well-defined residual rank such that $G=\cup_{n\in\mathbb{N}}O_n$,\ then $G$ has a well-defined residual rank and $\textphnc{f}(G)=\sup^+_{n\in\mathbb{N}}(\textphnc{f}_G(O_n))=\sup_{n\in\mathbb{N}}(\textphnc{f}(\textup{Res}(O_n)))+1$.
\end{prop}

\begin{proof}
Let $H$ a compactly generated subgroup of $G$ and $H=\langle K\rangle$,\ with $K$ compact. Because $\{O_n\}$ covers $K$,\ there exist $N\in\mathbb{N}$ such that for every $n\geqslant N$,\ $K\subset O_n$. Hence $H\leqslant O_n$. By Proposition \ref{prop:easier residual height} $\textphnc{y}_{O_n}(H)=\textphnc{y}_{G}(H)$. Lemma \ref{lemma:properties residue that will be reused} also gives us that $\textup{Res}(O_n)=\textup{Res}_{G}(O_n)\geqslant \textup{Res}_{O_n}(H)=\textup{Res}_G(H)$. Remark \ref{remark:less or equal height} then implies $\textphnc{y}_G(O_n)\geqslant \textphnc{y}_{G}(H)$. Hence the residual rank of $G$ is well-defined.  Gathering all these points it then follows that $\sup_{H\in\mathcal{K}(G)}\textphnc{y}_G(H)=\sup_{n\in\mathbb{N}}\textphnc{y}_G(O_n)$. Proposition \ref{prop:simplified residual rank on compactly generated} and the definition of residual rank then gives us $\textphnc{f}(G)=\sup^+_{n\in\mathbb{N}}(\textphnc{y}_G(O_n)+1)=\sup_{n\in\mathbb{N}}(\textphnc{f}(\textup{Res}(O_n)))+1$.
\end{proof}

\begin{theorem}\label{thrm:for elementary groups ranks are the same}
If $G$ is an elementary group then the residual rank of $G$ is well-defined and $\xi(G)=\textphnc{f}(G)$.
\end{theorem}

\begin{proof}
If $\xi(G)=1$ or $\textphnc{f}(G)=1$ then $G=\{1\}$ and the result follows. By Proposition \ref{prop:rank 2 is the same in both ranks} the result is also true when $\xi(G)=2$.\\
\indent Assume that for every $\beta<\alpha$,\ if $G$ is elementary with $\xi(G)=\beta+1$ then $\textphnc{f}(G)=\beta+1$. Let $G$ be an elementary group such that $\xi(G)=\alpha+1$. The induction step is proven by working with two cases, looking at when $G$ is compactly generated and when $G$ is not.\\
\indent Assume $G$ is compactly generated. Then $\xi(G)=\xi(\textup{Res}(G))+1$ and,\ by the induction hypothesis,\ the residual rank of $\textup{Res}(G)$ is well-defined and $\xi(\textup{Res}(G))=\textphnc{f}(\textup{Res}(G))$. Since, by Proposition \ref{prop:simplified residual rank on compactly generated} $\textphnc{f}(G)=\textphnc{f}(\textup{Res}(G))+1$ the result follows for the compactly generated case.\\
\indent On the other hand,\ let $G$ not compactly generated. Let $\{O_n\}_{n\in\mathbb{N}}$ be an increasing collection of open compactly generated subgroups such that $G=\cup_{n\in\mathbb{N}}O_n$. Then,\ for every $n\in\mathbb{N}$,\ $\xi(O_n)\leqslant \alpha+1$. If there is $N\in\mathbb{N}$ such that for every $n\geqslant N$,\ $\xi(O_n)=\alpha+1$,\ the argument for compactly generated groups of rank $\alpha+1$ implies that $\xi(O_n)=\textphnc{f}(O_n)$. On the other hand,\ if for every $n\in\mathbb{N}$,\ $\xi(O_n)<\alpha+1$,\ because they are all compactly generated,\ the induction hypothesis implies $O_n$ has a well-defined residual rank and $\textphnc{f}(O_n)=\xi(O_n)$. In either case,\ Proposition \ref{coro:simplified residual rank general} implies $\textphnc{f}(G)=\sup_{n\in\mathbb{N}}(\textphnc{f}(\textup{Res}(O_n)))+1=\sup_{n\in\mathbb{N}}(\xi(\textup{Res}(O_n)))+1=\xi(G)=\alpha+1$.
\end{proof}

\indent Now that we know that on elementary groups the decomposition rank is the same as the residual rank we can use this property to have a better notion of the behaviour of the residual height on the class of elementary groups.

\begin{prop}\label{prop:finding comp gen open with same residual rank as K}
Let $G$ be a t.d.l.c.s.c.~group. If $K\in \mathcal{K}(G)$,\ there is $E$, an open compactly generated subgroup of $G$, such that $K\leqslant E$,\ $\textphnc{f}(E)=\textphnc{y}_G(E)+1=\textphnc{y}_G(K)+1$ and $\textup{Res}(E)=\textup{Res}_G(K)$.
\end{prop}

\begin{proof}
Let $E$ be an open,\ compactly generated subgroup such that $\textup{Res}_G(K)=\textup{Res}(E)$,\ as given in Theorem \ref{thrm:colin theorem G}. Because $E$ is open then,\ by Proposition \ref{prop:easier residual height},\ $\textphnc{y}_G(E)=\textphnc{y}_E(E)\geqslant \textphnc{y}_G(K)=\textphnc{y}_E(K)$ but,\ because $\textup{Res}_G(K)=\textup{Res}(E)$,\ Remark \ref{remark:less or equal height} implies $\textphnc{y}_E(K)= \textphnc{y}_E(E)$. Hence the result follows.
\end{proof}


\begin{coro}\label{coro:height well behaved on elementary}
Let $G$ be an elementary group. Then for every $K\in\mathcal{K}(G)$ we have $$\xi(K)\leqslant \textphnc{y}_G(K)+1.$$
\end{coro}

\begin{proof}
If $G$ is elementary then by Proposition \ref{prop:decomposition rank closed subgroup},\ $K$ is elementary. Because $K$ is elementary then $\textphnc{f}(K)=\textphnc{y}_K(K)+1=\xi(K)$. Let $E$ be an open, compactly generated subgroup of $G$ such that $\textphnc{y}_G(E)=\textphnc{y}_G(K)$,\ $K\leqslant E$ and $\textup{Res}(E)=\textup{Res}_G(K)=\textup{Res}_E(K)\geqslant \textup{Res}(K)$ (last equality and the inequality follows from Lemma \ref{lemma:properties residue that will be reused}). Such a subgroup exists by Proposition \ref{prop:finding comp gen open with same residual rank as K}. Because $E$ is open,\ Proposition \ref{prop:easier residual height} implies that $\textphnc{y}_E(E)=\textphnc{y}_G(E)$ and $\textphnc{y}_E(K)=\textphnc{y}_G(K)$. By Remark \ref{remark:less or equal height} it is then the case $\xi(K)= \textphnc{y}_K(K)+1\leqslant\textphnc{y}_E(K)+1=\textphnc{y}_G(K)+1$.
\end{proof}
\indent  Notice that if $\xi(E)=\xi(K)$ then Theorem \ref{thrm:for elementary groups ranks are the same} implies that $\xi(K)=\textphnc{y}_G(K)+1$. \\
\indent These results then show us that in the universe of elementary groups, the residual height and rank always behave well. 

\begin{lemma}\label{lemma:residual rank closed under subgroups}
Let $G$ be a t.d.l.c.s.c.~group and $H\leqslant G$. If the residual rank of $G$ is well-defined then so is the residual rank of $H$,\ $\textphnc{y}_G(K)\geqslant\textphnc{y}_H(K)$ and $\textphnc{f}(G)\geqslant \textphnc{f}(H)$.
\end{lemma}

\begin{proof}
Let $K\in \mathcal{K}(H)$. Then $K\in\mathcal{K}(G)$ and $\textup{Res}_G(K)\geqslant \textup{Res}_H(K)$. Hence if $K\succ_H L$ then $K\succ_G L$. An induction on the height of $K$ as in Corollary \ref{coro:height decreases on surject} implies $\textphnc{y}_G(K)\geqslant\textphnc{y}_H(K)$. It then follows by definition that the residual rank of $H$ is well-defined and $\textphnc{f}(G)\geqslant \textphnc{f}(H)$.
\end{proof}

\begin{lemma}\label{lemma:residual rank closed extensions}
Let $G$ be a t.d.l.c.s.c.~group. If $N$ is a closed, normal subgroup of $G$ such that $N$ has well-defined residual rank, then for all $K\in \mathcal{K}_{\omega_1}(G)$,
$$\textphnc{y}_{G/N}(\overline{\pi(K)})\leqslant\textphnc{y}_G(K)\leqslant \textphnc{f}(N)-1+\textphnc{y}_{G/N}(\overline{\pi(K)}),$$
where $\pi:G\rightarrow G/N$ is the surjection of $G$ under $N$. If $G$ also has well-defined residual rank, we then have
$$\textphnc{f}(G)\leqslant \textphnc{f}(N)-1+\textphnc{f}(G/N).$$
\end{lemma}

\begin{proof}
The proof of the first statement will be done on induction first by showing that if $\textphnc{y}_{G}(K)$ is well defined, then so is $\textphnc{y}_{G/N}(\overline{\pi(K)})$. We then make an induction over $\textphnc{y}_{G}(K)$ to show the right inequality.\\
\indent Let $E\leqslant G$ be a compactly generated subgroup such that $\textup{Res}_G(K)=\textup{Res}_G(E)=\textup{Res}(E)$. Such a group exists by Proposition \ref{prop:finding comp gen open with same residual rank as K}. Notice that $\pi(E)$ is an open subgroup of $G/N$, hence Proposition \ref{coro:height decreases on surject} implies $\textphnc{y}_G(E)=\textphnc{y}_G(K)\geqslant \textphnc{y}_{G/N}(\pi(E))$. Notice that $\pi(E)\geqslant \pi(K)$, hence as $\pi(E)$ is closed we have $\pi(E)\geqslant \overline{\pi(K)}$. Therefore $\textphnc{y}_G(K)\geqslant\textphnc{y}_{G/N}(\pi(E))\geqslant \textphnc{y}_{G/N}(\overline{\pi(K)})$, that is, the height of $\overline{\pi(K)}$ is well defined and the left inequality follows.\\
\indent To show the right side of the inequality, our base cases are when $K\in\mathcal{K}_{\omega_1}(G)$ are such that $\textphnc{y}_G(K)\leqslant \textphnc{f}(N)-1$. For these cases, the right side of the inequality follows trivially.\\
\indent Assume now that for every $\beta<\alpha$, if $K\in\mathcal{K}_{\omega_1}(G)$ satisfies $\textphnc{y}_G(K)\leqslant\textphnc{f}(N)-1+\beta$ then $\textphnc{y}_{G/N}(\overline{\pi(K)})\geqslant \beta$.\\
\indent Let $K\in\mathcal{K}_{\omega_1}(G)$ be such that $\textphnc{y}_G(K)=\textphnc{f}(N)-1+\alpha$. Let $E$ a compactly generated, open subgroup of $G$ such that $\textup{Res}(E)=\textup{Res}_{G}(K)=\textup{Res}_{G}(E)$. Let $L\in\mathcal{K}(H)$ such that $L\prec \overline{\pi(K)}$. By Proposition \ref{coro:height decreases on surject}, we can find $L'\prec E$ such that $\pi(L')=L$. As $\textup{Res}_{G}(K)=\textup{Res}_{G}(E)$ it is also the case $L'\prec K$. We then have $\textphnc{y}_G(L')<\textphnc{f}(N)-1+\alpha$ hence, by the induction hypothesis,  $\textphnc{y}_G(L')<\textphnc{f}(N)-1+\textphnc{y}_{G/N}(\pi(L))$ and $\textphnc{y}_{G/N}(L)\geqslant \alpha-1$. As this is the case for all $L\prec \overline{\pi(K)}$, it follows that $\textphnc{y}_{G/N}(\overline{\pi(K)})\geqslant \alpha$, proving the induction hypothesis.\\
\indent Now assume that $G$ has well-defined residual rank. By Proposition \ref{coro:height decreases on surject},  $G/N$ has well-defined residual. By the first statement and the definition of the residual rank, it is then the case that for every $K\in\mathcal{K}(G)$, we have that 
$$\textphnc{y}_G(K)<\textphnc{f}(N)-1+\textphnc{f}(G/N).$$
Again, by the definition of residual rank, we obtain the following
$$\textphnc{f}(G)=\left(\sup_{K\in\mathcal{K}(G)}\!^+\textphnc{y}_G(K)+1\right)\leqslant \textphnc{f}(N)-1+\textphnc{f}(G/N).$$
\end{proof}

It would be ideal to show that if $N$ and $G/N$ have well-defined residual rank, then so does $G$. An initial idea would be given $K\in\mathcal{K}(G)$, proving that $K$ has well-defined residual height by induction on the height of $\overline{\pi(K)}$. The problem with that is: we don't know if given a chain $\{L_n\}_{n\in\mathbb{N}}$ as in Corollary \ref{coro:when do we have no residual rank}, if it is possible that $\textphnc{y}_{G/N}(\overline{\pi(L_{m_1})})=\textphnc{y}_{G/N}(\overline{\pi(L_{m_2})})$, for all $m_1,\ m_2\in \mathbb{N}$. If a group with such a chain exists, the result wouldn't be true in general.

\subsection{Residual rank and non-elementary groups}
In this Section we focus on showing not all t.d.l.c.s.c.~groups have well-defined residual rank and that there are groups with well-defined residual rank that are not elementary. The following result gives us conditions for a group to have such an infinite descending chain, that is, it satisfies the condition for Corollary \ref{coro:when do we have no residual rank}.

\begin{prop}
Let $G$ be a t.d.l.c.s.c.,\ non-discrete,\ compactly generated,\ simple group. If there exists $H$ a closed proper subgroup of $G$ such that $G\cong H$,\ then the residual rank of $G$ is not well-defined.
\end{prop}

\begin{proof}
 As $G$ is simple and non-discrete, it follows that $\textup{Res}(G)=G$. Define recursively $\{G_n\}_{n\in\mathbb{N}}$ a collection of compactly generated subgroups of $G$ as follows: $G_0=G$ and $G_{n+1}$ is a proper,\ closed subgroup of $G_n$ and isomorphic to $G$. As $G_n$ normalizes itself,\ it follows that $G_{n+1}<G_{n}=\textup{Res}(G_{n})\leqslant \textup{Res}_{G}(G_{n})$. Hence $G_{n+1}\prec G_n$ for all $n\in\mathbb{N}$. Such an infinitely descending chain implies the rank is not well-defined by Corollary \ref{coro:when do we have no residual rank}.
\end{proof}

The Neretin's groups of sphereomorphisms $N_p$ \cite{Neretin} always satisfy this property,\ giving a countable collection of example of groups without a well-defined residual rank.\\
\indent Another question that arises naturally is showing that the class of groups with well-defined residual rank contains non-elementary groups. For the following proof we define $\mathcal{T}_{d+1}$ the infinite $(d+1)$-regular tree. The group $\textup{Aut}(\mathcal{T}_{d+1})$\index{$\textup{Aut}(\mathcal{T}_{d+1})$|ndx} is defined to be the group of automorphisms of $\mathcal{T}_{d+1}$ with a basis of compact,\ open subgroups generated by the set $\{\textup{Fix}(F);\ F\subset V\mathcal{T}_{d+1}$ is a finite subset$\}$. This group has the simple,\ non-discrete,\ compactly generated,\ open subgroup $\textup{Aut}^+(\mathcal{T}_{d+1})$  generated by the edge fixing automorphisms \cite{Tits1970}. Hence,\ for $d>1$ the groups are $\textup{Aut}^+(\mathcal{T}_{d+1})$ and $\textup{Aut}(\mathcal{T}_{d+1})$ are not elementary.

\begin{prop}
Given $d>1$ a natural number,\ the groups $\textup{Aut}^+(\mathcal{T}_{d+1})$ and $\textup{Aut}(\mathcal{T}_{d+1})$ have well-defined residual rank and $\textphnc{f}(\textup{Aut}(\mathcal{T}_{d+1}))=\textphnc{f}(\textup{Aut}^+(\mathcal{T}_{d+1}))=3$.
\end{prop}

\begin{proof}
Let $G<\textup{Aut}^+(\mathcal{T}_{d+1})$ a proper,\ compactly generated,\ open subgroup. By \cite[Theorem 11]{AutTd} it follows that if $G$ is a proper,\ open subgroup then it is compact. Proposition \ref{prop:easier residual height} then implies that $\textphnc{y}_G(G)=\textphnc{y}_{\textup{Aut}^+(\mathcal{T}_{d+1})}(G)=1$. With Proposition \ref{prop:finding comp gen open with same residual rank as K} it follows that this is also the case for all proper,\ compactly generated,\ closed subgroups of $\textup{Aut}^+(\mathcal{T}_{d+1})$. Hence $\textphnc{y}_{\textup{Aut}^+(\mathcal{T}_{d+1})}(\textup{Aut}^+(\mathcal{T}_{d+1}))=2$ and $\textphnc{f}(\textup{Aut}^+(\mathcal{T}_{d+1}))=3$.\\
\indent For the case $\textup{Aut}(\mathcal{T}_{d+1})$,\ notice that $\textup{Res}(\textup{Aut}(\mathcal{T}_{d+1}))=\textup{Aut}^+(\mathcal{T}_{d+1})$,\ as $\textup{Aut}^+(\mathcal{T}_{d+1})$ is the only proper,\ normal,\ open subgroup of $\textup{Aut}(\mathcal{T}_{d+1})$. Hence the result follows.
\end{proof}
For the group $\textup{Aut}(\mathcal{T}_{d+1})$ we then have the following residual order hierarchy:
$$\mathcal{K}_0(\textup{Aut}(\mathcal{T}_{d+1}))=\{\{1_{\textup{Aut}(\mathcal{T}_{d+1})}\}\}$$
$$\mathcal{K}_1(\textup{Aut}(\mathcal{T}_{d+1}))=\{K\leqslant \textup{Aut}^+(\mathcal{T}_{d+1});\ K \text{ is a compact subgroup}\}$$
$$\mathcal{K}_2(\textup{Aut}(\mathcal{T}_{d+1}))=\mathcal{K}(\textup{Aut}(\mathcal{T}_{d+1})).$$

\indent The next result shows that, for non-elementary groups, the residual rank behaves quite differently from the decomposition rank.

\begin{lemma}
If $G$ is a compactly generated group with well-defined residual rank such that $\textup{Res}(G)=G$ then $\textphnc{f}(G\times G)=\textphnc{f}(G)-1+\textphnc{f}(G)$.
\end{lemma}

\begin{proof}
First, we will prove by contradiction that $G\times G$ has well-defined residual rank. We will then show that the rank of $G\times G$ is as given in the statement.\\
\indent Assume $G\times G$ doesn't have well defined rank. Then by Corollary \ref{coro:when do we have no residual rank}, there exists $\{L_n\}_{n\in\mathbb{N}}$ a collection such that for all $n\in\mathbb{N}$, $L_{n+1}\prec L_n$. Let $\pi_i:G\times G\rightarrow G$ the surjection into the $i$-th coordinate. Then either $\{\pi_1(L_n)\}_{n\in\mathbb{N}}$ or $\{\pi_2(L_n)\}_{n\in\mathbb{N}}$ is an infinite collection of non-trivial groups. In either cases, as $\textup{Res}_{G\times G}(K\times L)=\textup{Res}_G(K)\times \textup{Res}_G(L)$ for every $K,\ L\in\mathcal{K}(G)$, it follows that $G$ has an infinite descending chain, a contradiction to the assumption $G$ has well-defined residual rank.\\
\indent Let $K,\ L \in \mathcal{K}(G)$ be such that $K\succ L$. Then the assumption that $\textup{Res}(G)=G$ implies $K\times G \succ L\times G\succ \{1\}\times G$. By Remark \ref{remark:less or equal height} it is then the case:
$$\textphnc{y}_{G\times G}(G\times G)\geqslant \textphnc{y}_G(G)+\textphnc{y}_G(K)$$
for every $K\in\mathcal{K}(G)$. Taking the supremum gives us:
$$\textphnc{y}_{G\times G}(G\times G)\geqslant \textphnc{y}_G(G)+\textphnc{y}_G(G).$$

By Lemma \ref{lemma:residual rank closed extensions} it is also the case $\textphnc{y}_{G\times G}(G\times G)\leqslant \textphnc{y}_G(G)+\textphnc{y}_G(G)$, hence 
$$\textphnc{y}_{G\times G}(G\times G)= \textphnc{y}_G(G)+\textphnc{y}_G(G).$$ 
As $G\times G$ is compactly generated, Proposition \ref{prop:simplified residual rank on compactly generated} gives us the desired result
$$\textphnc{f}(G\times G)=\textphnc{y}_{G\times G}(G\times G)+1=\textphnc{y}_G(G)+\textphnc{y}_G(G)+1=\textphnc{f}(G)-1+\textphnc{f}(G).$$
\end{proof}

As an example of the result above, let $G=\textup{Aut}(\mathcal{T}_{d+1})$, $U=\textup{Aut}(\mathcal{T}_{d+1})_{(v)}$ for $v$ a vertex of $V\mathcal{T}_{d+1}$. One can show that $G\times G$ has height $4$ as follows:
$$
G\times G \succ G\times U \succ G\times \{1_G\} \succ U\times \{1_G\} \succ \{1_G\}\times \{1_G\}
$$
 Applying this lemma with Proposition \ref{coro:simplified residual rank general} gives us the following:

\begin{coro}
Let $G$ be a compactly generated group with well-defined residual rank such that $\textup{Res}(G)=G$,\ and let $U\in\mathcal{U}(G)$. Then $\textphnc{f}(\bigoplus_{n\in\mathbb{N}}(G,U))=\textphnc{f}(G)\omega +1$.
\end{coro}

We then have the following properties on the class of groups with well-defined residual rank:

\begin{lemma}
The class of t.d.l.s.c.~groups with well-defined residual rank is closed under:
\begin{itemize}
    \item Closed subgroups (Lemma \ref{lemma:residual rank closed under subgroups})
    \item Hausdorff quotients (Corollary \ref{coro:height decreases on surject})
    \item Increasing countable unions (Proposition \ref{coro:simplified residual rank general})
\end{itemize}    
\end{lemma}

\section{Last comments}

\subsection{Intermediate steps}

The following definition will be essential to build groups with the desired rank,\ as it will allow us to work with the elementary/residual rank of a group in a more precise way.

\begin{defi}\index{collection witnessing the rank of $G$|ndx}
Let $G$ be a non-trivial group with well-defined residual rank. We say that a collection $\{G_{[\beta+1]}\}_{\beta<\alpha}$ of closed compactly generated subgroups of $G$ \textbf{witness that the rank of $G$ is at least $\alpha+1$} if $G_{[1]}\neq \{1\}$,\ for every $\beta<\alpha$ we have $\textphnc{y}_G(G_{[\beta+1]})=\beta+1$,\ and given $\beta<\beta'$ then  $G_{[\beta+1]}\prec G_{[\beta'+1]}$. If $\textphnc{f}(G)=\alpha+1$, we say that this collection witnesses the rank of $G$.
\end{defi}

A group $G$ with rank $\omega+2$ that has a collection of groups witnessing the rank will have a collection of compactly generated subgroups that can be ordered in the following way:

$$\begin{tikzcd}
G_{[1]} & G_{[2]} \arrow[l] & G_{[3]} \arrow[l] & \ldots \arrow[l] & G_{[n]}\arrow[l] & \ldots \arrow[l] & G_{[\omega+1]} \arrow[l]
\end{tikzcd}$$

\noindent and a group $G$ with rank $\omega+2$ that has no collection of groups witnessing the rank might have a collection of compactly generated subgroups that can be ordered in the following way:

$$\begin{tikzcd}
          &                     & G_{[\omega+1]} \arrow[lld] \arrow[ld] \arrow[d] \arrow[rd] \arrow[rrd] &                     &        \\
G_{[1]}^{(1)} & G_{[2]}^{(2)} \arrow[d] & G_{[3]}^{(3)} \arrow[d]                                                  & G_{[4]}^{(4)} \arrow[d] & \ldots \\
          & G_{[1]}^{(2)}           & G_{[2]}^{(3)} \arrow[d]                                                  & G_{[3]}^{(4)} \arrow[d] &        \\
          &                     & G_{[1]}^{(3)}                                                            & G_{[2]}^{(4)} \arrow[d] &        \\
          &                     &                                                                      & G_{[1]}^{(4)}           &       
\end{tikzcd}$$

\noindent without any subset that can be ordered as the first example. It then might be possible to have a group with residual rank $\omega+2$ but no collection witnessing its rank. It is still not known if a group satisfying such a property does exist.

\begin{theorem}\label{thrm:making big ranks better with residue order}\label{coro:condition for the rank to be additive}
Suppose $L$ and $K$ are non-trivial compactly generated elementary groups and $K$ acts on a countable set $X$ with compact open point stabilizers (the action does not need to be faithful). If $U\in\mathcal{U}(L)$ is such that $\langle\!\langle U\rangle\!\rangle_L=L$ and there exist $\{K_{[\alpha+1]}\}_{\alpha<\textphnc{y}_K(K)}$ a collection witnessing the rank of $K$ such that $K_{[1]}$ has an infinite orbit $Y_1\subset X$ then
$$\textphnc{f}(L\wr_U(K,\ X))=\textphnc{f}(L)-1+\textphnc{f}(K).$$
If $\textphnc{f}(K)>\omega$ then the condition of $K_{[1]}$ having an infinite orbit $Y_1\subset X$ is not necessary.
\end{theorem}

\begin{proof}
Let $\{K_{[\alpha+1]}\}_{\alpha<\textphnc{y}_K(K)}$ be a collection witnessing the rank of $K$.\\
\indent For the case that there is $Y_1\subset X$, fix such an $Y_1$ an infinite orbit of $K_{[1]}$ and define, for each $\alpha<\textphnc{y}_K(K)$, the subsets $Y_{\alpha+1}:=K_{[\alpha+1]}Y_{1}$.\\
\indent If $K_{[1]}$ has no infinite orbit, then, as seen in the proof of Proposition \ref{prop: bad bound rank},\ even though $K_{[1]}$ doesn't have an infinite orbit in $X$,\ as $K_{[2]}$ is not residually discrete it will have at least one infinite orbit in $X$,\ say $Y_2$. Fix $Y_1$ a finite orbit of $K_{[1]}$ and $Y_2$ an infinite orbit of $K_{[2]}$ and define, for $1<\alpha<\textphnc{y}_K(K)$, the sets $Y_{\alpha+1}=K_{[\alpha+1]}Y_{2}$.\\
\indent We define the subgroups $G_{\alpha+1}:=L\wr_U(K_{[\alpha+1]},Y_{\alpha+1})$,\ $J_{\alpha+1}:=\bigoplus_{y\in Y_{\alpha+1}}(L,\ U)$,\ and $G:=L\wr_U(K,\ X)$. Notice that because the actions of $K_{[\alpha+1]}$ on $Y_{\alpha+1}$ are transitive,\ the groups $G_{\alpha+1}$ are compactly generated.\\
\indent Claim 1: If $Y_{\alpha+1}$ is infinite then $J_{\alpha+1}\leqslant \textup{Res}_G(G_{\alpha+1})$ and given $\beta<\alpha$,\ $G_{\beta+1}\leqslant \textup{Res}_G(G_{\alpha+1})$.\\
\indent Let $O\leqslant G$ be an open subgroup such that for every $g\in G_{\alpha+1}$ we have $gO=Og$. Then as $J_{\alpha+1}\leqslant G_{\alpha+1}$,\ $O\cap J_{\alpha+1}$ is normal in $J_{\alpha+1}$. Since $J_{\alpha+1}$ 
contains the product over $Y_{\alpha+1}$ of copies of $U$, and $O$ is open, then  $U^{Y_{\alpha+1}\backslash I}\leqslant O$ for some finite subset $I\subset Y_{\alpha+1}$. Since the group $K_{[\alpha+1]}$ acts transitively in $Y_{\alpha+1}$,\ we get that $U^{Y_{\alpha+1}}\leqslant O$. Because $O\cap J_{\alpha+1}$ is normal on $J_{\alpha+1}$ and $\langle\!\langle U\rangle\!\rangle_L=L$ we get that $J_{\alpha+1}\leqslant O$. Hence $J_{\alpha+1}\leqslant \textup{Res}_G(G_{\alpha+1})$. By Lemma \ref{lemma:residual rank and semidirect product} the subgroup $\textup{Res}_G(G_{\alpha+1})$ must also contain $\textup{Res}_K(K_{[\alpha+1]})$,\ so
$$L\wr_U(\textup{Res}_K(K_{[\alpha+1]}),Y_{\alpha+1})=J_{\alpha+1}\rtimes \textup{Res}_K(K_{[\alpha+1]}) \leqslant \textup{Res}_G(G_{\alpha+1}).$$
The assumption $\{K_{[\alpha+1]}\}_{\alpha<\textphnc{y}_K(K)}$ is a collection witnessing the rank of $K$ and the fact that for every $\alpha<\beta$,\ $J_{\alpha+1}\leqslant J_{\beta+1}$,\ implies that for $\beta<\alpha$,
$$G_{\beta+1}=J_{\beta+1}\rtimes K_{[\beta+1]}\leqslant J_{\alpha+1}\rtimes \textup{Res}_K(K_{[\alpha+1]})\leqslant \textup{Res}_G(G_{\alpha+1}).$$
\indent Claim 2: If $K_{[1]}$ has an infinite orbit $Y_1$ then $\textphnc{y}_L(L)+1\leqslant\textphnc{y}_G(G_1)$. Otherwise,\ $\textphnc{y}_L(L)+1\leqslant\textphnc{y}_G(G_2)$.\\
\indent Assume first $K_{[1]}$ has an infinite orbit $Y_1$. As $J_1=\bigoplus_{y\in Y_1}(L,\ U)$ is an infinite product, by Corollary \ref{coro:height well behaved on elementary} and Lemma \ref{lemma:infinite restricted prod rank},\ $\textphnc{f}(J_1)=\textphnc{y}_{J_1}(J_1)+1\geqslant \textphnc{y}_L(L)+1$. Together with Claim $1$ and the observation after Definition \ref{defi:residual height/rank},\ we get $\textup{Res}_{G}(J_1)<J_1\leqslant \textup{Res}_G(G_1)$ and $\textphnc{y}_G(G_1)\geqslant \textphnc{y}_G(J_1)+1 \geqslant\textphnc{y}_L(L)+1$.\\
\indent Assume now $K_{[1]}$ only has finite orbits. Because $Y_2$ is an infinite orbit of $K_{[2]}$, in a similar way to the case $K_{[1]}$ has an infinite orbit,\ it follows that $\textphnc{y}_G(G_2)\geqslant \textphnc{y}_G(J_2)+1 \geqslant\textphnc{y}_L(L)+1$.\\
\indent It then follows that either for $i=1$ or $2$,
$$\textphnc{y}_G(G_i)\geqslant \textphnc{y}_L(L)+1.$$ 
Notice that from Claim 1, we have that for every $\alpha>\beta\geqslant 1$, it is the case $G_{\alpha+1}\succ G_{\beta+1}$. If $K_{[1]}$ has an infinite orbit, it is also the case that $G_2\succ G_1$.\\
\indent If $K_{[1]}$ has an infinite orbit we then have
$$\textphnc{y}_G(G)> \textphnc{y}_G(G_{\alpha+1})>\textphnc{y}_G(G_{\beta+1})\geqslant\textphnc{y}_L(L)+1$$
for every $\textphnc{y}_K(K)>\alpha>\beta\geqslant 0$. In other words,
$$\textphnc{y}_G(G)\geqslant \sup_{\alpha<\textphnc{y}_K(K)} \textphnc{y}_G(G_{\alpha+1})\geqslant \sup_{\alpha<\textphnc{y}_K(K)}\textphnc{y}_{L}(L)+\textphnc{y}_K(K_{\alpha+1}).$$
Therefore, for this case we have $\textphnc{f}(G)\geqslant \textphnc{y}_L(L)+\textphnc{y}_K(K)+1$.\\
\indent If $K_{[1]}$ has no infinite orbit and $\textphnc{f}(K)>\omega$, it is the case that $\ldots\succ G_n\succ \ldots G_3\succ G_2\succ J_2$. Therefore
$$\textphnc{y}_G(G_{\omega+1})\geqslant \textphnc{y}_L(L)+\textphnc{y}_K(K_{[\omega+1]})>\sup_{n\in \mathbb{N},\ n>1}\textphnc{y}_{L}(L)+\textphnc{y}_K(K_{[n]})=\textphnc{y}_L(L)+\omega.$$
Hence, for every $\textphnc{y}_K(K)>\alpha>\omega$ we have $\textphnc{y}_G(G_{\alpha+1})\geqslant \textphnc{y}_L(L)+\textphnc{y}_K(K_{[\alpha+1]})$. It is then also the case that $\textphnc{f}(G)\geqslant \textphnc{y}_L(L)+\textphnc{y}_K(K)+1$.\\
\indent The equality follows from Lemma \ref{lemma:residual rank closed extensions} \textit{5.},\ as in this case, the group achieves the maximal possible rank.
\end{proof}

\begin{prop}\label{coro:still have collection witnessing the rank}
Assume $L$ and $K$ are elementary groups satisfying the conditions from Theorem \ref{thrm:making big ranks better with residue order} 
and $L$ also has a collection $\{L_{[\alpha+1]}\}_{\alpha<\textphnc{y}_L(L)}$ witnessing its rank. Then $L\wr_U(K,\ X)$ has a collection witnessing that its rank is $$\textphnc{y}_L(L)+\textphnc{y}_K(K)+1.$$ 
\end{prop}

\begin{proof}
Denote $G:=L\wr_U(K,\ X)$ and let $\{L_{[\alpha+1]}\}_{\alpha<\textphnc{y}_L(L)}$, and $\{K_{[\beta+1]}\}_{\beta<\textphnc{y}_K(K)}$ be collections witnessing the ranks of $L$ and $K$,\ respectively. We will prove first for the case there exists $Y_1$ an infinite orbit of $K_{[1]}$. Let $Y_1$ such an orbit and fix $y\in Y_1$. Define $L^{\{y\}}_{[\alpha+1]}\leqslant \bigoplus_{x\in X}(L,\ U)$ the subgroup that is $L_{[\alpha+1]}$ in the $y$ coordinate and the trivial group for $x\neq y$. For $\beta<\textphnc{y}_K(K)$ define $Y_{\beta+1}=K_{[\beta+1]} Y_1$. Then,\ for $\gamma<\textphnc{y}_L(L)+\textphnc{y}_K(K)$ we define $G_{[\gamma+1]}$ as follows:
$$G_{[\gamma+1]}:=\left\{\begin{array}{ll}
    L^{\{y\}}_{[\gamma+1]} & \mbox{if }\gamma< \textphnc{y}_L(L)   \\
    \bigoplus_{x\in Y_{\beta+1}}(L,\ U)\rtimes K_{[\beta+1]} & \mbox{if } \gamma=\textphnc{y}_L(L)+\beta. 
\end{array}
\right.$$ 
\indent For the case $K_{[1]}$ has no infinite orbit, let $Y_2$ be an infinite orbit of $K_{[2]}$. This orbit exists as $K_{[2]}$ is not residually discrete, as seen in the proof of Proposition \ref{prop: bad bound rank}. As in the other case, fix $y\in Y_2$, define $L_{[\alpha+1]}^{\{y\}}\leqslant \bigoplus_{x\in X}(L, \ U)$ and for $1<\beta<\textphnc{y}_K(K)$ the sets $Y_{\beta+1}=K_{[\beta+1]}Y_2$. Then for $\gamma <\textphnc{y}_L(L)+\textphnc{y}_K(K)$ we define $G_{[\gamma+1]}$ as follows:
$$G_{[\gamma+1]}:=\left\{\begin{array}{ll}
       L^{\{y\}}_{[\gamma+1]} & \mbox{if }\gamma<\textphnc{y}_L(L)   \\
    \bigoplus_{x\in Y_{n+1}}(L,\ U)\rtimes K_{[n+1]} & \mbox{if } \gamma=\textphnc{y}_L(L)+n\\
    \bigoplus_{x\in Y_{\beta+1}}(L,\ U)\rtimes K_{[\beta+1]} & \mbox{if } \gamma=\textphnc{y}_L(L)+\beta \mbox{ and }\beta>\omega.
\end{array}
\right.$$
\indent The fact $\{L_{[\alpha+1]}\}_{\alpha<\textphnc{y}_L(L)}$ is a collection witnessing the rank of $L$ and the proof of Theorem \ref{thrm:making big ranks better with residue order} implies that given $\gamma<\gamma'$,\ $G_{[\gamma+1]}<\textup{Res}_G(G_{[\gamma'+1]})\leqslant \textup{Res}(G)$. The condition that $\textphnc{y}_G(G_{[\gamma+1]})\geqslant \gamma+1$ follows from the construction.
\end{proof} 

Notice that in both Theorem \ref{thrm:making big ranks better with residue order} and Proposition \ref{coro:still have collection witnessing the rank}, to prove the case $K_1$ has no infinite orbit we just find a way to avoid the group $\bigoplus_{x\in Y_1}(L,\ U)\rtimes K_{[1]}$, as we don't know if this group has same height as $\bigoplus_{x\in Y_1}(L,\ U)$ in relation to $G$ or not.



\begin{prop}\label{prop:collection witnessing on ex(g,u)}
Assume $L$ is an elementary group and that $L$ acts on a countable set $X$ with compact open point stabilizers. If $U\in\mathcal{U}(L)$ is such that $\langle\!\langle U\rangle\!\rangle_L=L$ and there is a collection $\{L_{[\alpha+1]}\}_{\alpha<\textphnc{y}_L(L)}$ witnessing the rank of $L$ such that $L_1$ has an infinite orbit $Y_1\subset X$, then the group $E_X(L,\ U)$ has a collection witnessing that its rank is 
$$\textphnc{f}(E_X(L,\ U))=\omega^{\theta+1}+2,$$
where $\theta$ is the leading exponent in the Cantor normal form of $\textphnc{y}_L(L)$. If $L$ is also topologically perfect, then $P_X(L,\ U)$ is the monolith and has rank $\omega^{\theta+1}+1$. 
\end{prop}

\begin{proof}
To prove this result, we iterate the construction given in Proposition \ref{coro:still have collection witnessing the rank} to create the collection witnessing the rank of $E_X(L,\ U)$. \\
\indent For this proof, we use denote the iteration of the wreath product of $L$ $n$ times $L_n$, as in Definition \ref{defi:wr product}.\\
\indent Assume that $\textphnc{y}_{L}(L)=\omega^\theta k+\gamma$ for some $\omega^\theta>\gamma$ and $k\geqslant 1$ a natural number. Let $L_{n}$ be a closed subgroup of $E_X(L,\ U)$,\ as seen in Lemma \ref{lemma:Gn 101}. By Theorem \ref{thrm:making big ranks better with residue order} we have that $L_{1}<L_{2}<\ldots<L_{n}<\ldots$.  By Proposition \ref{coro:still have collection witnessing the rank},\ $L_{n}$ has a collection witnessing that its rank is greater or equal than 
$$\left(\underbrace{\textphnc{y}_L(L) + \ldots + \textphnc{y}_L(L)}_{n\rm\ times}\right)+1=\left(\underbrace{(\omega^{\theta}k+\gamma) + \ldots + (\omega^{\theta}k+\gamma)}_{n\rm\ times}\right)+1=\omega^{\theta}kn+\gamma+1\geqslant \omega^{\theta}n.$$
Notice that the collection witnessing the rank of $L_{i}$,\ as defined in the wreath product case,\ is contained in the collection witnessing the rank of $L_{i+1}$. Let $\{G_{[\gamma+1]}\}_{\gamma<\textphnc{y}_L(L).\omega}$ be the union of all these collections. It is a collection of $E_X(L,\ U)$ witnessing that the rank is greater or equal than $\omega^{\theta+1}+1$. Lemma \ref{lemma:Gn 101} and Proposition \ref{prop: monolith tree} imply that $L_{n}<P_X(L,\ U)$ and that $P_X(L,\ U)$ is the monolith, hence $L_{n}<P_X(L,\ U)\leqslant \textup{Res}(E_X(L,\ U))$,\ for all $n\in\mathbb{N}$. Hence 
$$\textphnc{y}_{E_X(L,\ U)}(E_X(L,\ U))\geqslant \sup_{n\in\mathbb{N}}\textphnc{y}_{L_{n}}(L_{n}) +1\geqslant \sup_{n\in\mathbb{N}}\omega^{\theta}n+1=\omega^{\theta+1}+1.$$
\indent Now it remains to prove that the equation above is an equality. For $v\in V\vec{T}_{\chi}$ a fixed vertex, let $H_n:=\langle L_{n},\ W_{(T_0)},\ W_{(T_R \cup B_n(v))}\rangle$, where $B_n(v)=\cup_{0\leqslant i\leqslant n}\textup{inn}^i(v)$, $T_0=\cup_{i\in\mathbb{N}}\textup{inn}^i(v)\cup \{v\}$ and $T_R=V\vec{T}_{\chi}\backslash T_0$ as in the proof of Proposition \ref{prop:grp is elementary}. As seen in Proposition \ref{prop:grp is elementary}, the subgroups $H_n$ are open in $P_X(G,\ U)$,\ and for $n\in\mathbb{N}$ the subgroup $W_{(T_R \cup B_n(v))}W_{(T_0)}$ is a normal compact subgroup of $H_n$ such that $H_n/(W_{(T_R \cup B_n(v))}W_{(T_0)})\cong L_{n}$ (the subgroup $W_{(T_R \cup B_n(v))}W_{(T_0)}$ is denoted $L_nW_{(T_0)}$ in Proposition \ref{prop:grp is elementary}). We thus deduce that $H_n$ is an open elementary subgroup of $P_X(L,\ U)$. Furthermore,\ Theorem \ref{thrm:making big ranks better with residue order} and Lemma \ref{lemma:short exact sequence decomposition rank} ensure
\begin{equation}\label{eq:3.1}
    \omega^\theta n\leqslant \textphnc{f}(H_n)< \omega^{\theta+1}.
\end{equation}
\indent The collection $\{H_n\}_{n\in\mathbb{N}}$ is directed. So $\bigcup_{n\in\mathbb{N}}H_n$ is an open elementary subgroup of $P_X(L,\ U)$. On the proof of Proposition \ref{prop:grp is elementary} we see that
$$\bigcup_{n\in\mathbb{N}}H_n=E_X(L,\ U)_{(T_R)}W_{(T_0)}.$$
Moreover,\ following the fact that the rank of a group is the supremum of the rank of all open subgroups,\ equation (\ref{eq:3.1}) gives us that $\textphnc{f}(E_X(L,\ U)_{(T_R)}W_{(T_0)})=\omega^{\theta+1}+1$. Fixing $t$ a translation toward the end $\chi$,\ the family $(t^n E_X(L,\ U)_{(T_R)}W_{(T_0)}t^{-n})_{n\geqslant 0}$ is directed. In view of Lemma \ref{lemma:Gn 101} we get that
$$\bigcup_{n\geqslant 0}t^n E_X(L,\ U)_{(T_R)}W_{(T_0)}t^{-n}=P_X(L,\ U).$$
As $\textphnc{f}(t^n E_X(L,\ U)_{(T_R)}W_{(T_0)}t^{-n})=\textphnc{f}(t^m E_X(L,\ U)_{(T_R)}W_{(T_0)}t^{-m})$, for every $n,\ m \in\mathbb{Z}$, a direct consequence of Proposition \ref{coro:simplified residual rank general} is that $\textphnc{f}(P_X(L,\ U))=\textphnc{f}(t^n E_X(L,\ U)_{(T_R)}W_{(T_0)}t^{-n})=\omega^{\theta+1}+1$.\\
\indent The discrete residual of $E_X(L,\ U)$ is contained in $P_X(L,\ U)$,\ therefore $\textphnc{f}(\textup{Res}(E_X(L,\ U)))\leqslant \textphnc{f}(P_X(L,\ U))$. Thus,\ $\textphnc{f}(E_X(L,\ U))=\textphnc{f}(\textup{Res}(E_X(L,\ U)))+1\leqslant \textphnc{f}(P_X(L,\ U))+1$,\ and we deduce that $$\omega^{\theta+1}+1\leqslant\textphnc{f}(E_X(L,\ U))\leqslant \omega^{\theta+1}+2.$$
Non-trivial compactly generated groups cannot have a residual/decomposition rank of the form $\lambda+1$,\ with $\lambda$ a limit ordinal,\ thus $\textphnc{f}(E_X(L,\ U))=\omega^{\theta+1}+2$. It then follows that the collection $\{G_{[\gamma+1]}\}_{\gamma<\textphnc{y}_L(L).\omega+1}$, with $G_{[\textphnc{y}_L(L).\omega+1]}=E_X(L,\ U)$, is a collection witnessing the rank of $E_X(L,\ U)$.\\
\indent The final claim is immediate from Proposition \ref{prop: monolith tree}.
\end{proof}

\begin{coro}\label{coro:adding 1 to the rank}
Let $L$ and $K$ elementary compactly generated groups such that $\xi(K)=2$. If $U\in\mathcal{U}(L)$ is such that $\langle\!\langle U\rangle\!\rangle_L=L$ 
and $(K,\ X)$ is a permutation group satisfying the conditions from Theorem \ref{thrm:result for higher rank by wesolek}, by the proof of Theorem \ref{thrm:making big ranks better with residue order}, 
$$\textphnc{f}(L\wr_U(K,\ X))=\textphnc{f}(L)+1.$$
\end{coro}

\begin{proof}
Theorem \ref{thrm:making big ranks better with residue order} implies 
$$\textup{Res}(L\wr_U(K,\ X))\geqslant \bigoplus_{x\in X}(L,\ U)\rtimes \textup{Res}_K(K)=\bigoplus_{x\in X}(L,\ U),$$
where the equality follows from the fact $K$ is compactly generated and $\xi(K)=2$. Because $\bigoplus_{x\in X}(L,\ U)$ is an open,\ normal subgroup of $L\wr_U(K,\ X)$ then
$$\textup{Res}(L\wr_U(K,\ X))\leqslant \bigoplus_{x\in X}(L,\ U)$$
That is,\ $\textup{Res}(L\wr_U(K,\ X))=\bigoplus_{x\in X}(L,\ U)$. Hence by Theorem \ref{thrm:for elementary groups ranks are the same} and Lemma \ref{lemma:infinite restricted prod rank} we have $\textphnc{y}_{L\wr_U(K,\ X)}(L\wr_U(K,\ X))=\textphnc{y}_L(L)+1$,\ and the result follows.
\end{proof}

\begin{prop}\label{prop:one more prop to order in the quotient}
Let $G$ be an elementary group and $N\trianglelefteq G$ a compact normal subgroup. If $N\trianglelefteq G$ is closed and $\{1\}\prec K_1\prec K_2$ then $\{1\}\preccurlyeq K_1N/N\prec K_2N/N$.
\end{prop}

\begin{proof}
Notice that $K_2\nleqslant N$,\ as the claim implies the height of $K_2$ is at least $2$. Hence $K_2N/N\neq \{1\}$. Let $\overline{O}\leqslant G/N$ be an open subgroup such that for all $\overline{k}\in K_2N/N$,\ $\overline{k} \overline{O}=\overline{O} \overline{k}$, that is, $\overline{O}$ is an open subgroup of $G/N$ normalised by $K_2/N$. Define $O$ be the inverse image of $\overline{O}$ under the quotient under $N$. Notice that $O$ is normalised by $K_2$,\ hence $K_1\leqslant O$ and $K_1N/N\leqslant \overline{O}$. Taking the intersection on all such $\overline{O}$ gives us that $K_1N/N\leqslant \textup{Res}_{G/N}(K_2N/N)$. As $G$ is elementary,\ the remark after Proposition \ref{prop:res do what i want in elementary} implies $K_1N/N\neq K_2N/N$,\ giving us the desired result.
\end{proof}

\begin{coro}\label{equation:maybe 4}
Let $G$ be an elementary group and $N\trianglelefteq G$ a compact normal subgroup. If $\{G_{[\alpha+1]}\}_{\alpha<\gamma}$ is a collection witnessing that the rank of $G$ is at least $\gamma+1$ then
$$\textphnc{y}_{H}(G_{\alpha+1}N/N)\in \{-1+\alpha+1,\ \alpha+1\}$$
for every $\alpha<\gamma$. If $\alpha>\omega$ then 
$$\textphnc{y}_{H}(G_{\alpha+1}N/N)=\textphnc{y}_G(G_{\alpha+1})=\alpha+1.$$
\end{coro}

\begin{proof}
Let $\{G_{[\alpha+1]}\}_{\alpha<\gamma}$ be a collection witnessing that the rank of $G$ is at least $\gamma+1$. By Lemma \ref{lemma:short exact sequence decomposition rank} it follows that for every $\alpha<\gamma$
$$\xi(G_{[\alpha+1]}N/N)\leqslant\xi(G_{[\alpha+1]})\leqslant \xi(G_{[\alpha+1]}\cap N)-1+\xi(G_{[\alpha+1]}N/N).$$
As $N$ is compact then $\xi(G_{[\alpha+1]}\cap N)\leqslant 2$. We then get
$$\xi(G_{[\alpha+1]}N/N)\leqslant\xi(G_{[\alpha+1]})\leqslant 1+\xi(G_{[\alpha+1]}N/N).$$
Hence,\ by Corollary \ref{coro:height well behaved on elementary},
$$\textphnc{y}_{H}(G_{[\alpha+1]}N/N)\leqslant \textphnc{y}_G(G_{[\alpha+1]})=\alpha+1\leqslant 1+\textphnc{y}_{H}(G_{[\alpha+1]}N/N)$$
that is,\ $\textphnc{y}_{H}(G_{[\alpha+1]}N/N)\in \{-1+\alpha+1,\ \alpha+1\}$. If $\alpha>\omega$ then the inequality above implies that
$$1+\textphnc{y}_H(G_{[\alpha+1]}N/N)\leqslant 1+\alpha+1=\alpha+1.$$
Therefore $\textphnc{y}_H(G_{[\alpha+1]}N/N)\leqslant 1+\alpha+1=\alpha+1$.
\end{proof}

\begin{prop}\label{prop:collection witnessing rank well after quotient}
Let $G$ be an elementary group and $N\trianglelefteq G$ a compact normal subgroup. If $G$ has a collection $\{G_{[\alpha+1]}\}_{\alpha<\gamma}$ witnessing that the rank of $G$ is at least $\gamma+1$, for $\gamma\geqslant\omega$, then $H:=G/N$ has a collection $\{H_{[\alpha+1]}\}_{\alpha<\gamma}$ witnessing that its rank is at least $\gamma+1$.\\
\indent In particular, if $\textphnc{f}(G)>\omega$ and $G$ has a collection witnessing its rank, then so does $H$.
\end{prop}

\begin{proof}
\indent Let $\pi:G\rightarrow G/N=H$ be the projection. For our proof, we want to use a collection $\{G_{[\alpha+1]}\}_{\alpha<\gamma}$ that witness the rank of $G$ is at least $\gamma+1$ and, if $\gamma>\omega$, build a chain $\{H_{[\alpha+1]}\}_{\alpha<\gamma}$ witnessing that the rank of $H$ is at least $\gamma+1$. We will begin by building a collection $\{H_{[n+1]}\}_{n<\omega}$ witnessing that the rank of $H$ is at least $\omega+1$ and then extend it.\\
\indent Let $\{G_{[n+1]}\}_{n<\omega}$ be a collection witnessing that the rank of $G$ is at least $\omega+1$. As $\textphnc{y}_{H}(G_{[n+1]}N/N)\in \{n,\ n+1\}$ (Corollary \ref{equation:maybe 4}) the collection $\{\pi(G_{[n+1]})\}_{n<\omega}$ might not cover all ranks below $\omega$. It might also be the case that there is $n\in\mathbb{N}$ such that $\textphnc{y}_{H}(\pi(G_{[n]}))=\textphnc{y}_{H}(\pi(G_{[n+1]}))$. Hence, to build the collection $\{H_{[n+1]}\}_{n<\omega}$ witnessing that the rank of $H$ is at least $\omega+1$ we need to manage such problems. We will begin finding which ranks we can get from $\{\pi(G_{[n]})\}_{n\in\mathbb{N}}$ and then add groups to our collection with the necessary properties so we get a chain witnessing the rank of $H$ is at least $\omega+1$.\\
\indent Our first incomplete collection will be built recursively in the following way:
\begin{enumerate}
    \item Let $k=1$;
    \item If $\pi(G_{[1]})\neq \{1\}$,\ define $H_{[1]}:=\pi(G_{[1]})$;
    \item If $\pi(G_{[1]})=\{1\}$, discard the group $\pi(G_{[1]})$ and go to step $4$;
    \item Assume you just looked at the group $\pi(G_{[k]})$ for $k=n$. Let $k=n+1$;
    \item If $\textphnc{y}_{H}(\pi(G_{[k]}))=m$ and $H_{[m]}$ was not defined,\ define $H_{[m]}:=\pi(G_{[k]})$ and go to step $4$;
    \item If $\textphnc{y}_{H}(\pi(G_{[k]}))=m$ and $H_{[m]}$ was already defined, discard the group $\pi(G_{[k]})$ and go to step $4$.
\end{enumerate}
Corollary \ref{equation:maybe 4} implies that this recursion gives us an infinite collection of groups with strictly increasing rank that is contained in $\{H_{[n+1]}\}_{n\in S}\subset\{\pi(G_{[n+1]})\}_{n<\omega}$, where $S:=\{n\in \omega;\ H_{[n+1]} \text{ was defined}\}$. As $\{G_{[n+1]}\}_{n<\omega}$ is a collection witnessing the rank of $G$ is at least $\omega+2$ then for every $n<m$, by definition, it is the case that $G_{[n]}\prec G_{[m]}$. Hence Proposition \ref{prop:one more prop to order in the quotient} implies that given $n,\ m\in S$ with $n<m$, then $H_{[n+1]}\prec H_{[m+1]}$.\\
\indent Assume $m\in\omega\backslash S$ is the smallest element such that $H_{[m+2]}$ was defined but $H_{[m+1]}$ was not. As the collection $S$ is unbounded such element always exist. Because $\textphnc{y}_{H}(H_{[m+2]})=m+2$ then there exist $K\prec H_{[m+2]}$ such that $\textphnc{y}_H(K)=m+1$. Define $H_{[m+1]}:=\overline{\langle K,\ \pi(G_{[m+1]}) \rangle}$. Because $K,\ \pi(G_{[m+1]})\leqslant \textup{Res}_H(H_{[m+2]})$,\ then $H_{[m+1]}$ is a compactly generated subgroup of $\textup{Res}_{H}(H_{[m+2]})$. Because $H$ is elementary,\ Proposition \ref{prop:res do what i want in elementary} implies $\textup{Res}_H(H_{[m+2]})\ngeqslant H_{[m+2]}$,\ hence $H_{[m+1]}\prec H_{[m+2]}$. Suppose now that for $n<m$ the group $H_{[n+1]}$ was defined. The proof of Proposition \ref{prop:one more prop to order in the quotient} gives us that $H_{[n+1]}\leqslant \textup{Res}_H(\pi(G_{[m+1]}))$ hence,\ because $\textup{Res}_H(\pi(G_{[m+1]}))\leqslant \textup{Res}_H(H_{[m+1]})$ then the difference in height implies $H_{[n+1]}\prec H_{[m+1]}$. We can then repeat the argument of this paragraph for the new set $S\cup \{m\}$ and, recursively, define $H_{[n+1]}$ for every $n<\omega$. We then have our desired collection $\{H_{[n+1]}\}_{n<\omega}$ witnessing that the rank of $H$ is at least $\omega+1$.\\
\indent Now assume $G$ has a collection $\{G_{[\alpha+1]}\}_{\alpha<\gamma}$ witnessing that its rank is at least $\gamma+1$. We want to build a collection $\{H_{[\alpha+1]}\}_{\alpha<\gamma}$ witnessing that the rank of $H$ is at least $\gamma+1$.\\
\indent Let $\{G_{[n+1]}\}_{n<\omega}\subset \{G_{[\alpha+1]}\}_{\alpha<\gamma}$ the subcollection witnessing that the rank of $G$ is at least $\omega+2$. We can build $\{H_{[n+1]}\}_{n<\omega}$ the collection witnessing that the rank of $H$ is at least $\omega+2$ obtained from $\{\pi(G_{[n+1]})\}_{n<\omega}$ as given in the proof above. Now let $\alpha>\omega$. Notice that Corollary \ref{equation:maybe 4} implies that $\textphnc{y}_H(\pi(G_{[\alpha+1]}))=\alpha+1$. We can then define, for $\textphnc{y}_G(G)>\alpha>\omega$, $H_{[\alpha+1]}=\pi(G_{[\alpha+1]})$. We can then extend our initial collection $\{H_{[n+1]}\}_{n<\omega}$ and obtain $\{H_{[\alpha+1]}\}_{\alpha<\textphnc{y}_G(G)}$.\\
\indent It remains to prove that for every $\beta<\alpha$ we have that $H_{[\beta+1]}\prec H_{[\alpha+1]}$. Defining $H_{[\alpha+1]}:=\pi(G_{[\alpha+1]})$ for every $\alpha<\gamma$ and adding to the collection then gives us the desired result.\\
\indent The case $G$ has a collection $\{G_{[\alpha+1]}\}_{\alpha<\textphnc{f}(G)-1}$ witnessing its rank is just the particular case $\gamma=\textphnc{f}(G)-1$, hence we can also obtain a collection $\{H_{[\alpha+1]}\}_{\alpha<\textphnc{f}(G)-1}$ witnessing that the rank of $H$ is at least $\textphnc{f}(G)$. As $\textphnc{f}(G)\geqslant\textphnc{f}(H)$, it then follows that $\textphnc{f}(H)=\textphnc{f}(G)$ and the result follows.
\end{proof}


\subsection{Higher bounds for the decomposition rank}

Now that we have a result related to building groups with higher ranks,\ we need to find a way to build groups satisfying the conditions from the proposition so we can iterate the construction.

\begin{lemma}\label{lemma:how to iterate tree construction for rank}
Suppose that $G$ is a compactly generated elementary group with an open $L\trianglelefteq G$,\ such that $G=L\rtimes
\langle t \rangle$ where $|t|=\infty$. If $L$ is topologically perfect and there is $U\in\mathcal{U}(L)$ such that $\langle\!\langle U\rangle\!\rangle_L=L$,\ then there exists a compactly generated elementary group $E$ such that $E$ is topologically perfect,\ $\xi(E)=\xi(G)$,\ and there is $V\in\mathcal{U}(E)$ such that $\langle\!\langle V \rangle\!\rangle_E=E$.
\end{lemma}

\begin{proof}
First,\ we find a nice generating set for $L$. Fix $K\subseteq G$ a compact generating set for $G$. Because $L$ is open,\ $G= L\rtimes \langle t\rangle$ and $K$ is compact,\ there is $n\geqslant 0$ such that $\cup_{i=-n}^n t^iL$ is a finite covering of $K$. Defining $K_i:=t^iL\cap L$ we get that
$$K=K_{-n}\cup K_{-(n-1)}\cup\hdots\cup K_{n-1}\cup K_n.$$
Setting $Y:=\cup_{i=-n}^n t^{-i}K_i$,\ the group $G$ is generated by $Y\cup \{t\}$,\ and $Y$ is a compact subset of $L$. In particular $L=\langle t^iYt^{-i}|\ i\in \mathbb{Z} \rangle$.\\
\indent We now define the group $E$. For this proof,\ we can use any simple non-commutative group that acts transitively on a finite set,\ but for the sake of simplicity,\ we will take the alternating group $A_5$ acting on the set $[0,\ 4]=\{0,\ 1,\ 2,\ 3,\ 4\}$. Let $A_5$ act on $G^{[0,4]}:=\prod_{i\in [0,4]}G$ by shifting the domain. Define the following semidirect product
$$M:=G^{[0,4]}\rtimes A_5\cong L^{[0,4]}\rtimes (\mathbb{Z}^{[0,4]}\rtimes A_5).$$
Let $H:=\langle\!\langle A_5 \rangle\!\rangle_{\mathbb{Z}^{[0,4]}\rtimes A_5}=\langle fgf^{-1} | \ f\in \mathbb{Z}^{[0,4]},\ g\in A_5 \rangle$. Notice that $H\cap \mathbb{Z}^{[0,4]}$ has finite index in $H$, as $(H\cap \mathbb{Z}^{[0,4]})A_5=H$. Hence, as $H\cap \mathbb{Z}^{[0,4]}$ is finitely generated,\ the group $H$ is also finitely generated. Note that because $A_5$ is perfect,\ then $A_5\leqslant [H,\ H]$. Since $[H,\ H]$ is the commutator of a normal subgroup of $\mathbb{Z}^{[0,4]}\rtimes A_5$,\ it is also normal in $\mathbb{Z}^{[0,4]}\rtimes A_5$. By the construction of $H$,\ it follows that $[H,H]=H$,\ that is,\ $H$ is perfect.\\
\indent For each $g\in L$,\ let $f_g\in L^{[0,4]}$ be the element of $M$ which takes value $g$ on $0$ and the identity on $\{1,\ 2,\ 3,\ 4\}$. Fix $U\in\mathcal{U}(L)$ such that $\langle\!\langle U\rangle\!\rangle_L=L$. Define $E:=\langle H,\ U^{[0,4]},\ f_g|\ g\in Y \rangle $. The group $E$ is elementary since it is an open subgroup of the elementary group $M$,\ and is compactly generated by construction. \\
\indent To prove $E$ is topologically perfect,\ define $r\in \mathbb{Z}^{[0,4]}$ as
$$r(i):=\left\{
 \begin{array}{ll}
    1 & \text{if }i=0,\ \\
    -1 & \text{if }i=1,\\
    0 & \text{otherwise}.
\end{array}
\right.$$
Note that $r\in H$,\ hence $r\in E$. For each $g\in Y$ and $n\in\mathbb{Z}$,\ the element 
\begin{equation}\label{eq:number 10}
(nr)f_g(-nr)=f_{t^ngt^{-n}}    
\end{equation}
is also in $E$. By our choice of $Y$,\ we conclude that $L^0$,\ the set of functions in $L^{[0,4]}$ supported on $0$,\ is contained in $E$,\ and since $A_5$ acts transitively on $[0,4]$,\ it is indeed the case that $L^{[0,4]}\leqslant E$. The groups $H$ and $L^{[0,4]}$ are topologically perfect,\ and moreover,\ $E=\langle H,\ L \rangle$. We thus deduce that $E$ is topologically perfect.\\
\indent Set $V=U^{[0,4]}\rtimes A_5$. The group $V$ is a compact open subgroup of $E$. It is also the case that for every $g\in L$, $(nr)f_g(-nr)=f_{t^ngt^{-n}}$ (Equation \ref{eq:number 10}) is an element of $\langle\!\langle V\rangle\!\rangle_E$. The group $\langle\!\langle V\rangle\!\rangle_E$ thus contains $f_g$ for all $g\in\langle\!\langle U\rangle\!\rangle_L=L$. Since $A_5\leqslant V$,\ we conclude that $H,\ L^{[0,4]}\leqslant \langle\!\langle V\rangle\!\rangle_E$,\ and thus $\langle\!\langle V\rangle\!\rangle_E=E$.\\
\indent To prove $\xi(E)=\xi(G)$,\ let $\pi_0:G^{[0,4]}\rightarrow G$ be the projection on the coordinate indexed by $0$. Note that $L\leqslant \pi_0(E\cap G^{[0,4]})$,\ and since $r$ as above is an element of $E$,\ the generator $t$ is also an element of $\pi_0(E\cap G^{[0,4]})$. Hence,\ the image equals $G$. From Proposition \ref{prop:quotient elementary 101},\ it follows that $\xi(E)\geqslant \xi(G)$. On the other hand,\ Lemma \ref{lemma:cocompact has same rank as group} and Proposition \ref{prop:product same rank} imply $\xi(G)=\xi(M)\geqslant \xi(L)$. We thus deduce that $\xi(E)=\xi(G)$.
\end{proof}

\begin{prop}\label{prop:extra condition for semidirect product z and collection witnessing}
Suppose that $G$ is a compactly generated elementary group with an open $L\trianglelefteq G$,\ such that $G=L\rtimes
\langle t \rangle$ where $|t|=\infty$. If $G$ and  $L$ satisfy the conditions of Lemma \ref{lemma:how to iterate tree construction for rank} and $G$ has a collection witnessing its rank that is contained in $L$,\ then the new group $E$ generated by Lemma \ref{lemma:how to iterate tree construction for rank} also has a collection witnessing its rank.
\end{prop}

\begin{proof}
Because $G$ and $E$ are elementary and contains copies of $L$ it follows,\ from Corollary \ref{coro:height well behaved on elementary} and the assumption,\ that the collection witnessing the rank of $G$ can be used to witness that the rank of $E$ is at least $\alpha+2$. Because $\xi(G)=\xi(E)=\alpha+2$,\ the result follows.
\end{proof}

We now almost have what is necessary to build groups with higher decomposition ranks. The construction of $E_X(G,\ U)$ and Proposition \ref{prop:collection witnessing on ex(g,u)} gives us conditions for creating groups with a higher rank. Lemma \ref{lemma:how to iterate tree construction for rank} allows us to build a group with almost all necessary conditions necessary to apply Proposition \ref{prop:collection witnessing on ex(g,u)}, except an action on a countable set $X$ with compact, open point stabilizers and a collection witnessing its rank.\\
\indent To simplify the notation and the recursion on the next theorem, we denote $\omega^0:=0$.

\begin{theorem}\label{thrm:result for higher rank by wesolek}
Let $G$ be a compactly generated topologically perfect elementary group with $\xi(G)=2$,\ and $U\in \mathcal{U}(G)$ a proper subgroup such that $\langle\!\langle U\rangle\!\rangle_G=G$. Let $X$ be a countable set such that $(G,\ X)$ is a transitive permutation t.d.l.c.s.c.~group. For each $n\geqslant 0$,\ there exist a topologically perfect compactly generated elementary group $F_n(G)$ and a countable set $X_n$ such that $\xi(F_n(G))=\omega^n+2$,\ there exist $U_n(G)\in\mathcal{U}(F_n(G))$ with $\langle\!\langle U_n(G)\rangle\!\rangle_{F_n(G)}=F_n(G)$, $(F_n(G),X_n)$ is a transitive t.d.l.c.s.c.~permutation group and $F_n(G)$ has a chain witnessing its rank.
\end{theorem}

\begin{proof}
For the proof first define $F_0(G):=G$,\ $U_0(G):=U$ and $X_0:=X$. By assumption,\ they satisfy the conditions on the statement, as the empty collection witness the rank of $F_0(G)$.\\
\indent We will now prove the result by induction which will be divided in three parts. First,\ we build a topologically perfect group with an open compact subgroup satisfying the desired properties. After that,\ we construct a countable set for the group to act on transitively and with open,\ compact point stabilizers. As the last step,\ we trim up the topologically perfect group so that the action is also faithful and the group still satisfies all the other properties. We reiterate the construction with this new permutation group.\\
\indent Given $n\geqslant 0$,\ assume the permutation group $(F_n(G),\ X_n)$ and $U_n(G)$ its subgroup were defined, and $F_n(G)$ has a collection witnessing its rank. Define $E_{n+1}:=E_{X_n}(F_n(G),U_n(G))$. By Proposition \ref{prop:collection witnessing on ex(g,u)},\ $\xi(E_{n+1})=\omega^{n+1}+2$ and $M_{n+1}:=P_{X_n}(F_n(G),\ U_n(G))$ is its monolith. By Proposition \ref{prop:new grp is comp gener},\ $E_{n+1}$ is compactly generated. By Proposition \ref{prop:collection witnessing on ex(g,u)} the group $E_{n+1}$ has a chain witnessing that its rank is $\omega^{n+1}+2$. Because $M_{n+1}$ is the monolith of $E_{n+1}$,\ $M_{n+1}$ is topologically perfect. More than that,\ given $V_{n+1}\in\mathcal{U}(M_{n+1})$ we have that $\langle\!\langle V_{n+1}\rangle\!\rangle_{M_{n+1}}=M_{n+1}$ (Proposition \ref{prop:gluing result remaining}). Fix such a $V_{n+1}$. By Proposition \ref{prop:writing ex(g,U) as semidirect product} we have that $E_{n+1}=M_{n+1}\rtimes \mathbb{Z}$. Therefore, we can apply Lemma \ref{lemma:how to iterate tree construction for rank} to build the group $F_{n+1}'(G)$ and its subgroup $U_{n+1}'(G)$ (as the groups $E$ and $V$ from the lemma, respectively), and these will satisfy $\xi(F_{n+1}'(G))=\xi(E_{n+1}(G))=\omega^{n+1}+2$, $\langle\!\langle U_{n+1}'(G)\rangle\!\rangle_{F_{n+1}'(G)}=F_{n+1}'(G)$ and $F_{n+1}'(G)$ is both topologically perfect and compactly generated. By applying Proposition \ref{prop:extra condition for semidirect product z and collection witnessing} it also follows that the group $F_{n+1}'(G)$ has a collection witnessing that its rank is $\omega^{n+1}+2$.\\
\indent Assume now that the group $F_{n+1}'(G)$ and its subgroup $U_{n+1}'(G)$ have been defined,\ for $n\geqslant 0$. Define $X_{n+1}:=F_{n+1}'(G)/U_{n+1}'(G)$,\ a set of cosets,\ and let $F_{n+1}'(G)$ act on it by left multiplication. Define the group $F_{n+1}(G):=F_{n+1}'(G)/F_{n+1}'(G)_{(X_{n+1})}$,\ the quotient of $F_{n+1}'(G)$ under the normal core of $U_{n+1}'(G)$. It is clear that $(F_{n+1}(G),X_{n+1})$ is a transitive t.d.l.c.s.c.~permutation group. It is also true that $F_{n+1}(G)$ is compactly generated,\ topologically perfect and $U_{n+1}(G)$,\ the image of $U_{n+1}'(G)$ under the quotient,\ is such that $\langle\!\langle U_{n+1}(G)\rangle\!\rangle_{F_{n+1}(G)}=F_{n+1}(G)$. We need to prove that $\xi(F_{n+1}(G))=\omega^n+2$.\\
\indent Notice that as $F_{n+1}'(G)_{(X_{n+1})}$ is the normal core of $U_{n+1}'(G)$, it is compact. By Proposition \ref{prop:collection witnessing rank well after quotient}, as $F_{n+1}'(G)$ has a collection witnessing its rank is $\omega^{n+1}+2$, then so does $F_{n+1}(G)$. Then $\xi(F_{n+1}'(G)_{(X_{n+1})})=2$. By Propositions \ref{prop:quotient elementary 101},\ \ref{lemma:short exact sequence decomposition rank} and the exact sequence
$$\begin{tikzcd}
\{1\} \arrow[r] & F_{n+1}'(G)_{(X_{n+1})} \arrow[r] & F_{n+1}'(G) \arrow[r] & F_{n+1}(G) \arrow[r] & \{1\}
\end{tikzcd}$$
we get that $\xi(F_{n+1}(G))\leqslant\xi(F_{n+1}'(G))=\omega^{n+1}+2\leqslant 2+\xi(F_{n+1}(G))$. Hence,\ by the properties of infinite ordinals,\ it follows that $\xi(F_{n+1}(G))=\omega^{n+1}+2$.
\end{proof}

For the sake of concreteness,\ one could take $G$ on the theorem to be any infinite,\ finitely generated simple group with a non-trivial finite subgroup $U$. It then follows that $\langle\!\langle U\rangle\!\rangle_G=G$. Examples of groups satisfying these properties are countable products of simple groups with any proper non-trivial open subgroup, Thompson's group $V$ with $U=A_5$, and Tarski monster groups with $U$ any non-trivial $1$-generated subgroup. On the case of countable, discrete groups, one can define $X:=G/U$,\ and it will follow that $(G,\ X)$ is a transitive t.d.l.c.s.c.~permutation group. Because there exist continuum many non-isomorphic Tarski monster groups and continuum many products of finite, simple groups,\ we have at least continuum many ways to build the groups $F_n(G)$.

\subsection{Building elementary groups with given rank}

Our main goal in defining the residual height and rank is using them with the construction from Theorem \ref{thrm:result for higher rank by wesolek} to build groups with any given rank below $\omega^\omega$. 



\begin{remark}\label{prop:making other properties work addition on rank}
Let $L,\ K$ be non-trivial compactly generated groups and $X_L,\ X_K$ countable sets. If $L$ acts on $X_L$ and $K$ acts on $X_K$ with compact open point stabilizers (the action does not need to be faithful) and there exist $U_L\in\mathcal{U}(L)$,\ $U_K\in \mathcal{U}(K)$ such that $\langle\!\langle U_L\rangle\!\rangle_L=L$,\ $\langle\!\langle U_K\rangle\!\rangle_K=K$ then $G:=K\wr_{U_L}(L,X_K)$ acts on $X:=X_L\times X_K$ by the imprimitive action (Definition \ref{defi:wr product}) with compact open point stabilizers, and $U:=\prod_{x\in X_K}U_L\rtimes U_K$ is a compact open subgroup of $G$ such that $\langle\!\langle U\rangle\!\rangle_G=G$.
\end{remark}

With this remark and Theorem \ref{thrm:making big ranks better with residue order}  we have all that is necessary to build a group with a given decomposition rank below $\omega^\omega+1$. For the sake of concreteness, we will simplify a notation to build them. 

\begin{notation}[Iterated product]\index{iterated product|ndx}
Let $(G,\ X)$,\ $(G',X')$ and $(G'',X'')$ be permutation t.d.l.c.s.c.~groups and $U\in\mathcal{U}(G),\ \ U'\in\mathcal{U}(G')$ fixed open,\ compact subgroups. We denote by 
$$G\wr_{U}G'\wr_{U'}G'':=G\wr_{U}(G'\wr_{U'}(G'',X''),(X'\times X''))$$
where the action of the groups is given by the imprimitive action as defined in Definition \ref{defi:wr product}. This notation is used for sequences of groups longer than three by taking the wreath products from left to right, and actions given by the imprimitive action (Definition \ref{defi:wr product}).\\
\indent For $(G,\ X)$ a permutation group and $U\in\mathcal{U}(G)$ we will denote 
$$(G)_0:=\{1\}, \ \ (U)_0:=\{1\},$$
$$(G)_1:=G,\ \ (U)_1:= U,$$
and for $n>1$
$$(G)_n:=G\wr_{U}((G)_{n-1},X^{n-1}),\ \ (U)_{n}:=\left(\prod_{x\in X^{n-1}}U\right)\rtimes (U)_{n-1}.$$
\end{notation}

Now,\ for the construction,\ let $\alpha< \omega^\omega+1$ be a non limit ordinal. Using the Cantor normal form we can write it as
$$\alpha=\sum_{i=1}^n \omega^i a_i+(a_0+1)$$
where $a_i\in \mathbb{N}$. Pick a group $G$ satisfying the conditions from Theorem \ref{thrm:result for higher rank by wesolek} and build the permutation groups $(F_n(G),X_n)$ and its subgroups $U_n(G)$ as in Theorem \ref{thrm:result for higher rank by wesolek}. If $a_0>0$ then we build the compactly generated topologically perfect groups
$$(F_n(G))_{a_n}\wr_{(U_n(G))_{a_n}}(F_{n-1}(G))_{a_{n-1}}\wr_{(U_{n-1}(G))_{a_{n-1}}}\ldots\wr_{(U_{1}(G))_{a_1}} (F_0(G))_{a_0},$$
where $a_i$ is the coefficient next to $\omega^i$ in the Cantor normal form of $\alpha$. If $a_0=0$,\ we build the group
$$\textup{Res}\left((F_n(G))_{a_n}\wr_{(U_n(G))_{a_n}}(F_{n-1}(G))_{a_{n-1}}\wr_{(U_{n-1}(G))_{a_{n-1}}}\ldots\wr_{(U_{1}(G))_{a_1}} (F_0(G))_{1}\right).$$
By definition of the decomposition rank,\ Theorem \ref{coro:condition for the rank to be additive} and Corollary \ref{coro:adding 1 to the rank} imply that these groups have rank $\alpha$.\\
\indent One can also build a group with rank $\omega^\omega+1$ as follows:
$$\bigoplus_{n\in\mathbb{N}}(F_n(G),U_n(G)).$$
The rank of this group follows from Proposition \ref{prop:product same rank},\ the definition of local direct product (Definition \ref{defi:local direct product}),\ and the definition of decomposition rank (Definition \ref{defi:decomposition rank}).

\subsection{A short stroll on chief factors}

In this subsection,\ we will show the existence of chief factors of stacking type with decomposition rank $\omega^n+1$,\ for $n>0$. These results where first conjectured at \cite[Subsection 9.2]{elementary_groups_second}.\\
\indent Given $\Tilde{\delta}=\{v_n\}_{n\in\mathbb{Z}}$ the monochromatic line,\ we define the function $\eta:V\vec{T_{\chi}}\rightarrow \mathbb{Z}$ as $\eta(v_0)=0$ and for each $e=(v,w)\in E\vec{T_{\chi}}$ oriented edge,\ $\eta(w)=\eta(v)+1$. The horoballs of $\vec{T_{\chi}}$ in relation to $\Tilde{\delta}$ are defined as $\mathcal{H}_n:=\{v\in V\vec{T_{\chi}};\ \eta(v)\geqslant n\}$,\ for $n\in\mathbb{Z}$. We define $H_n:=\textup{Fix}_{P_X(G,\ U)}(\mathcal{H}_n)\leqslant P_X(G,\ \ U)$. These are normal subgroups of $P_X(G,\ U)$,\ and for $n\leqslant m$ we have $H_n\leqslant H_m$. We thus obtain normal factors $H_{n+1}/H_n$ of $P_X(G,\ U)$,\ which we call horoball factors.



\begin{lemma}\label{lemma:chief factor on EXGU}
Suppose that $(G,\ X)$ is a transitive t.d.l.c.s.c.~permutation group,\ $G$ is topologically perfect,\ and $U \in\mathcal{U}(G)$ is non-trivial. Form $E_X (G,\ U )$ and let $\Tilde{\delta}=\{v_n\}_{n\in\mathbb{Z}}$ be the bi-infinite monochromatic ray in $\vec{T_{\chi}}$ . Then,
\begin{enumerate}
    \item If $N\trianglelefteq P_X(G,\ U)$ is a proper,\ closed,\ non-trivial,\ normal subgroup,\ then there exist $n\in\mathbb{Z}$ such that $H_n\leqslant N\leqslant H_{n+1}$.   
    \item If $G$ is simple then every chief factor of $P_X(G,\ U)$ has the form $H_{n+1} /H_n$ for some $n \in\mathbb{Z}$,\ where $H_n$ is the pointwise stabilizer of $\mathcal{H}_n$ in $P_X (G,\ U )$.
    \item If $G$ is a countable discrete simple group or $\xi(G)\geqslant\omega+2$,\ then $P_X(G,\ U)$ is a stacking type chief factor (Definition \ref{defi:chief fac types}) of $E_X(G,\ U)$.
\end{enumerate}
\end{lemma}

\begin{proof}
$1.$ Assume $h\in N\backslash\{1\}$ and let $v\in V\vec{T}_\delta$ be such that $h.v\neq v$. For $n\in\mathbb{Z}$,\ $P_X(G,\ U)$ acts transitively on $\mathcal{H}_n\backslash\mathcal{H}_{n+1}$. Then for every $w\in V\vec{T}_{\delta}$ such that $\eta(v)=\eta(w)=n$ there exist $g\in P_X(G,\ U)$ such that $ghg^{-1}.w=gh.v\neq w$. Hence there exist no point fixed by $N$ in $\mathcal{H}_n\backslash\mathcal{H}_{n+1}$.\\
\indent Given $u\in V\vec{T}_{\delta}$,\ define $T^u\subset V\vec{T}_{\delta}$ to be the complement of $\cup_{n\in\mathbb{N}}\textup{inn}^n(u)$ and the group $L_u:=E_X(G,\ U)_{(T^{u})}$ (as in Proposition \ref{prop: monolith tree}).\\
\indent Let $w\in V\vec{T}_\delta$ with $\eta(w)=n$ and let $u\in inn(w)$. As shown above,\ there exist $g_w\in P_X(G,\ U)$ such that $g_w. w\neq w$. Fix such $g_w$ and let $x,\ y\in L_w$. Notice that $\textup{supp}(g_{w}xg_{w}^{-1})=g_{w}\textup{supp}(x)\subset \vec{T}_{\chi}\backslash T^{g_w.u}$, hence both $x$ and $y$ commute with $g_{w}xg_{w}^{-1}$. This implies $[[x,g_w],y]=[x,y]$. Since $N$ is normal, it is also the case that $[[x,h],y]\in N$. We then conclude that $\overline{[L_u,\ L_u]}\leqslant N$.\\
\indent As $G$ is topologically perfect then for every $n\in\mathbb{N}$ the groups $G_n$ as denoted in Lemma \ref{lemma:Gn 101} are all topologically perfect. In the view of the same lemma,\ these groups can be embedded in $\overline{[L_w,\ L_w]}$ for every $w\in V\vec{T}_{\delta}$. Hence $\overline{[L_w,\ L_w]}=L_w$. With the claim above,\ it then follows that for every $w\in  \mathcal{H}_n\backslash \mathcal{H}_{n+1}$ and $u\in inn(w)$,\ we have $L_{u}\leqslant N$. As $H_n\leqslant\overline{\langle L_{u}; \eta(w)=n, \ u\in inn(w)\rangle}$ it then follows that $H_n\leqslant  N$. We then proved that if $v\in V\vec{T}_{\chi}$ is not fixed by $N$ is such that $\eta(v)=n$, then $H_n\leqslant  N$.\\
\indent As $\overline{\bigcup_{n\in\mathbb{Z}}H_n}=P_X(G,\ U)$,\ the fact $N$ is a proper, closed subgroup there must exist $v\in V\vec{T}_{\chi}$ such that $N$ fixes $v$, otherwise $N=P_X(G,\ U)$. It is then the case that there exists a minimal $n\in\mathbb{Z}$ such that $H_n< N\leqslant H_{n+1}$. \\
\indent $2.$ Let $N$ be a proper normal subgroup of $P_X(G,\ U)$. As seen in $1.$ it follows that there exist a minimal $n\in\mathbb{Z}$ such that $H_n< N\leqslant H_{n+1}$. Notice that $N/H_n$ can be seen as a closed normal subgroup of $\bigoplus_{x\in X}(G,\ U)$. Because $G$ is simple it then follows that $N/H_n$ is normal in $\bigoplus_{x\in X}(G,\ U)$ if,\ and only if,\ there exist a subset $I\subset X$ such that $N/H_n\cong \bigoplus_{x\in X\backslash I}(G,\ U) \times \prod_{i\in I}\{1\}$. As the action of $P_X(G,\ U)$ is transitive on the horoballs it then follows that $N/H_n\cong \bigoplus_{x\in X}(G,\ U)$,\ hence $N=H_{n+1}$.\\
\indent $3.$ As shown above,\ both cases implies that $P_X(G,\ U)$ has no simple non-trivial normal subgroup. Theorem \ref{thrm:chief factors for elementary} then implies or $P_X(G,\ U)$ is of weak type with rank $\omega+1$ or is of stacking type of rank greater or equal to $\omega+1$. If $\xi(G)\geqslant \omega+2$ it then follows that $\xi(P_X(G,\ U))>\omega+2$ and hence $P_X(G,\ U)$ is of stacking type.\\
\indent Assume now $\xi(G)=2$ and $G$ is simple. As seen on $2.$,\ it follows that every non-abelian chief factor $\mathfrak{a}$ has $H_{n+1}/H_n$ as the representative of $\mathfrak{a}$. Hence,\ as the centralizer of a chief block is always a normal closed subgroup,\ $C_{P_X(G,\ U)}(\mathfrak{a})=\{g\in P_X(G,\ U); \ [g,\ H_n]\subset H_{n+1}\}=H_{n+1}$. It then follows that $H_{n+2}/H_{n+1}$ is a minimal normal factor covering $\mathfrak{a}$,\ hence $\mathfrak{B}_{P_X(G,\ U)}^{min}\neq \emptyset$,\ that is,\ $P_X(G,\ U)$ is not of weak type. As shown above,\ $P_X(G,\ U)$ cannot be of semisimple type. Hence Theorem \ref{thrm:chief factors for elementary} implies $P_X(G,\ U)$ is of stacking type.
\end{proof}

Given $n\in\mathbb{N}$,\ $n>0$,\ because the groups $F_n(G)$ built at Theorem \ref{thrm:result for higher rank by wesolek} are topologically perfect with decomposition rank greater than $\omega+2$,\ Lemma \ref{lemma:chief factor on EXGU} implies that the group $P_{X_{n}}(F_n(G),\ U_n(G))$ is a stacking type chief factor of $E_{X_n}(F_n(G),\ U_n(G))$ with decomposition rank $\omega^{n+1}+1$ (Proposition \ref{prop:collection witnessing on ex(g,u)}). If $G$ is a simple discrete group it also follows that $P_{X_0}(F_0(G),U_0(G))$ is a stacking type chief factor of $E_{X_0}(F_0(G),U_0(G))$ of rank $\omega+1$.

Let $D$ be the infinite dihedral group and $A=C_2\leqslant D$ be the orientation flip. Set $X:=D/A$. Then $(D,\ X)$ is a t.d.l.c.s.c.~permutation group. The following proposition is a construction of a chief factor of weak type and decomposition rank $\omega+1$ given on \cite[Proposition 9.8]{chief_series}. A proof was given in the same article assuming Proposition \ref{prop: monolith tree} and Corollary \ref{coro:group omega+2}.

\begin{prop}{\cite[Proposition 9.8]{chief_series}}
The monolith of $E_X(D,\ A)$ is a chief factor of weak type (Definition \ref{defi:chief fac types}) with rank $\omega+1$.
\end{prop}

\subsection{Further questions}
To construct groups with higher decomposition rank, a path that seems promising is adapting the pre-wreath structure work from \cite{first_article_about_pre_wreath} into the context of elementary groups. With general results about pre-wreath structures on the context of t.d.l.c. groups we might then be able to adapt the work from \cite{second_article_about_pre_wreath} and get all elementary groups with rank up to $\varepsilon_0:=\omega^{\omega^{\omega^{.^{.^.}}}}$. The group $E_X(G,\ U)$ already follows a similar construction to the Brin-Navas group defined on \cite{first_article_about_pre_wreath}.\\
\indent It is still not known yet if the residual height is well behaved when taking subgroups in the sense that given $H<G$ and $K\in \mathcal{K}(H)\subset \mathcal{K}(G)$, then $\textphnc{y}_G(K)=\textphnc{y}_H(K)$.
\begin{question}
Is there a t.d.l.c.s.c.~group $G$ such that there exist $K\in\mathcal{K}(G)$ and $ H\leqslant G$ a closed subgroup with $\textphnc{y}_H(K)< \textphnc{y}_G(K)$?
\end{question}

\indent It is also easy to see that,\ by Proposition \ref{prop:res do what i want in elementary},\ that the residual order on $\mathcal{K}(G)$,\ for $G$ an elementary group,\ is always a partial order. Not only that,\ any group that satisfies the equivalent properties in this proposition are such that the residual order is a partial order.

\begin{question}
Is there a t.d.l.c.s.c.~group $G$ such that the residual order on $\mathcal{K}(G)$ is not a partial order and the residual rank of $G$ is well-defined?
\end{question}

\bibliographystyle{alpha}
\bibliography{bib.bib}

\end{document}